\documentclass[a4paper,10pt,leqno]{amsart}
        \title[Vanishing of Nil-terms and negative $K$-theory]{Vanishing of Nil-terms and negative $K$-theory for additive categories}
       \author{Bartels, A.}
       \address{Westf\"alische Wilhelms-Universit\"at M\"unster\\
               Mathematicians Institut\\
               Einsteinium.~62,
               D-48149 M\"unster, Germany}
        \email{bartelsa@math.uni-muenster.de}
        \urladdr{http://www.math.uni-muenster.de/u/bartelsa}
        \author{L\"uck, W.}
        \address{Mathematicians Institut der Universit\"at Bonn\\
                Endenicher Allee 60\\
                53115 Bonn, Germany}
         \email{wolfgang.lueck@him.uni-bonn.de}
          \urladdr{http://www.him.uni-bonn.de/lueck}
         \date{August 2022}
    \makeatletter
     \@namedef{subjclassname@2020}{%
  \textup{2020} Mathematics Subject Classification}
\makeatother

\usepackage[utf8]{inputenc}
  \usepackage{hyperref}
  \usepackage{comment}
  \usepackage{calc}
  \usepackage{enumerate,amssymb,bm}
  \usepackage[arrow,curve,matrix,tips,2cell]{xy}
    \SelectTips{eu}{10} \UseTips
    \UseAllTwocells
  \usepackage{tikz}
  \usetikzlibrary{decorations.pathreplacing}
  \usepackage{color}
  \usepackage{scalerel}
  \usepackage{braket}
  \usepackage{mathtools}
  \usepackage{bbm}
  \usepackage{enumitem}
  \usepackage{float}
  \DeclareMathAlphabet{\matheurm}{U}{eur}{m}{n}

\DeclareMathAlphabet{\matheurm}{U}{eur}{m}{n}
\newcommand{\addcat}{\matheurm{Add\text{-}Cat}}
\newcommand{\addcatic}{\matheurm{Add\text{-}Cat}_{ic}}

\newcommand{\Chcat}{\matheurm{Ch}}
\newcommand{\Chcathf}{\matheurm{Ch}^{\operatorname{hf}}}

\newcommand{\MODcatl}[1]{#1\text{-}\matheurm{MOD}}
\newcommand{\MODcatr}[1]{\matheurm{MOD}\text{-}#1}

\newcommand{\OrG}[1]{\matheurm{Or}(#1)}

\newcommand{\Spectra}{\matheurm{Spectra}}

\newcommand{\Waldho}{\matheurm{Wald}^{\operatorname{ho}}}

\DeclareMathOperator{\aut}{aut}

\DeclareMathOperator{\cok}{cok}
\DeclareMathOperator{\colim}{colim}

\DeclareMathOperator{\cone}{cone}

\DeclareMathOperator{\cyl}{cyl}

\DeclareMathOperator{\F}{F}

\DeclareMathOperator{\FL}{FL}
\DeclareMathOperator{\FP}{FP}
\DeclareMathOperator{\fgf}{fgf}
\DeclareMathOperator{\fgp}{fgp}

\DeclareMathOperator{\HNil}{HNil}
\DeclareMathOperator{\hocofib}{hocofib}
\DeclareMathOperator{\hocolim}{hocolim}

\DeclareMathOperator{\id}{id}
\DeclareMathOperator{\im}{im}

\DeclareMathOperator{\Idem}{Idem}

\DeclareMathOperator{\mor}{mor}
\DeclareMathOperator{\Nil}{Nil}

\DeclareMathOperator{\ob}{ob}
\DeclareMathOperator{\pdim}{pdim}

\DeclareMathOperator{\proaut}{pro-aut}
\DeclareMathOperator{\pt}{pt}
\DeclareMathOperator{\pr}{pr}
\DeclareMathOperator{\res}{res}

\newcommand{\VCYC}{{{\mathcal V}\mathrm{cyc}}}

  \newcommand{\IN}{\mathbb{N}}

  \newcommand{\IZ}{\mathbb{Z}}

  \newcommand{\cala}{\mathcal{A}}
  \newcommand{\calb}{\mathcal{B}}
  \newcommand{\calc}{\mathcal{C}}

  \newcommand{\cali}{\mathcal{I}}

    \newcommand{\call}{\mathcal{L}}

  \newcommand{\cals}{\mathcal{S}}
  
  \newcommand{\calt}{\mathcal{T}}
    \newcommand{\calu}{\mathcal{U}}

  \newcommand{\calw}{\mathcal{W}}
  \newcommand{\calx}{\mathcal{X}}
  \newcommand{\caly}{\mathcal{Y}}

  \newcommand{\bfa}{\mathbf{a}}
  
  \newcommand{\bfb}{\mathbf{b}}

  \newcommand{\bfd}{\mathbf{d}}
  
  \newcommand{\bfE}{\mathbf{E}}

  \newcommand{\bff}{\mathbf{f}}

  \newcommand{\bfH}{\mathbf{H}}
  \newcommand{\bfh}{\mathbf{h}}
  \newcommand{\bfHP}{\mathbf{HP}}
  
  \newcommand{\bfi}{\mathbf{i}}

  \newcommand{\bfK}{\mathbf{K}}
  
  \newcommand{\bfKinftyW}{{\mathbf K}^{\infty,\operatorname{W}}}

  \newcommand{\bfl}{\mathbf{l}}

  \newcommand{\bfN}{\mathbf{N}}

  \newcommand{\bfr}{\mathbf{r}}
  
  \newcommand{\bfs}{\mathbf{s}}
  
  \newcommand{\bfT}{\mathbf{T}}
  \newcommand{\bft}{\mathbf{t}}

 \newcommand{\bfEinfty}{\mathbf{E}^{\infty}}
 \newcommand{\bfKinfty}{\mathbf{K}^{\infty}}

 \newcommand{\bfKNilinfty}{\mathbf{K}_{\Nil}^{\infty}}
 \newcommand{\bfNK}{\mathbf{N}\mathbf{K}}
 \newcommand{\bfNKinfty}{\mathbf{N}\mathbf{K}^{\infty}}

\newcommand{\EGF}[2]{E_{#2}(#1)}

\newcommand{\squarematrix}[4]
{\left( \begin{array}{cc} #1 & #2 \\ #3 &
#4
\end{array} \right)
}

\newcommand{\trun}{[\,]}

\newcounter{commentcounter}

\theoremstyle{plain}
\newtheorem{theorem}{Theorem}[section]

\newtheorem{lemma}[theorem]{Lemma}
\newtheorem{corollary}[theorem]{Corollary}

\newtheorem*{theorem*}{Theorem}
\newtheorem*{theoremA*}{Theorem A}
\newtheorem*{theoremB*}{Theorem B}

\theoremstyle{definition}
\newtheorem{definition}[theorem]{Definition}

\newtheorem{example}[theorem]{Example}

\newtheorem{remark}[theorem]{Remark}
\newtheorem{notation}[theorem]{Notation}

\newtheorem*{definition*}{Definition}

\theoremstyle{remark}

\makeatletter\let\c@equation=\c@theorem\makeatother

\theoremstyle{definition}

\newcounter{othercommentcounter}

\newcommand{\version}[1]              %marks the date of last editing and compilation
{\begin{center} last edited on #1\\
last compiled on \today\\
name of tex-file: \jobname
\end{center}}

\begin{document}

\begin{abstract}
  We extend the notion of regular coherence from rings to additive categories and show
  that well-known consequences of regular coherence for rings also apply to additive
  categories.  For instance the negative $K$-groups and all twisted Nil-groups vanish for
  an additive category, if it is regular coherent. This will be applied to nested
  sequences of additive categories, motivated by our ongoing project to determine the
  algebraic $K$-theory of the Hecke algebra of a reductive $p$-adic group.
\end{abstract}

\maketitle

\newlength{\origlabelwidth}\setlength\origlabelwidth\labelwidth

%%%%%%%%%%%%%%%%%%%%%%%%%%%%%%%%%%%%%%%%%%%%%%%%%%%%%%%%%%%%%%%%%%%%
%%%%%%%%%%%%%%%%%%%%%%%%%% Introduction %%%%%%%%%%%%%%%%%%%%%%%%%%%%%%%%
%%%%%%%%%%%%%%%%%%%%%%%%%%%%%%%%%%%%%%%%%%%%%%%%%%%%%%%%%%%%%%%%%%%%

\typeout{------------------- Introduction -----------------}
\section{Introduction}%
\label{sec:introduction}

%%%%%%%%%%%%%%%%%%%%%%%%%%%%%%%%%%%%%%%%%%%%%%%%%%%%%%%%%%%%%%%%%%%%

\subsection*{Background}

The Bass-Heller-Swan Theorem gives  isomorphisms 
\begin{equation*}
K_n R[t,t^{-1}] \cong K_{n-1}(R) \oplus K_n(R) \oplus \widetilde{\Nil}_{n-1}(R) \oplus \widetilde{\Nil}_{n-1}(R)
\end{equation*}
for K-theory of rings.  For regular rings all Nil groups
$\widetilde{\Nil}_{n}(R)$, $n \in \IZ$ and negative K-groups
$K_{n}(R)$, $n \in \IZ_{< 0}$ vanish and this simplifies the
Bass-Heller-Swan formula.
Waldhausen~\cite{Waldhausen(1978genfreeI+II),Waldhausen(1978genfreeIII+IV)}
proved far reaching extensions of the Bass-Heller-Swan formula for
other group rings.  He also introduced regular coherence for rings and
proved generalizations of the above vanishing results for regular
coherent rings.
Waldhausen's motivation was that some group rings are regular coherent
(but not regular) and this allowed him to bootstrap K-theory
computations for group rings.  The Bass-Heller-Swan Theorem is also an
important ingredient in K-theory computations via the Farrell-Jones
conjecture.  If $R$ is regular, then so is $R[t,t^{-1}]$, but we do
not know, whether the same inheritance statement holds for regular
coherence.  This is one reason, why we will not only concentrate on
regular coherence here, but also on regularity.

The goal of this paper is to extend the notions of regularity and
regular coherence from rings to additive categories and to extend the
vanishing results in $K$-theory to additive categories.  The basic
strategy will be to embed a given additive category $\cala$ in the
category of ${\IZ\cala}$-modules.  The latter category is abelian.
This mimics the additive subcategory of finitely generated free
$R$-modules of the abelian category of all $R$-modules and allows the
extension of arguments and definitions from rings to additive
categories.  This is a standard construction and has been used for a
long time, for example to define Noetherian additive categories and
global dimension for additive categories.

We also extract intrinsic characterizations on the level of additive categories.  For
instance, we call a sequence $A_0 \xrightarrow{f_0} A_1 \xrightarrow{f_1} A_2$ in $\cala$
\emph{exact at $A_1$}, if $f_1 \circ f_0 = 0$ and for every object $A$ and morphism
$g \colon A \to A_1$ with $f_1 \circ g = 0$ there exists a morphism
$\overline{g} \colon A \to A_0$ with $f_0 \circ \overline{g} = g$, see
Definition~\ref{def:exact_sequence_in_cala}.  We show in
Lemma~\ref{lem:intrinsic_reformulation_of_regular_coherent}~%
\ref{lem:intrinsic_reformulation_of_regular_coherent:not_uniform} that an idempotent
complete additive category $\cala$ is regular coherent, if and only if for every morphism
$f_1\colon A_1 \to A_0$ we can find a sequence of finite length in $\cala$
\[0 \to A_n \xrightarrow{f_n} A_{n-1} \xrightarrow{f_{n-1}} \cdots \xrightarrow{f_2} A_{1}
  \xrightarrow{f_1} A_0,
\]
which is exact at $A_i$ for $i = 1,2, \ldots, n$. It is \emph{$l$-uniformly regular coherent}
if the number $n$ can be arranged to satisfy $n \le l$ for every morphisms $f_1$.

%%%%%%%%%%%%%%%%%%%%%%%%%%%%%%%%%%%%%%%%%%%%%%%%%%%%%%%%%%%%%%%%%%%%

\subsection*{Our motivation}
In our experience it is often more convenient to work with additive categories in place of
rings in connection with K-theory.  Sometimes a minor drawback is that results for
K-theory of rings have not been fully developed for additive K-theory, although often they are
really no more complicated.  This paper takes care of the extension of regular coherence
from rings to additive categories that we expect to be helpful.

More concretely, we rely on the present paper in our ongoing work aimed at the computation
of the K-theory of Hecke algebras of reductive $p$-adic groups.  There we apply regular
coherence and the K-theory vanishing results to certain additive categories. Namely, we consider
a decreasing nested sequence of additive 
subcategories $\cala_* = \big(\cala_0 \supseteq \cala_1 \supseteq \cala_2 \supseteq \cdots\big)$, see
Definition~\ref{def:inverse_system_of_additive_categories}, and associate to it the
additive category $\cals(\cala_*)$, called \emph{sequence category}, and a certain
quotient additive category $\call(\cala_*)$, called \emph{limit category}, see
Definition~\ref{def:calo_upper_0(cals_ast)}. An object in $\cals(\cala_*)$ is a sequence
$\underline{A} = (A_m)_{m \ge 0}$ of objects in $\cala_0$ such that for every $l \in \IN$
almost all $\phi_m$ lie in $\cala_l$. A morphism
$\underline{\phi} \colon \underline{A} \to \underline{A'}$ in $\cals(\cala_*)$ consists of
a sequence of morphisms $\phi_m \colon A_m \to A'_m$ in $\cala_0$ such that for every
$l \in \IN$ almost all $\phi_m$ lie in $\cala_l$.\footnote{These categories come in our
  situation from controlled algebra; typically control conditions get more restrictive
  with $m\to \infty$.}. If the system $\cala_*$ is constant, i.e., $\cala_m = \cala_0$,
then $\cals(\cala_*)= \prod_{m \in \IN} \cala_0$ and $\call(\cala_*)$ is the quotient
additive category
$\prod_{m \in \IN} \cala_0 \bigl/ \bigoplus_{m \in \IN}\cala_0$.

Typically each $\cala_m$ will be regular, but this does not imply that $\cals(\cala_*)$ or
$\call(\cala_*)$ is regular as well, see
Remark~\ref{lem:Advantage_of_the_notion_l-uniformly_regular_coherent}.  The problem is
that the property Noetherian does not pass to infinite products of additive categories,
see
Example~\ref{exa:The_property_Noetherian_does_not_pass_to_the_sequence_category}. Therefore
we have to discard the condition Noetherian in our considerations.

%%%%%%%%%%%%%%%%%%%%%%%%%%%%%%%%%%%%%%%%%%%%%%%%%%%%%%%%%%%%%%%%%%%%

\subsection*{Main results}

As mentioned above we discuss various regularity properties, which are well-known for rings
and extend them to additive categories. As long as we are concerned with the notion
regular or Noetherian, we follow the standard proof for rings, which carry over to additive
categories.  This is done for the convenience of the reader. 

As described above, we need to discard the property Noetherian and stick to regular
coherence and the new notion of uniform regular coherence.  These notions do pass to
infinite products of additive categories, see Lemma~\ref{lem:regular_coherent_and_products},
and more generally
under a certain exactness condition about $\cala_*$ to   the additive categories
$\cals(\cala_*)$ and  $\call(\cals_*)$, see Lemma~\ref{lem:inverse_systems_and_regularity}.
We remark that algebraic $K$-theory does commute
with infinite products for additive categories, see~\cite{Carlsson(1995)} and
also~\cite[Theorem~1.2]{Kasprowski-Winges(2020)},  but not with infinite products of rings. 

We will show the vanishing of twisted Nil-terms and of the negative $K$-theory for regular
coherent additive categories in Sections~\ref{sec:Vanishing_of_Nil-terms} and
Section~\ref{sec:Vanishing_of_negative_K-groups}.

The for us most valuable result is the technical
Theorem~\ref{the:assembly_for_cald_upper_0(cals_ast)_with_Zr-action},
whose proof relies on the vanishing of twisted Nil-terms. It will be a key
ingredient in our project to extend the $K$-theoretic Farrell-Jones Conjecture for discrete
groups to reductive $p$-adic groups, notably, when we want to reduce  the family of
subgroups, which map with a compact kernel to $\IZ$, to the family of compact open
subgroups.
For discrete groups there is a well-known similar reduction relying also on regularity conditions.
However, in the discrete case it typically suffices to use regularity for rings, while
our approach to $K$-theory of reductive $p$-adic groups necessitates the use of
the weaker regularity conditions introduced in the present paper.

%%%%%%%%%%%%%%%%%%%%%%%%%%%%%%%%%%%%%%%%%%%%%%%%%%%%%%%%%%%%%%%%%%%%

\subsection*{Acknowledgements}

This paper is funded by the ERC Advanced Grant ``KL2MG-interactions''
(no.  662400) of the second author granted by the European Research
Council, by the Deutsche Forschungsgemeinschaft (DFG, German Research
Foundation) \-– Project-ID 427320536 \-- SFB 1442, as well as under
Germany's Excellence Strategy \-- GZ 2047/1, Projekt-ID 390685813,
Hausdorff Center for Mathematics at Bonn, and EXC 2044 \-- 390685587,
Mathematics M\"unster: Dynamics \-- Geometry \-- Structure.
\medskip

\noindent
The paper is organized as follows:

\tableofcontents

%%%%%%%%%%%%%%%%%%%%%%%%%%%%%%%%%%%%%%%%%%%%%%%%%%%%%%%%%%%%%%%%%%%%%
%%%%%%%%%%%%%%%%%%%%%%%%%%%%%% Section 1 %%%%%%%%%%%%%%%%%%%%%%%%%%%%%%%%
%%%%%%%%%%%%%%%%%%%%%%%%%%%%%%%%%%%%%%%%%%%%%%%%%%%%%%%%%%%%%%%%%%%%%

\typeout{----- Section: $\IZ$-categories, additive categories and idempotent completions  -------}

\section{$\IZ$-categories, additive categories, and idempotent completions}%
\label{sec:IZ-categories_additive_categories_and_idempotent_completions}

%%%%%%%%%%%%%%%%%%%%%%%%%%%%%%%%%%%%%%%%%%%%%%%%%%%%%%%%%%%%%%%%%%%%%

\subsection{$\IZ$-categories}%
\label{subsec:Z-categories}

A \emph{$\IZ$-category} is a category $\cala$ such that for every two objects $A$ and $A'$
in $\cala$ the set of morphisms $\mor_{\cala}(A,A')$ has the structure of a $\IZ$-module,
for which composition is a $\IZ$-bilinear map.  Given a ring $R$, we denote by
$\underline{R}$ the $\IZ$-category with precisely one object, whose $\IZ$-module of
endomorphisms is given by $R$ with its $\IZ$-module structure and composition is given by
the multiplication in $R$.

%%%%%%%%%%%%%%%%%%%%%%%%%%%%%%%%%%%%%%%%%%%%%%%%%%%%%%%%%%%%%%%%%%%%%

\subsection{Additive categories}%
\label{subsec:additive_categories}
 
An \emph{additive category} is a $\IZ$-category such that for any two objects $A_1$ and
$A_2$ there is a model for their direct sum, i.e., an object $A$ together with morphisms
$i_k \colon A_k \to A$ for $k = 1,2$ such that for every object $B$ in $\cala$ the $\IZ$-map
\[
\mor_{\cala}(A,B) \xrightarrow{\cong} \mor_{\cala}(A_1,B) \times \mor_{\cala}(A_2,B),
\quad f \mapsto (f \circ i_0, f\circ i_1)
\]
is bijective.

Given a ring $R$, the category $\MODcatl{R}_{\fgf}$
of finitely generated free left $R$-modules carries an obvious structure of an additive
category.

An \emph{equivalence} $F \colon \cala \to \cala'$ of  $\IZ$-categories or of
 additive categories respectively is a functor of  $\IZ$-categories or of
 additive categories respectively such that for all objects $A_1, A_2$ in
$\cala$ the induced map
$F_{A_1,A_2} \colon \mor_{\cala}(A_1,A_2) \xrightarrow{\cong} \mor_{\cala}(F(A_1),F(A_2))$
sending $f$ to $F(f)$ is bijective, and for any object $A'$ in $\cala'$ there exists an
object $A$ in $\cala$ such that $F(A)$ and $A'$ are isomorphic in $\cala'$. This is
equivalent to the existence of a functor $F \colon \cala' \to \cala$ of 
$\IZ$-categories or of  additive categories respectively such that both
composites $F \circ F'$ and $F' \circ F$ are naturally equivalent as such functors to the
identity functors.

Given a $\IZ$-category, let $\cala_{\oplus}$ be the associated additive
category, whose objects are finite tuples of objects in $\cala$ and whose
morphisms are given by matrices of morphisms in $\cala$ (of the right size) and the direct
sum is given by concatenation of tuples and the block sum of matrices, see for
instance~\cite[Section~1.3]{Lueck-Steimle(2016BHS)}. 

Let $R$ be a  ring.
Then we can consider the additive category $\underline{R}_{\oplus}$.
The obvious inclusion of  additive categories 
\begin{equation}
\theta_{\fgf} \colon \underline{R}_{\oplus} \xrightarrow{\simeq} \MODcatl{R}_{\fgf}
\label{underline(r)_oplus_to_MODcat(R)_fgf}
\end{equation}
is an equivalence of additive categories. Note that
$\underline{R}_{\oplus}$ is small, in contrast to $\MODcatl{R}_{\fgf}$.

%%%%%%%%%%%%%%%%%%%%%%%%%%%%%%%%%%%%%%%%%%%%%%%%%%%%%%%%%%%%%%%%%%%%%

\subsection{Idempotent completion}%
\label{subsec:idempotent_completion}

Given an additive category $\cala$, its \emph{idempotent completion}
$\Idem(\cala)$ is defined to be the following additive category. Objects are
morphisms $p \colon A \to A$ in $\cala$ satisfying $p \circ p = p$.  A morphism $f$ from
$p_1 \colon A_1 \to A_1$ to $p_2 \colon A_2 \to A_2$ is a morphism $f \colon A_1 \to A_2$
in $\cala$ satisfying $p_2 \circ f \circ p_1 = f$.  The structure of an additive
category on $\cala$ induces the structure of an additive category on $\Idem(\cala)$
in the obvious way.  The
identity of an object $(A,p)$ is given by the morphism $p \colon (A,p) \to (A,p)$.
A functor of additive categories $F \colon \cala\to \cala'$
induces a functor $\Idem(F) \colon \Idem(\cala) \to \Idem(\cala')$ of additive
categories by sending $(A,p)$ to $(F(A),F(p))$.

There is a obvious embedding
\[
\eta(\cala)\colon\cala \to \Idem(\cala)
\]
sending an object $A$ to
$\id_A \colon A \to A$ and a morphism $f \colon A \to B$ to the morphism given by $f$ again.  
An  additive category $\cala$ is called
\emph{idempotent complete}, if $\eta(\cala)\colon \cala \to \Idem(\cala)$ is an
equivalence of  additive categories, or, equivalently, 
if for every idempotent $p \colon A \to A$ in $A$
there exists objects $B$ and $C$ and an isomorphism 
$f \colon A \xrightarrow{\cong} B \oplus C$ in $\cala$
such that $f \circ p \circ f^{-1} \colon B \oplus C \to B \oplus C$ is given by 
$\squarematrix{\id_B}{0}{0}{0}$.  The idempotent completion $\Idem(\cala)$ of an 
additive  category $\cala$ is idempotent complete.

For a ring $R$, let $\MODcatl{R}_{\fgp}$ be
the  additive category of finitely generated projective $R$-modules. We
get an equivalence of  additive categories
$\Idem(\MODcatl{R}_{\fgf}) \xrightarrow{\simeq} \MODcatl{R}_{\fgp}$ by sending an object
$(F,p)$ to $\im(p)$. It and the functor of~\eqref{underline(r)_oplus_to_MODcat(R)_fgf}
induce an equivalence of  additive categories
\begin{equation}
\theta_{\fgp}  \colon \Idem\bigl(\underline{R}_{\oplus}\bigr) \xrightarrow{\simeq}  \MODcatl{R}_{\fgp}.
\label{underline(r)_oplus_to_MODcat(R)_fgp}
\end{equation}
Note that $\Idem\bigl(\underline{R}_{\oplus}\bigr)$ is small, in contrast to
$\MODcatl{R}_{\fgp}$.

%%%%%%%%%%%%%%%%%%%%%%%%%%%%%%%%%%%%%%%%%%%%%%%%%%%%%%%%%%%%%%%%%%%%%

\subsection{Twisted finite Laurent category}%
\label{subsec:Twisted_finite_Laurent_category}

Let $\cala$ be an additive category. Let $\Phi \colon \cala \to \cala$ be an
automorphism of additive categories.

\begin{definition}[{Twisted finite Laurent category $\cala_{\Phi}[t,t^{-1}]$}]\label{def:A_phi[t,t(-1)]}
  Define the \emph{$\Phi$-twisted finite Laurent category}  $\cala_{\Phi}[t,t^{-1}]$ as follows. It has the
  same objects as $\cala$.  Given two objects $A$ and $B$, a morphism 
  $f  \colon A \to B$ in $\cala_{\Phi}[t,t^{-1}]$ is a formal sum 
  $f = \sum_{i \in    \IZ} f_i \cdot t^i$, where $f_i \colon \Phi^i(A) \to B$ is a morphism in
  $\cala$ from $\Phi^i(A)$ to $B$ and only finitely many of the morphisms $f_i$ are non-trivial. If 
  $g   = \sum_{j \in \IZ} g_j \cdot t^j$ is a morphism in $\cala_{\Phi}[t,t^{-1}]$
  from $B$ to $C$, we define the composite $g \circ f \colon A \to C$ by
  \[
  g \circ f := \sum_{k \in \IZ} \biggl( \sum_{\substack{i,j \in \IZ,\\i+j = k}}
    g_j \circ \Phi^j(f_i)\biggr) \cdot t^k.
  \]  
  The  direct sum and the structure of a $\IZ$-module on the
  set of morphism from $A$ to $B$ in $\cala_{\Phi}[t,t^{-1}]$ are
  defined in the obvious way using the corresponding structures of
  $\cala$. We sometimes also  write $\cala_{\Phi}[\IZ]$ instead of $\cala_{\Phi}[t,t^{-1}]$.
\end{definition}

\begin{example}\label{exa:Finitely_generated_free_R-modules}
  Let $R$ be a  ring with an automorphism
  $\phi \colon R \xrightarrow{\cong} R$ of rings. Let $R_{\phi}[t,t^{-1}]$ be
  the ring of $\phi$-twisted finite Laurent series with coefficients in $R$.  We obtain
  from $\phi$ an automorphism
  $\Phi \colon \underline{R} \xrightarrow{\cong} \underline{R}$ of $\IZ$-categories.
  There is an obvious isomorphism of $\IZ$-categories
  \begin{equation}
    \underline{R}_{\Phi}[t,t^{-1}] \xrightarrow{\cong} \underline{R_{\phi}[t,t^{-1}]}.
    \label{underline_and_Laurent}
  \end{equation}
  We obtain equivalences of  additive categories
  \begin{eqnarray*}
    (\underline{R}_{\oplus})_{\Phi}[t,t^{-1}] 
    &\xrightarrow{\simeq}  &
                             \MODcatl{R_{\phi}[t,t^{-1}]}_{\fgf};
    \\
    \Idem\bigl((\underline{R}_{\oplus})_{\Phi}[t,t^{-1}]\bigr)
    &\xrightarrow{\simeq}  &
                             \MODcatl{R_{\phi}[t,t^{-1}]}_{\fgp}.
  \end{eqnarray*}
\end{example}

\begin{definition}[{$\cala_{\Phi}[t]$} and {$\cala_{\Phi}[t^{-1}]$}]\label{def:A_phi(t_and_t(-1))} 
  Let $\cala_{\Phi}[t]$ and $\cala_{\Phi}[t^{-1}]$ respectively be the additive
  subcategory of $\cala_{\Phi}[t,t^{-1}]$, whose set of objects is the set of objects in
  $\cala$ and whose morphism from $A$ to $B$ are given by finite formal Laurent series
  $\sum_{i \in \IZ} f_i \cdot t^i$ with $f_i = 0$ for $i < 0$ and $i > 0$ respectively.
\end{definition}

%%%%%%%%%%%%%%%%%%%%%%%%%%%%%%%%%%%%%%%%%%%%%%%%%%%%%%%%%%%%%%%%%%%%%
%%%%%%%%%%%%%%%%%%%%%%%%%%%%%% Section 2 %%%%%%%%%%%%%%%%%%%%%%%%%%%%%%%%
%%%%%%%%%%%%%%%%%%%%%%%%%%%%%%%%%%%%%%%%%%%%%%%%%%%%%%%%%%%%%%%%%%%%%

\typeout{------ Section: The algebraic $K$-theory of additive categories -------------------}

%%%%%%%%%%%%%%%%%%%%%%%%%%%%%%%%%%%%%%%%%%%%%%%%%%%%%%%%%%%%%%%%%%%%%
\section{The algebraic $K$-theory of $\IZ$-categories}%
\label{sec:The_algebraic_K-theory_of_Z-categories}

Let $\cala$ be an  additive category. One can interprete it as an exact category in
the sense of Quillen or as a category with cofibrations and weak equivalence in the sense
of Waldhausen and obtains the \emph{connective algebraic $K$-theory spectrum }
$\bfK(\cala)$ by the constructions due to Quillen~\cite{Quillen(1973)} or
Waldhausen~\cite{Waldhausen(1985)}.  A construction of the \emph{non-connective $K$-theory
spectrum} $\bfKinfty(\cala)$ of an  additive category can be found for instance
in~\cite{Lueck-Steimle(2014delooping)} or~\cite{Pedersen-Weibel(1985)}.

  \begin{definition}[Algebraic $K$-theory of  $\IZ$-categories]%
\label{def:Algebraic_K-theory_of_Z-categories}
    We will define the \emph{algebraic $K$-theory spectrum $\bfKinfty(\cala)$} of the 
   $\IZ$-category $\cala$ to be  the non-connective algebraic $K$-theory spectrum of the 
     additive  category $\cala_{\oplus}$. Define for $n \in \IZ$
   \[
   K_n(\cala) := \pi_n(\bfKinfty(\cala)).
   \]
   The connective algebraic $K$-theory spectrum  $\bfK(\cala)$ is defined to be the  connective 
  algebraic $K$-theory spectrum of the additive  category $\cala_{\oplus}$.
  \end{definition}

  If $\cala$ is an additive category and $i(\cala)$ is the underlying $\IZ$-category,
  then there is a canonical  equivalence of additive categories $i(\cala)_{\oplus} \to \cala$. Hence 
  there are canonical weak homotopy equivalences $\bfK(i(\cala)) \to \bfK(\cala)$
  and $\bfKinfty(i(\cala)) \to \bfKinfty(\cala)$.

  A functor $F \colon \cala \to \cala'$ of $\IZ$-categories
  induces a map of spectra
   \begin{equation}
   \bfKinfty(F) \colon \bfKinfty(\cala) \to \bfKinfty(\cala').
   \label{bfK(F)_colon_bfK(cala)_to_bfK(cala')}
 \end{equation}

 We call a full additive subcategory $\cala$ of $\cala'$ \emph{cofinal}, if for any
 object $A'$ in $\cala'$ there is an object $A$ in $\cala$ together with morphisms
 $i \colon A' \to A$ and $r \colon A' \to A$ satisfying $r \circ i = \id$.

   \begin{lemma}\label{K_and_cofinal}
     Let $I \colon \cala \to \cala'$ be the inclusion of a full cofinal additive subcategory.
     \begin{enumerate}
     \item\label{K_and_cofinal:connective} The induced map
       \[
       \pi_n(\bfK(I)) \colon \pi_n(\bfK(\cala)) \to \pi_n(\bfK(\cala'))
     \]
     is bijective for $n \ge 1$;
   \item\label{K_and_cofinal:non-connective}
     The induced map
     \[
      \bfKinfty(I) \colon \bfKinfty(\cala) \to \bfKinfty(\cala')
     \]
     is a weak homotopy equivalence.
   \end{enumerate}
 \end{lemma}
 \begin{proof}~\ref{K_and_cofinal:connective} This is proved for $\cala' = \Idem(\cala)$
   in~\cite[Theorem~A.9.1.]{Thomason-Trobaugh(1990)}.  Now the general case follows from
   the observation that $\Idem(\cala) \to \Idem(\cala')$ is an equivalence of additive
   categories.  \\[1mm]~\ref{K_and_cofinal:non-connective} This follows from
   assertion~\ref{K_and_cofinal:connective}
   and~\cite[Corollary~3.7]{Lueck-Steimle(2014delooping)}.
 \end{proof}

%%%%%%%%%%%%%%%%%%%%%%%%%%%%%%%%%%%%%%%%%%%%%%%%%%%%%%%%%%%%%%%%%%%%
%%%%%%%%%%%%%%%%%%%%%%%%%%%%%%% Section 3 %%%%%%%%%%%%%%%%%%%%%%%%%%%%%%
%%%%%%%%%%%%%%%%%%%%%%%%%%%%%%%%%%%%%%%%%%%%%%%%%%%%%%%%%%%%%%%%%%%%

\typeout{---- Section:  The Bass-Heller-Swan decomposition for  additive categories ----}

\section{The Bass-Heller-Swan decomposition for additive categories}%
\label{sec:The_Bass-Heller-Swan_decomposition_for_additive_categories}

Denote by $\addcat$ the category of additive categories. Let us consider the group $\IZ$
as a groupoid with one object and denote by $\addcat^\IZ$ the category of functors
$\IZ\to\addcat$, with natural transformations as morphisms. Note that an object of this
category is a pair $(\cala,\Phi)$ consisting of an additive category together with an
automorphism $\Phi \colon \cala \xrightarrow{\cong} \cala$ of additive categories.  We
recall from~\cite[Theorem~0.1]{Lueck-Steimle(2016BHS)} using the notation
of this paper here and in the sequel:

\begin{theorem}[The Bass-Heller-Swan decomposition for non-connective $K$-theory of
  additive categories]\label{the:BHS_decomposition_for_non-connective_K-theory}
  Let $\Phi \colon \cala \to \cala$ be an  automorphism of additive categories.

  \begin{enumerate}
  \item\label{the:BHS_decomposition_for_non-connective_K-theory:BHS-iso}

  There exists a weak homotopy equivalence of
  spectra, natural in $(\cala,\Phi)$, 
  \[
  \bfa^{\infty} \vee \bfb_+^{\infty}\vee \bfb_-^{\infty} \colon \bfT_{\bfKinfty(\Phi^{-1})} \vee
  \bfNKinfty(\cala_{\Phi}[t]) \vee \bfNKinfty(\cala_{\Phi}[t^{-1}]) \xrightarrow{\simeq}
  \bfKinfty(\cala_{\Phi}[t,t^{-1}])
\]
where $\bfT_{\bfKinfty(\Phi^{-1})}$ is the mapping torus of
$\bfKinfty(\Phi^{-1}) \colon \bfKinfty(\cala) \to \bfKinfty(\cala)$ and
$\bfNKinfty(\cala_{\Phi}[t^{\pm}])$ is the homotopy fiber of the map
$\bfKinfty(\cala_{\Phi}[t^{-1}]) \to \bfKinfty(\cala)$ given by evaluation $t = 0$;
  
\item\label{the:BHS_decomposition_for_non-connective_K-theory:Nil} There exist a functor
  $\bfEinfty\colon \addcat^\IZ\to\Spectra$ and weak homotopy equivalences of spectra,
  natural in $(\cala,\Phi)$,
  \begin{eqnarray*}
   \Omega \bfNKinfty(\cala_{\Phi}[t]) & \xleftarrow{\simeq} & \bfE^{\infty}(\cala,\Phi);
   \\
   \bfKinfty(\cala) \vee  \bfE^{\infty} (\cala,\Phi)  & \xrightarrow{\simeq} & \bfKNilinfty(\cala,\Phi),
  \end{eqnarray*}
  where $\bfKNilinfty(\cala,\Phi)$ is the non-connective $K$-theory of a certain
  Nil-category $\Nil(\cala,\Phi)$. 
  \end{enumerate}
\end{theorem}

\begin{theorem}[Fundamental sequence of $K$-groups]\label{the:Fundamental_sequence_of_K-groups}
  Let $\cala$ be an additive category. Then there exists for $n \in \IZ$ a split exact sequence, natural in
  $\cala$
  \begin{multline}
    0  \to K_n(\cala) \xrightarrow{(k_+)_* \oplus -(k_-)_*} K_n(\cala[t]) \oplus  K_n(\cala[t^{-1}])
    \xrightarrow{(l_+)_* \oplus (l_-)_*}
    \\
    K_n(\cala[t,t^{-1}]) \xrightarrow{\delta_n} K_{n-1}(\cala) \to 0,
    \end{multline}
    where $(k_+)_*$, $(k_-)_*$, $(l_+)_*$, and $(l_-)_*$ are induced by the obvious
    inclusions $k_+$, $k_-$, $l_+$, and $l_-$ and $\delta_n$ is the composite of the
    inverse of the (untwisted) Bass-Heller-Swan isomorphism
    \[K_n(\cala) \oplus K_{n-1}(\cala) \oplus N\!K_n(\cala[t]) \oplus N\!K_n(\cala[t^{-1}])
      \xrightarrow{\cong} K_n(\cala[t,t^{-1}]),
    \]
    see Theorem~\ref{the:BHS_decomposition_for_non-connective_K-theory}, with the
    projection onto the summand $K_{n-1}(\cala)$.
  \end{theorem}
  \begin{proof} This follows directly from the untwisted version of
    Theorem~\ref{the:BHS_decomposition_for_non-connective_K-theory}.
  \end{proof}

There is also a version for the \emph{connective $K$-theory spectrum} $\bfK$. Denote by
 $\addcatic\subset\addcat$ the full subcategory of idempotent complete additive categories.

 \begin{theorem}[The Bass-Heller-Swan decomposition for connective $K$-theory of additive
   categories]\label{the:BHS_decomposition_for_connective_K-theory}
   Let $\cala$ be an additive category, which is idempotent complete. Let
   $\Phi \colon \cala \to \cala$ be an automorphism of additive categories.

  \begin{enumerate}
  \item\label{the:BHS_decomposition_for_connective_K-theory:BHS-iso}

  Then there is a weak equivalence of  spectra, natural in $(\cala,\Phi)$,
  \[
  \bfa \vee \bfb_+\vee \bfb_- \colon \bfT_{\bfK(\Phi^{-1})} \vee
  \bfNK(\cala_{\Phi}[t]) \vee \bfN\bfK(\cala_{\Phi}[t^{-1}]) \xrightarrow{\simeq}
  \bfK(\cala_{\Phi}[t,t^{-1}])
  \]
  where $\bfT_{\bfK(\Phi^{-1})}$ is the mapping torus of
  $\bfK(\Phi^{-1}) \colon \bfK(\cala) \to \bfK(\cala)$ and $\bfNK(\cala_{\Phi}[t^{\pm}])$
  is the homotopy fiber of the map $\bfK(\cala_{\Phi}[t^{-1}]) \to \bfK(\cala)$ given by
  evaluation $t = 0$;
  
\item\label{the:BHS_decomposition_for_connective_K-theory:Nil} There exist a functor
  $\bfE\colon (\addcatic)^\IZ\to\Spectra$ and weak homotopy equivalences of spectra,
  natural in $(\cala,\Phi)$,
  \begin{eqnarray*}
   \Omega \bfN\bfK(\cala_{\Phi}[t]) & \xleftarrow{\simeq} & \bfE(\cala,\Phi);
   \\
   \bfK(\cala) \vee  \bfE (\cala,\Phi)  & \xrightarrow{\simeq} & \bfK(\Nil(\cala,\Phi)),
  \end{eqnarray*}
  where $\bfK(\Nil(\cala,\Phi))$ is the connective $K$-theory of a certain  Nil-category $\Nil(\cala,\Phi)$.
   \end{enumerate}
\end{theorem}

The purpose of the following sections is to find properties of $\cala$, which imply for
any automorphism $ \Phi$ the vanishing of the Nil-terms above and are hopefully inherited by the
passage from $\cala$ to $\cala[t,t^{-1}]$.

%%%%%%%%%%%%%%%%%%%%%%%%%%%%%%%%%%%%%%%%%%%%%%%%%%%%%%%%%%%%%%%%%%%%
%%%%%%%%%%%%%%%%%%%%%%%%%%%% Section 4 %%%%%%%%%%%%%%%%%%%%%%%%%%%%%%%%
%%%%%%%%%%%%%%%%%%%%%%%%%%%%%%%%%%%%%%%%%%%%%%%%%%%%%%%%%%%%%%%%%%%%

\typeout{------ Section:  $\IZ\cala$-modules and the Yoneda embedding ----------}

\section{$\IZ\cala$-modules and the Yoneda embedding}%
\label{sec:IZ_cala-modules_and_the_Yoneda_embedding}

%%%%%%%%%%%%%%%%%%%%%%%%%%%%%%%%%%%%%%%%%%%%%%%%%%%%%%%%%%%%%%%%%%%%%

\subsection{Basics about $\IZ\cala$-modules}%
\label{subsec:Basics_about_IZ_cala-modules}

Let $\cala$ be a  $\IZ$-category. 
We denote by $\MODcatl{\IZ\cala}$ and $\MODcatr{\IZ\cala}$ respectively the
abelian category of covariant or contravariant respectively functors of
$\IZ$-categories $\cala$ to $\MODcatl{\IZ}$.  The abelian structure comes from
the abelian structure in $\MODcatl{\IZ}$.  For instance, a sequence
$F_0 \xrightarrow{T_1} F_1 \xrightarrow{T_2} F_2$ in $\MODcatr{\IZ\cala}$ is declared
to be \emph{exact}, if for each object $A \in \cala$ the evaluation at $A$ yields an exact
sequence of $\IZ$-modules
$F_0(A) \xrightarrow{T_1(A)} F_1(A) \xrightarrow{T_2(A)} F_2(A)$. The cokernel and kernel
of a morphism $T \colon F_0 \to F_1$ are defined by taking for each object $A \in \cala$
the kernel or cokernel of the morphism $T(A) \colon F_0(A) \to F_1(A)$ in
$\MODcatr{\IZ}$.

In the sequel $\IZ\cala$-module means contravariant $\IZ\cala$-module, unless
specified explicitly differently.

Given an object $A$ in $\cala$, we obtain an object $\mor_{\cala}(?,A)$ in
$\MODcatr{\IZ\cala}$ by assigning to an object $B$ the $\IZ$-module
$\mor_{\cala}(B,A)$ and to a morphism $g \colon B_0 \to B_1$ the $\IZ$-homomorphism
$g^* \colon \mor_{\cala}(B_1,A) \to \mor_{\cala}(B_0,A)$ given by precomposition with $g$.

The elementary proof of the following lemma is left to the reader.

\begin{lemma}[Yoneda Lemma]\label{lem:Yoneda_lemma}
For each object $A$ in $\cala$ and each object $M$ in $\MODcatr{\IZ\cala}$, we obtain an
isomorphism of $\IZ$-modules
\[
\mor_{\MODcatr{\IZ\cala}}(\mor_{\cala}(?,A),M(?)) \xrightarrow{\cong} M(A), \quad T \mapsto T(A)(\id_A).
\]
\end{lemma}

We call a $\IZ\cala$-module $M$ \emph{free}, if it is isomorphic as $\IZ\cala$-module to
$\bigoplus_{I} \mor_{\cala}(?,A_i)$ for some collection of objects $\{A_i \mid i \in I\}$
in $\cala$ for some index set $I$. A $\IZ\cala$-module $M$ is called \emph{projective}, if
for any epimorphism $p \colon N_0 \to N_1$ of $\IZ\cala$-modules and morphism
$f \colon M \to N_1$ there is a morphism $\overline{f} \colon M \to N_0$ with
$p \circ \overline{f} = f$.  A $\IZ\cala$-module $M$ is \emph{finitely generated}, if there
exists a collection of objects $\{A_j \mid j \in J\}$ in $\cala$ for some finite index set
$J$ and an epimorphism of $\IZ\cala$-modules
$\bigoplus_{j \in J} \mor_{\cala}(?,A_j) \to M$.  Equivalently, $M$ is finitely generated,
if there exists a finite collection of objects $\{A_j \mid j \in J\}$ in $\cala$ together
with elements $x_j \in M(A_j)$ such that for any object $A$ and any $x \in M(A)$ there are
morphisms $\varphi \colon A \to A_j$ such that $x = \sum_{j} M(\varphi_j)(x_j)$.  (The
$x_j$ are the images of $\id_{A_j}$ under the above epimorphism.)  Given a collection of
objects $\{A_i \mid i \in I\}$ in $\cala$ for some index set $I$, the free
$\IZ\cala$-module $\bigoplus_{I} \mor_{\cala}(?,A_i)$ is finitely generated, if and only if
$I$ is finite. A $\IZ\cala$-module $M$ is \emph{finitely presented}, if there are finitely
generated free $\IZ\cala$-modules $F_1$ and $F_0$ and an exact sequence
$F_1 \to F_0 \to M \to 0$.  We say that a $\IZ\cala$-module \emph{has projective dimension
  $\le d$}, denoted by $\pdim_{\IZ\cala}(M) \le d$, for a natural number $d$, if there
exists an exact sequence $0 \to P_d \to P_{d-1} \to \cdots \to P_1 \to P_0 \to M \to 0$
such that each $\IZ \cala$-module $P_i$ is projective. If we replace projective by free,
we get an equivalent definition, if $d \ge 1$.  We call a $\IZ\cala$-module \emph{of type
  $\FL$} or \emph{of type $\FP$} respectively, if there exists an exact sequence of finite
length $0 \to F_n \to F_{n-1} \to \cdots \to F_1 \to F_0 \to M \to 0$ such that each
$\IZ \cala$-module $F_i$ is finitely generated free or finitely generated projective
respectively.

\begin{remark}\label{rem:comparing_notions_of_IZCala_modules}
Note the setting in this paper is different from the one appearing in~\cite{Lueck(1989)},
since here a $\IZ\cala$-module $M$ satisfies $M(f+g) = M(f) + M(g)$ for two morphisms
$f,g \colon A \to B$, which is not required in~\cite{Lueck(1989)}.  Nevertheless many of
the arguments in~\cite{Lueck(1989)} carry over to the setting of this paper because of the
Yoneda Lemma~\ref{lem:Yoneda_lemma}, which replaces the corresponding Yoneda Lemma
in~\cite[Subsection~9.16 on page 167]{Lueck(1989)}.
\end{remark} 

However, the next result has no analogue in the setting of~\cite{Lueck(1989)}.

\begin{lemma}\label{lem:compatibility_with_direct_sums} Let $\cala$ be an additive
  category. For a $\IZ\cala$-module $M$ and objects $A_1, A_2,\ldots, A_n$, we obtain a natural
  isomorphism
  \[
    \bigoplus_{i=1}^n M(\pr_i) \colon \bigoplus_{i = 1}^n M(A_i) \xrightarrow{\cong}
    M\biggl(\bigoplus_{i = 1}^n A_i\biggr),
  \]
  where $\pr_j \colon \bigoplus_{i=1}^n A_i \to A_j$ is the canonical projection for
  $j = 1,2 \ldots, n$.
\end{lemma}
\begin{proof}
  One easily checks, using the fact that the functor $M$ is compatible with the
  $\IZ$-module structures on the morphisms, that the inverse is given
  \[
    M\biggl(\bigoplus_{i = 1}^n A_i\biggr) \to \bigoplus_{i = 1}^n M(A_i), \quad x \mapsto
    \bigl(M(k_i)(x)\bigr)_i,
  \]
  where $k_j \colon A_j \to \bigoplus_{i=1}^n A_i$ is the inclusion of the $j$-the summand
  for $j = 1,2 \ldots, n$.
\end{proof}

\begin{lemma}\label{lem:IZ_cala_modules}
Let $\cala$ be a $\IZ$-category.
\begin{enumerate}

\item\label{lem:IZ_cala_modules:stably-free_implies_projective}
Every free $\IZ \cala$-module is projective;

\item\label{lem:IZ_cala_modules:FF_and_exact_sequences} Let $0 \to M \to M' \to M''\to 0$
  be an exact sequence of $\IZ\cala$-modules.  If both $M$ and $M''$ are free or
  projective respectively, then $M'$ is free or projective respectively;

\item\label{lem:IZ_cala_modules:exact_sequences_FL} Let $0 \to M \to M' \to M''\to 0$
  be an exact sequence of $\IZ\cala$-modules.  If two of the $\IZ\cala$-modules
  $M$, $M'$ and $M''$ are of type $\FL$ or $\FP$ respectively, then all three are of type
  $\FL$ or $\FP$ respectively;

\item\label{lem:IZ_cala_modules:finite_dimension} Let $C_*$ be a projective
  $\IZ\cala$-chain complex i.e., a $\IZ\cala$-chain complex, all whose chain
  modules $C_n$ are projective.  Then the following assertions are equivalent:
  \begin{enumerate}
  \item\label{lem:IZ_cala_modules:finite_dimension:(1)} Consider a natural number
    $d$.  Let $B_d(C_*)$ be the image of $c_{d+1} \colon C_{d+1} \to C_d$ and
    $j \colon B_d(C) \to C_d$ be the inclusion.  There is a $\IZ\cala$-submodule
    $C_d^{\perp}$ such that for the inclusion $i \colon C_d^{\perp} \to C_d$ the map
    $i \oplus j \colon C_d^{\perp} \oplus B_d(C_*) \to C_d$ is an isomorphism.  Moreover,
    the following chain map from a $d$-dimensional projective $\IZ\cala$-chain complex
    to $C_*$ is a $\IZ\cala$-chain homotopy equivalence
\[
\quad \quad \quad \quad \quad \xymatrix{
\cdots \ar[r]
& 
0 \ar[r] \ar[d]
& 
0 \ar[r] \ar[d]
& 
C_d^{\perp} \ar[r]^{c_d \circ i} \ar[d]^{i}
& 
C_{d-1} \ar[r]^{c_{d-1}} \ar[d] ^{\id_{C_{d-1}}}
& 
\cdots \ar[r]^{c_1} 
&
C_0 \ar[d]^{\id_{C_0}}
\\
\cdots \ar[r]^{c_{d+3}}
& 
C_{d+2} \ar[r]^{c_{d+2}}
& 
C_{d+1}\ar[r]^{c_{d+1}} 
& 
C_d \ar[r]^{c_d}
& 
C_{d-1} \ar[r]^{c_{d-1}}
& 
\cdots \ar[r]^{c_1} 
&
C_0; 
\\
}
\]
\item\label{lem:IZ_cala_modules:finite_dimension:(2)} $C_*$ is $\IZ\cala$-chain
  homotopy equivalent to a $d$-dimensional projective $\IZ\cala$-chain complex;
\item\label{lem:IZ_cala_modules:finite_dimension:(3)} $C_*$ is dominated by
  $d$-dimensional projective $\IZ\cala$-chain complex $D_*$, i.e., there are
  $\IZ\cala$-chain maps $i \colon C_* \to D_*$ and $r_* \colon D_* \to   C_*$ 
  satisfying $r_* \circ i_* \simeq \id_{C_*}$;
\item\label{lem:IZ_cala_modules:finite_dimension:(4)}
  $B_d(C_*)$ is a direct summand in $C_d$ and $H_i(C_*) = 0$ for $i \ge d+1$;
\item\label{lem:IZ_cala_modules:finite_dimension:(5)} $H^{d+1}_{\IZ\cala}(C_*;M) :=
  H^{d+1}(\hom_{\IZ\cala}(C_*,M))$ vanishes for every $\IZ\cala$-module $M$ and
  $H_i(C_*) = 0$ for all $i  \geq d+1$;
\end{enumerate}

\item\label{lem:IZ_cala_modules:exact_sequences_homological_dimension} Let
  $0 \to M \to M' \to M''\to 0$ be an exact sequence of $\IZ\cala$-modules.
  \\  If $\pdim_{\IZ\cala}(M), \pdim_{\IZ\cala}(M'') \leq d$, then $\pdim_{\IZ\cala}(M') \le d$;
  \\  If $\pdim_{\IZ\cala}(M), \pdim_{\IZ\cala}(M') \leq d$, then $\pdim_{\IZ\cala}(M'') \leq  d+1$;
  \\  If $\pdim_{\IZ\cala}(M') \leq d$, $\pdim_{\IZ\cala}(M'') \leq d+1$, then $\pdim_{\IZ\cala}(M) \leq d$;

\item\label{lem:IZ_cala_modules:additive} Suppose that  $\cala$ is an additive
  category. For two objects $A_0$ and $A_1$ in
  $\cala$ together with a choice of a direct sum $i_k \colon A_k \to A_0 \oplus A_1$ for
  $k = 0,1$, the induced $\IZ$-map
  \[
    {i_0}_* \oplus{i_1}_* \colon \mor_{\cala}(?,A_0) \oplus \mor_{\cala}(?,A_1)
    \xrightarrow{\cong} \mor_{\cala}(?,A_0 \oplus A_1)
  \]
  is an isomorphism.  In particular each finitely generated free
  $\IZ\cala$-module is isomorphic to $\IZ\cala$-module of the shape
  $\mor_{\cala}(?,A)$ for an appropriate object $A$ in $\cala$.

\end{enumerate}
\end{lemma}
\begin{proof}~\ref{lem:IZ_cala_modules:stably-free_implies_projective} This follows
  from the Yoneda Lemma~\ref{lem:Yoneda_lemma}.
  \\[1mm]~\ref{lem:IZ_cala_modules:FF_and_exact_sequences} This is obviously true.
  \\[1mm]~\ref{lem:IZ_cala_modules:exact_sequences_FL} The proof is
  analogous to the one of~\cite[Lemma~11.6 on page~216]{Lueck(1989)}. 
  \\[1mm]~\ref{lem:IZ_cala_modules:finite_dimension} The proof is analogous to the one
  of~\cite[Proposition~11.10 on page~221]{Lueck(1989)}.
  \\[1mm]~\ref{lem:IZ_cala_modules:exact_sequences_homological_dimension} This follows
  from~\ref{lem:IZ_cala_modules:finite_dimension} 
  for the projective dimension using the long exact {(co)}ho\-mo\-logy
  sequence associated to a short exact sequence of {(co)}chain complexes, since every
  $\IZ\cala$-module has a free resolution by the Yoneda Lemma~\ref{lem:Yoneda_lemma}.
  \\[1mm]~\ref{lem:IZ_cala_modules:additive} This is obvious and hence 
  the proof of Lemma~\ref{lem:IZ_cala_modules} is finished.
\end{proof}

If $M$ and $N$ are $\IZ\cala$-modules, then
$\hom_{\IZ\cala}(M,N)$ is the $\IZ$-module of $\IZ\cala$-homomorphisms
$M \to N$.  Given a contravariant or covariant $\IZ\cala$-module $M$ and a
$\IZ$-module $T$, we obtain a covariant or contravariant $\IZ\cala$-module
$\hom_{\IZ}(M,T)$ by sending an object $A$ to $\hom_{\IZ}(M(A),T)$. Given a
contravariant $\cala$-module $M$ and covariant $\IZ\cala$-module $N$, their
\emph{tensor product} $M \otimes_{\IZ \cala} N$ is the $\IZ$-module given by
$\bigoplus_{A \in \ob(\cala)} M(A) \otimes_{\IZ} N(A)/T$. Here $T$ is the
$\IZ$-submodule of $\bigoplus_{A \in \ob(\cala)} M(A) \otimes_{\IZ} N(A)$
generated by elements of the form $mf \otimes n - m \otimes fn$ for a morphism
$f \colon A \to B$ in $\cala$, $m \in M(A)$ and $n \in N(B)$, where $mf := M(f)(m)$ and
$fn = N(f)(n)$.  It is characterized by the property that for any $\IZ$-module $T$,
there are natural adjunction isomorphisms
\begin{eqnarray}
\hom_{\IZ}(M \otimes_{\IZ\cala} N,T) & \xrightarrow{\cong} & \hom_{\IZ\cala}(M,\hom_{\IZ}(N,T));
\label{adjunction_otimes_hom_N}
\\
\hom_{\IZ}(M \otimes_{\IZ\cala} N,T) & \xrightarrow{\cong} & \hom_{\IZ\cala}(N,\hom_{\IZ}(M,T)).
\label{adjunction_otimes_hom_M}
\end{eqnarray}

Let $F \colon \cala \to \calb$ be a functor of $\IZ$-categories. Then the restriction functor
\[
F^* \colon \MODcatr{\IZ\calb} \to \MODcatr{\IZ\cala}
\]
is given by precomposition with $F$. The induction functor 
\[
F_* \colon \MODcatr{\IZ\cala} \to \MODcatr{\IZ\calb}
\]
sends a contravariant $\IZ\cala$-module $M$ to
$M(?)  \otimes_{\IZ\cala}  \mor_{\calb}(??,F(?))$.  We get for a $\IZ\calb$-module an identification
$F^*N = \hom_{\IZ\calb}(\mor_{\cala}(??,F(?)),N(??))$ from the Yoneda
Lemma~\ref{lem:Yoneda_lemma}. We conclude from~\eqref{adjunction_otimes_hom_N}  
\begin{equation}
\hom_{\IZ\calb}(F_*M,N) \xrightarrow{\cong} \hom_{\IZ\cala}(M,F^*N)
\label{adjunction_(F_ast,F_upper_ast)}
\end{equation} 
for a $\IZ\cala$-module $M$ and a $\IZ\calb$-module $N$.  The counit
$\beta(N) \colon F_*F^*(N) \to N$ of the adjunction~\eqref{adjunction_(F_ast,F_upper_ast)}
is the adjoint of $\id_{F^*N}$ and sends the equivalence class of $x \otimes f$ for
$x \in N(F(A))$ and $f \in \mor_\calb(B,F(A))$ to $xf = N(f)(x)$.  The unit
$\alpha(M) \colon M \to F^*F_*(M)$ is the adjoint of $\id_{F_*M}$ and sends $x \in M(A)$
to the equivalence class of $x \otimes \id_{F(A)}$.

The functor $F^*$ is 
flat. The functor $F_*$ is compatible with direct sums over arbitrary index sets, is right
exact, see~\cite[Theorem~2.6.1. on page~51]{Weibel(1994)}, and $F_*\mor_{\IZ\cala}(?,C)$
is $\IZ\calb$-isomorphic to $\mor_{\IZ\calb}(?,F(C))$. In particular $F_*$ respect the
properties finitely generated, free, and projective.

%%%%%%%%%%%%%%%%%%%%%%%%%%%%%%%%%%%%%%%%%%%%%%%%%%%%%%%%%%%%%%%%%%%%%

\subsection{The Yoneda embedding}%
\label{subsec:The_Yoneda_embedding}

The \emph{Yoneda embedding} is the following covariant functor
\begin{equation} \iota \colon \cala \to \MODcatr{\IZ\cala}.
 \label{Yoneda_embedding}
\end{equation}
It sends an object $A$ to $\iota(A) = \mor_{\cala}(?,A)$ and a morphism
$f \colon A_0 \to A_1$ to the transformation
$\iota(f) \colon \mor_{\cala}(?,A_0) \to \mor_{\cala}(?,A_1)$ given by composition with
$f$. Let $\MODcatr{\IZ\cala}_{\cala}$ be the full subcategory of $\MODcatr{\IZ\cala}$
consisting of $\IZ\cala$-modules $\mor_{\cala}(?,A)$ for any object $A$ in $\cala$.
Let $\MODcatr{\IZ\cala}_{\fgf}$ be the full subcategory of $\MODcatr{\IZ\cala}$
consisting of finitely generated free $\IZ\cala$-modules.

\begin{definition}\label{def:exact_sequence_in_cala}
Let $\cala$ be a $\IZ$-category. We call a sequence
$A_0 \xrightarrow{f_0} A_1 \xrightarrow{f_1} A_2$ in $\cala$ \emph{exact at $A_1$}, if
$f_1 \circ f_0 = 0$ and, for every object $A$ and morphism $g \colon A \to A_1$ with
$f_1 \circ g = 0$, there exists a morphism $\overline{g} \colon A \to A_0$ with
$f_0 \circ \overline{g} = g$.
\end{definition}

\begin{lemma}\label{lem:Yoneda_equivalence} 
 If $\cala$ is a $\IZ$-category, the Yoneda embedding~\eqref{Yoneda_embedding}
  induces an equivalence of $\IZ$-categories denoted by the same symbol
\[
\iota \colon \cala \to \MODcatr{\IZ\cala}_{\cala}.
\]
If $\cala$ is an additive  category, the Yoneda embedding~\eqref{Yoneda_embedding}
  induces an equivalence of  additive categories denoted by the same symbol
\[
\iota \colon \cala \to \MODcatr{\IZ\cala}_{\fgf}.
\]
Both functors are faithfully flat.
\end{lemma}
\begin{proof}
This follows directly from the Yoneda Lemma~\ref{lem:Yoneda_lemma} and
Lemma~\ref{lem:IZ_cala_modules}~\ref{lem:IZ_cala_modules:additive}.
\end{proof}

The gain of Lemma~\ref{lem:Yoneda_equivalence} is that we have embedded $\cala$ as a full
subcategory of the abelian category $\MODcatr{\IZ\cala}$ and we can now do certain
standard homological constructions in $\MODcatr{\IZ\cala}$, which a priori make no sense
in $\cala$.

The elementary proof of the following lemma based on Lemma~\ref{lem:Yoneda_equivalence} is
left to the reader.

\begin{lemma}\label{lem:idempotent_complete_and_Zcala_modules}
  An additive category $\cala$ is idempotent complete, if and only if every finitely
  generated projective $\IZ\cala$-module is a finitely generated free $\IZ\cala$-module.
\end{lemma}

%%%%%%%%%%%%%%%%%%%%%%%%%%%%%%%%%%%%%%%%%%%%%%%%%%%%%%%%%%%%%%%%%%%%
%%%%%%%%%%%%%%%%%%%%%%%%%%%%%%%%%%% Section 5:%%%%%%%%%%%%%%%%%%%%%%%%%
%%%%%%%%%%%%%%%%%%%%%%%%%%%%%%%%%%%%%%%%%%%%%%%%%%%%%%%%%%%%%%%%%%%%

\typeout{----------- Section: Properties of additive categories  ----------}

\section{Regularity properties of additive categories}%
\label{sec:Regularity_properties_of_additive_IZ-categories}

%%%%%%%%%%%%%%%%%%%%%%%%%%%%%%%%%%%%%%%%%%%%%%%%%%%%%%%%%%%%%%%%%%%%

\subsection{Definition of regularity properties in terms of the Yoneda embedding}%
\label{subsec:Definition_of_regularity_properties_in_terms_of_the_Yoneda_embedding}

Recall  the following standard  ring theoretic notions:

\begin{definition}[Regularity properties of rings]%
\label{def:Regularity_properties_of_rings}

Let $R$ be a  ring and let $l$ be a natural number.

\begin{enumerate}

\item\label{def:Regularity_properties_of_rings;Noetherian}
We call $R$ \emph{Noetherian}, if  any $R$-submodule of a finitely generated
$R$-module is again finitely generated;

\item\label{def:Regularity_properties_of_rings:regular_coherent}
We call $R$  \emph{regular coherent}, if  every finitely presented
$R$-module $M$ is of type $\FP$;

\item\label{def:Regularity_properties_of_rings:l-uniformly_regular_coherent}
We call $R$  \emph{$l$-uniformly regular coherent}, if  every finitely presented
$R$-module $M$ admits an $l$-dimensional finite projective
  resolution, i.e., there exist an exact sequence
  $0 \to P_l \to P_{l-1} \to \cdots \to P_0 \to M \to 0$ such that each $P_i$ is finitely
  generated projective;

\item\label{def:Regularity_properties_of_rings:von_Neumann_regular} We call $R$
 \emph{von  Neumann regular}, if for any element $r \in R$ there exists an element $s \in R$ with
  $r = rsr$;

\item\label{def:Regularity_properties_of_rings:regular}
We call $R$ \emph{regular}, if  it is Noetherian and regular coherent;

\item\label{def:Regularity_properties_of_ringss:l-uniformly_regular_coherent}
We call $R$  \emph{$l$-uniformly regular}, if  it is Noetherian and $l$-uniformly regular coherent;
 
\item\label{def:Regularity_properties_of_rings:global_homological_dimension}
We say that $R$ has global dimension $\le l$, if 
each $R$-module $M$ has projective dimension $\le l$.

\end{enumerate}
\end{definition}

The notion von Neumann regular should not be confused with the notion regular. It stems
from operator theory.  A ring is von Neumann regular, if and only if it is $0$-uniformly
regular coherent. For more information about von Neumann regular rings, see for
instance~\cite[Subsection~8.2.2 on pages 325-327]{Lueck(2002)}.

Let $\cala$ be an additive category. Then we define analogously:

\begin{definition}[Regularity properties of additive categories]%
\label{def:Regularity_properties_of_additive_categories}

Let $\cala$ be an  
additive category and let $l$ be a natural number.

\begin{enumerate}

\item\label{def:Regularity_properties_of_additive_categories;Noetherian}
  We call $\cala$ \emph{Noetherian} if the category
  $\MODcatr{\IZ\cala}$ is Noetherian in the sense that any
  $\IZ\cala$-submodule of a finitely generated $\IZ\cala$-module is
  again finitely generated,
  see~\cite[p.18]{Mitchell(1972)}\footnote{In~\cite[page~18]{Mitchell(1972)}
    this is called left Noetherian; one obtains right Noetherian by
    working with $\MODcatl{\IZ\cala}$ in place of
    $\MODcatr{\IZ\cala}$};

\item\label{def:Regularity_properties_of_additive_categories:regular_coherent}
We call $\cala$ \emph{regular coherent}, if  every finitely presented
$\IZ\cala$-module $M$ is of type $\FP$;

\item\label{def:Regularity_properties_of_additive_categories:l-uniformyl_regular_coherent}
  We call $\cala$ \emph{$l$-uniformly regular coherent}, if every finitely presented
  $\IZ\cala$-module $M$ possesses an $l$-dimensional finite projective
  resolution, i.e., there exist an exact sequence
  $0 \to P_l \to P_{l-1} \to \cdots \to P_0 \to M \to 0$ such that each $P_i$ is finitely
  generated projective;

\item\label{def:Regularity_properties_of_additive_categories:regular}
We call $\cala$ \emph{regular}, if it is Noetherian and regular coherent;

\item\label{def:Regularity_properties_of_additive_categories:l-uniformly_regular_coherent}
  We call $\cala$ \emph{$l$-uniformly regular}, if is Noetherian and $l$-uniformly
  regular coherent;
 
\item\label{def:Regularity_properties_of_additive_categories:global_homological_dimension}
  We say that $\cala$ has global dimension $\le l$, if each $\IZ\cala$-module $M$ has
  projective dimension $\le l$, see~\cite[page~42]{Mitchell(1972)}. 

\end{enumerate}
\end{definition}

%%%%%%%%%%%%%%%%%%%%%%%%%%%%%%%%%%%%%%%%%%%%%%%%%%%%%%%%%%%%%%%%%%%%

\subsection{The definitions of the regularity properties for rings and additive categories are compatible}%
\label{subsec:The_definitions_of_the_regularity_properties_for_rings_and_additive_categories_are_compatible}

\begin{lemma}\label{lem:R-MOD-Versus_MOD-IZ_underline(R)_oplus} Let $R$ be a  ring.
 The functor
  \[
  F \colon \MODcatl{R} \to \MODcatr{\IZ\underline{R}_{\oplus}}
  \]
  sending    $M$   to    $\hom_R(\theta_{\fgf}(-),M)$    is   an    equivalence   of     additive
  categories, is faithfully  flat, and respects each of  the properties finitely
  generated, free  and projective, where the equivalence $\theta_{\fgf}$ has been defined 
in~\eqref{underline(r)_oplus_to_MODcat(R)_fgf}
\end{lemma}
\begin{proof} In the sequel we denote by $[n]$ the $n$-fold direct sum in
  $\underline{R}_{\oplus}$ of the unique object in $\underline{R}$. Notice that
  $\theta([n]) = R^n$. Define
\[
G \colon \MODcatr{\IZ\underline{R}_{\oplus}} \to \MODcatl{R}
\]
by sending $M$ to $M(\theta(1))$. There is a natural equivalence
$G\circ F \to \id_{\MODcatl{R}}$ of functors of additive categories, its value on
the $R$-module $M$ is given by evaluating at $1 \in R = \theta([1])$,
\[
 G \circ F(M) = \hom_R(\theta([1]),M) \xrightarrow{\cong} M.
\]
Next we construct an equivalence $S \colon F \circ G \to \id_{\MODcatl{R}}$ of functors of
additive categories.  For a $\IZ\cala$-module $N$ and objects $A_1, \ldots, A_n$, we obtain
from Lemma~\ref{lem:compatibility_with_direct_sums} a natural isomorphism
\[
\bigoplus_{i=1}^n N(\pr_i) \colon \bigoplus_{i = 1}^n N(A_i) 
\xrightarrow{\cong} N\biggl(\bigoplus_{i = 1}^n A_i\biggr),
\]
where $\pr_j \colon \bigoplus_{i=1}^n A_i \to A_j$ is the canonical projection for
$j = 1,2 \ldots, n$.  

Recall that $[n]$ is the $n$-fold direct sum of copies of $[1]$, in other words, we have an 
identification $[n] = \bigoplus_{i =1}^n [1]$. It induces an isomorphism
\[
\bigoplus_{i =1}^n \theta([1])  \xrightarrow{\cong} \theta([n]).
\]
Given an object $[n]$ in $\underline{R}_{\oplus}$ and an $R$-module $M$,
we define $S(M)([n])$ by the following composite of $R$-isomorphisms
\begin{multline*}
%S(M)([n]) \colon
F \circ G(M)([n]) = \hom_R\bigl(\theta([n]),M(\theta(1))\bigr)
\xrightarrow{\cong} 
\hom_R\biggl(\bigoplus_{k = 1}^n \theta([1]),M(\theta(1))\biggr)
\\
\xrightarrow{\cong} 
\bigoplus_{k = 1}^n  \hom_R\bigl(\theta([1]),M(\theta(1))\bigr)
=
\bigoplus_{k = 1}^n  \hom_R\bigl(R,M(\theta(1))\bigr)
\\
\xrightarrow{\cong} 
\bigoplus_{k = 1}^n  M(\theta(1))
\xrightarrow{\cong}
M\biggl(\bigoplus_{i = 1}^n \theta([1])\biggr)
=
M(\theta([n])).
\end{multline*}
The functor $F$ is faithfully exact, since for any object $[n]$ in
$\underline{R}_{\oplus}$ there is an $R$-isomorphism
$\bigoplus_{i= 1}^n M \xrightarrow{\cong} F(M)([n])$, natural in $M$.
Since $F$ is compatible with direct sums over arbitrary index sets and
sends $R$ to
$\hom_R(\theta(-),R) = \mor_{\underline{R}_{\oplus}}(?,[1])$, it
respects the properties finitely generated, free and projective.
\end{proof}

The following lemma implies in particular that the inclusion $i \colon \cala \to \Idem(\cala)$ induces equivalences 
\begin{equation*}
	\xymatrix{\MODcatr{\IZ\cala} \ar[rr]<.5ex>^{i_*} & & \MODcatr{\IZ\Idem(\cala)}\ar[ll]<.5ex>^{i^*}
	}. 
      \end{equation*}

\begin{lemma}\label{lem:inclusion_of_full_cofinal_subcategory}
  Let $i \colon \cala \to \cala'$ be a  inclusion of an additive  subcategory
  $\cala$ of the  additive  subcategory $\cala'$,
  which is full and cofinal,  for instance $\cala \to \cala' = \Idem(\cala)$.  
  Then:
\begin{enumerate}

\item\label{lem:inclusion_of_full_cofinal_subcategory:M_is_i_upper_ast_I_ast_M} 
 If $M$ is a $\IZ\cala$-module,   then the adjoint 
  \[
   \alpha(M) \colon M \xrightarrow{\cong}  i^*i_*M
   \] 
   of $\id_{i_*M}$ under the adjunction~\eqref{adjunction_(F_ast,F_upper_ast)} is an
   isomorphism of $\IZ\cala$-modules, natural in $M$;

\item\label{lem:inclusion_of_full_cofinal_subcategory:i_upper_ast_faithfully_flat} The
  restriction functor $i^* \colon \MODcatr{\IZ\cala'} \to \MODcatr{\IZ\cala}$ is
  faithfully flat. It sends a finitely generated $\IZ\cala'$-module to a finitely
  generated $\IZ\cala$-module and a projective $\IZ\cala'$-module to a projective
  $\IZ\cala$-module;

\item\label{lem:inclusion_of_full_cofinal_subcategory:i_lower_ast_faithfully_flat} The
  induction functor $i_* \colon \MODcatr{\IZ\cala} \to \MODcatr{\IZ\cala'}$ is
  faithfully flat. It sends a finitely generated $\IZ\cala$-module to a finitely generated
  $\IZ\cala'$-module and a projective $\IZ\cala$-module to a projective
  $\IZ\cala'$-module;

\item\label{lem:inclusion_of_full_cofinal_subcategory:i_upper_ast_i_ast_M_versus_M} 
If $M'$ is a $\IZ\cala'$-module,   then the adjoint 
  \[
   \beta(M') \colon i_*i^*M' \xrightarrow{\cong}  M'
   \] 
  of $\id_{i^*M'}$ under the  adjunction~\eqref{adjunction_(F_ast,F_upper_ast)}
  is an isomorphism of $\IZ\cala'$-modules, natural in $M'$;

\item\label{lem:inclusion_of_full_cofinal_subcategory:Noetherian}
$\cala$ is Noetherian, if and only $\cala'$ is Noetherian;

\item\label{lem:inclusion_of_full_cofinal_subcategory:regular_coherent}
The category $\cala$ is regular coherent or $l$-uniformly regular coherent respectively,
if and only if $\cala'$ is regular coherent or $l$-uniformly regular coherent;

\item\label{lem:inclusion_of_full_cofinal_subcategory:global_dimension}
The category $\cala$ is of global dimension $\le l$,
if and only if $\cala'$ is of global dimension $\le l$.

\end{enumerate}

\end{lemma}

\begin{proof}~\ref{lem:inclusion_of_full_cofinal_subcategory:M_is_i_upper_ast_I_ast_M}  
  This follows from the   Yoneda-Lemma~\ref{lem:Yoneda_lemma}, namely, an inverse 
  of $\alpha(M)$ is given by
  \begin{multline*}
    i^*i_*M = M(?) \otimes_{\IZ\cala} \mor_{\IZ\cala'}(i({?'}),i(?))  = M(?) \otimes_{\IZ\cala} \mor_{\IZ\cala}({?'},?)
     \xrightarrow{\cong} M({?'}), 
    \\
    x  \otimes \phi \mapsto x\phi = M(\phi)(x).
  \end{multline*}
  \\~\ref{lem:inclusion_of_full_cofinal_subcategory:i_upper_ast_faithfully_flat}
  Obviously $i^*$ is flat.  
  
  Consider a sequence of $\IZ\cala'$-modules
  $M_0 \xrightarrow{f_0} M_1 \xrightarrow{f_1} M_1$ such that restriction with $i$ yields
  the exact sequence of $\IZ\cala$-modules 
  $i^*M_0 \xrightarrow{i^*f_0} i^*M_1 \xrightarrow{i^*f_1} i^*M_1$.  We have to show for
  any object $A'$ in $\cala'$ that the sequence of $R$-modules
  $M_0(A') \xrightarrow{f_0(A')} M_1(A') \xrightarrow{f_1(A')} M_1(A')$ is exact. Since
  $\cala$ is by assumption cofinal in $\cala'$,  we can find an object $A$ in $\cala$ and
  and morphisms $j \colon A' \to i(A)$ and $r \colon i(A) \to A'$ in $\cala'$ satisfying
  $r \circ i = \id_{A'}$.  We obtain the following commutative diagram of $R$-modules
\[
\xymatrix@!C=6em{M_0(A') \ar[r]^{f_0(A')} \ar[d]^{M_0(j)}
&
M_1(A') \ar[r]^{f_1(A')} \ar[d]^{M_1(j)}
& M_2(A')  \ar[d]^{M_2(j)}
\\
M_0(i(A)) \ar[r]^{f_0(i(A))}  \ar[d]^{M_0(r)}
&
M_1(i(A)) \ar[r]^{f_1(i(A))} \ar[d]^{M_1(r)}
& M_2(i(A)) \ar[d]^{M_2(r)}
\\
M_0(A') \ar[r]^{f_0(A')} 
&
M_1(A') \ar[r]^{f_1(A')} 
& M_2(A')
}
\]
such that the composite of the two vertical arrows appearing in each of the three columns is
the identity.  Since the middle horizontal sequence is exact, an easy diagram chase
shows that the upper horizontal sequence is exact. This shows that $i^*$ is faithfully flat.

Consider an object $A'$ in $\cala'$. Since $\cala$ is by assumption cofinal in $\cala$, we
can find an object $A$ in $\cala$ and and morphism $j \colon A' \to i(A)$ and
$q \colon i(A) \to A'$ in $\cala'$ satisfying $q \circ j = \id_{A'}$.  Composition with $q$
and $j$ yield maps of $\IZ\cala'$-modules
$J \colon \mor_{\cala'}({?'},A') \to \mor_{\cala'}({?'},i(A))$ and
$Q \colon \mor_{\cala'}({?'},i(A)) \to \mor_{\cala'}({?'},A')$ satisfying
$Q \circ J = \id_{\mor_{\cala'}({?'},A')}$. If we apply $i^*$, we obtain homomorphisms of
$\IZ\cala'$-modules $i^*J \colon i^*\mor_{\cala'}({?'},A') \to i^*\mor_{\cala'}({?'},i(A))$
and $i^*Q \colon i^*\mor_{\cala'}({?'},i(A)) \to i^*\mor_{\cala'}({?'},A')$ satisfying
$i^*Q \circ i^* J = \id_{i^*\mor_{\cala'}({?'},A')}$.  Since
$i^*\mor_{\cala'}({?'},i(A)) = \mor_{\cala'}(i({?'}),i(A)) = \mor_{\cala}({?'},A)$, the
$\IZ\cala$-module $i^*\mor_{\cala'}({?'},A')$ is a direct summand in
$\mor_{\cala}({?'},A)$ and hence a finitely generated projective
$\IZ\cala$-module. 

Let $M'$ be a finitely generated $\IZ\cala'$-module. Fix 
an epimorphism $\mor_{\cala}({?'},A') \to M'$ for some object $A'$ in
$\cala'$. We conclude that the $\IZ\cala$-module $i^*M$ is a quotient of $\mor_{\cala}(?,A)$ 
for some object $A$ in $\cala$ and hence
finitely generated. Hence $i^*$ respects the property finitely generated.

Let $P$ be a projective $\IZ\cala'$-module. Then we can find a collection of objects
$\{A'_k \mid k \in K\}$ together with an epimorphism
$\bigoplus_{k \in K} \mor_{\cala'}(?,A'_k) \to P$  by the Yoneda
Lemma~\ref{lem:Yoneda_lemma}.  Since $P$ is projective, $P$ is a direct summand in
$\bigoplus_{k \in K} \mor_{\cala'}(?,A'_k)$.  This implies that  $i^*P$ is a direct summand in the direct
sum $\bigoplus_{k \in K} i^*\mor_{\cala'}(?,A'_i)$ of projective $\IZ\cala$-modules
and hence itself a projective $\IZ\cala$-module. Hence $i^*$ respects the property
projective.  
\\[1mm]~\ref{lem:inclusion_of_full_cofinal_subcategory:i_lower_ast_faithfully_flat} 
The faithful flatness follows from
assertions~\ref{lem:inclusion_of_full_cofinal_subcategory:M_is_i_upper_ast_I_ast_M}
and~\ref{lem:inclusion_of_full_cofinal_subcategory:i_upper_ast_faithfully_flat}. Since
$i_*\mor_{\cala}(?,A) = \mor_{\cala'}(?,i(A))$ holds for any object $A$ in $\cala$, the
functor $i_*$ respects the properties finitely generated and projective.
\\[1mm]~\ref{lem:inclusion_of_full_cofinal_subcategory:i_upper_ast_i_ast_M_versus_M} 
We begin with the case $M = \mor_{\cala'}({?'},i(A)) = i_* \mor_{\cala'}(?,A)$ for some object $A$ in $\cala$.
Then the claim follows from
assertion~\ref{lem:inclusion_of_full_cofinal_subcategory:M_is_i_upper_ast_I_ast_M} applied
to the $\IZ\cala$-module $\mor_{\cala'}(?,A)$, since in this case $\beta(M) = i_*\alpha(M)$. 
Consider an object $A'$ in $\cala'$. Since $\cala$ is by assumption cofinal in $\cala$, we
can find an object $A$ in $\cala$ and and morphism $j \colon A' \to i(A)$ and
$q \colon i(A) \to A'$ in $\cala'$ satisfying $q \circ j = \id_{A'}$.  Composition with $q$
and $j$ yield maps of $\IZ\cala'$-modules
$J \colon \mor_{\cala'}({?'},A') \to \mor_{\cala'}({?'},i(A))$ and
$Q \colon \mor_{\cala'}({?'},i(A)) \to \mor_{\cala'}({?'},A')$ satisfying
$Q \circ J = \id_{\mor_{\cala'}({?'},A')}$. Hence we get a commutative diagram of $\IZ\cala'$-modules
\[
\xymatrix@!C=14em{
i_*i^*\mor_{\cala'}({?'},A') \ar[r]^-{\beta(\mor_{\cala'}({?'},A'))} \ar[d]^{i_*i^*J}
& 
\mor_{\cala'}({?'},A') \ar[d]^J
\\
i_*i^*\mor_{\cala'}({?'},i(A)) \ar[r]^-{\beta(\mor_{\cala'}({?'},i(A)))}  \ar[d]^{i_*i^*Q}
& 
\mor_{\cala'}({?'},i(A)) \ar[d]^Q
\\
i_*i^*\mor_{\cala'}({?'},A') \ar[r]^-{\beta(\mor_{\cala'}({?'},A'))} 
& 
\mor_{\cala'}({?'},A') 
}
\]
such that the composite of the two vertical maps in each of the two columns is the
identity and the middle arrow is an isomorphism.  Hence the upper arrow is an isomorphism.

For any $\IZ\cala'$-module $M'$ we can find a collection of objects
$\{A'_k \mid k \in K\}$ in $\cala'$ together with an epimorphism
$f_0 \colon F_0 := \bigoplus_{k \in K} \mor_{\cala'}({?'},A'_k) \to M'$ by the Yoneda
Lemma~\ref{lem:Yoneda_lemma}. Repeating this construction for $\ker(f_0)$ instead of $M$, we obtain
another collection $\{A''_l \mid l \in L\}$ of objects in $\cala'$ together with a map
$f_1 \colon F_1 := \bigoplus_{l \in L} i^*\mor_{\cala'}(?,A'_l) \to F_0$ whose image is $\ker(f_1)$.
We obtain from assertions~\ref{lem:inclusion_of_full_cofinal_subcategory:i_upper_ast_faithfully_flat} 
and~\ref{lem:inclusion_of_full_cofinal_subcategory:i_lower_ast_faithfully_flat}
a commutative diagram of $\IZ\cala'$-modules with exact rows
\[
\xymatrix{i_*i^*F_1 \ar[r]^{i_*i^*f_1} \ar[d]^{\beta(F_1)}
&
i_*i^*F_0 \ar[r]^{i_*i^*f_0} \ar[d]^{\beta(F_0)}
& i_*i^*M' \ar[r] \ar[d]^{\beta(M)}
& 0
\\
F_1 \ar[r]^{f_1} 
&
F_0 \ar[r]^{f_0}
& M' \ar[r]
& 0.
}
\]
Since $\beta$ is compatible with direct sums over arbitrary index sets, the maps
$\beta(F_1)$ and $\beta(F_0)$ are isomorphisms.  Hence $\beta(M')$ is an isomorphism.
\\[1mm]~\ref{lem:inclusion_of_full_cofinal_subcategory:Noetherian},~%
~\ref{lem:inclusion_of_full_cofinal_subcategory:regular_coherent}
and~\ref{lem:inclusion_of_full_cofinal_subcategory:global_dimension} They follow now
directly from
assertions~\ref{lem:inclusion_of_full_cofinal_subcategory:M_is_i_upper_ast_I_ast_M},%
~\ref{lem:inclusion_of_full_cofinal_subcategory:i_upper_ast_faithfully_flat},%
~\ref{lem:inclusion_of_full_cofinal_subcategory:i_lower_ast_faithfully_flat}
and~\ref{lem:inclusion_of_full_cofinal_subcategory:i_upper_ast_i_ast_M_versus_M}.
\end{proof}

We conclude from Lemma~\ref{lem:R-MOD-Versus_MOD-IZ_underline(R)_oplus}
and Lemma~\ref{lem:inclusion_of_full_cofinal_subcategory}~%
\ref{lem:inclusion_of_full_cofinal_subcategory:Noetherian},~%
\ref{lem:inclusion_of_full_cofinal_subcategory:regular_coherent}, 
and~\ref{lem:inclusion_of_full_cofinal_subcategory:global_dimension}.

\begin{corollary}\label{cor:regualiyt_for_R-versus_underline(r)_oplus)} 
  Let $R$ be a  ring and let $l$ be a natural number.  Then the
  following assertions are equivalent: 

  \begin{enumerate}

  \item\label{cor:regualiyt_for_R-versus_underline(r)_oplus):R} 

  The ring $R$ is Noetherian,
  regular coherent, $l$-uniformly regular coherent, regular, uniformly $l$-regular, or of
  global dimension $\le l$ in the sense of
  Definition~\ref{def:Regularity_properties_of_rings} respectively;

  \item\label{cor:regualiyt_for_R-versus_underline(r)_oplus):underline(R)_oplus} 
  The  additive category $\underline{R}_{\oplus}$ is Noetherian, regular
  coherent, $l$-uniformly regular coherent, regular, uniformly $l$-regular, or of global
  dimension $\le d$ in the sense of
  Definition~\ref{def:Regularity_properties_of_additive_categories} respectively;

 \item\label{cor:regualiyt_for_R-versus_underline(r)_oplus):Idem(underline(R)_oplus)} 
  The  additive category $\Idem(\underline{R}_{\oplus})$ is Noetherian, regular
  coherent, $l$-uni\-form\-ly regular coherent, regular, uniformly $l$-regular, or of global
  dimension $\le l$ in the sense of
  Definition~\ref{def:Regularity_properties_of_additive_categories} respectively.

\end{enumerate}
\end{corollary}

%%%%%%%%%%%%%%%%%%%%%%%%%%%%%%%%%%%%%%%%%%%%%%%%%%%%%%%%%%%%%%%%%%%%

\subsection{Intrinsic  definitions of the regularity properties}%
\label{subsec:Intrinsic_definitions_of_the_regularity_properties}

One can give an intrinsic definition of the regularity properties above without referring
to the Yoneda embedding.  The situation is quite nice for regular coherent and
$l$-uniformly regular coherent for an idempotent complete additive category as as
explained below.

\begin{lemma}[Intrinsic Reformulation of regular coherent]%
\label{lem:intrinsic_reformulation_of_regular_coherent}
  Let $\cala$ be an idempotent complete additive category.
  
  \begin{enumerate}

  \item\label{lem:intrinsic_reformulation_of_regular_coherent:uniform_l_greater_equal_2}
    Let $l \ge 2$ be a natural
    number.  Then $\cala$ is $l$-uniformly regular coherent, if and only
    if for every morphism $f_1\colon A_1 \to A_0$ we can find a sequence of length $l$ in
    $\cala$
    \[0 \to A_l \xrightarrow{f_l} A_{l-1} \xrightarrow{f_{l-1}} \cdots \xrightarrow{f_2}
      A_{1} \xrightarrow{f_1} A_0
    \]
    which is exact at $A_i$ for $i = 1,2, \ldots, n$;

    \item\label{lem:intrinsic_reformulation_of_regular_coherent:uniform_l_is_1}
     $\cala$ is $1$-uniformly regular coherent, if and only
     if for every morphism $f\colon A_1 \to A_0$ we can find a factorization
     $A_1 \xrightarrow{f_1} B \xrightarrow{f_0} A_0$ of $f$ such that $f_1$ is surjective and $f_0$ is injective;

   \item\label{lem:intrinsic_reformulation_of_regular_coherent:uniform_l_is_0}
     The following assertions are equivalent:
     \begin{enumerate}

     \item\label{lem:intrinsic_reformulation_of_regular_coherent:uniform_l_is_0:(1)}
     $\cala$ is $0$-uniformly regular coherent;

      \item\label{lem:intrinsic_reformulation_of_regular_coherent:uniform_l_is_0:(2)}
      For every morphism $f_1 \colon A_1 \to A_0$ there exists a morphism $f_0 \colon A_0 \to A_{-1}$
      such that $A_1 \xrightarrow{f_1} A_0 \xrightarrow{f_0} A_{-1} \to 0$ is exact;

      \item\label{lem:intrinsic_reformulation_of_regular_coherent:uniform_l_is_0:(3)}
        For every morphism $f \colon A_1 \to A_0$ there exists a morphism $g \colon A_0 \to A_1$ satisfying
           $f \circ g \circ f = f$;
         \end{enumerate}
         
  \item\label{lem:intrinsic_reformulation_of_regular_coherent:not_uniform}
  $\cala$ is regular coherent, if and only if
  for every morphism $f_1\colon A_1 \to A_0$ we can find a sequence of finite length  in
  $\cala$
  \[0 \to A_n \xrightarrow{f_n} A_{n-1} \xrightarrow{f_{n-1}} \cdots \xrightarrow{f_2}
    A_{1} \xrightarrow{f_1} A_0
  \]
  which is exact at $A_i$ for $i = 1,2, \ldots, n$.

\end{enumerate}
\end{lemma}
\begin{proof}~\ref{lem:intrinsic_reformulation_of_regular_coherent:uniform_l_greater_equal_2}
  it suffices to prove that the following
  statements are equivalent:

\begin{enumerate}

\item[(a)] For any morphisms
  $f_1\colon P_1 \to P_0$ of finitely generated projective $\IZ\cala$-modules, we can
  find finitely generated projective $\IZ\cala$-modules $P_2, P_3, \ldots, P_l$ and an exact
  sequence of $\IZ\cala$-modules
  \[
    0 \to P_l \xrightarrow{f_l} P_{l-1}\xrightarrow{f_{l-1}} \cdots \xrightarrow{f_2}P_1
    \xrightarrow{f_1} P_0;
  \]

\item[(b)]  For any finitely presented
  $\IZ\cala$-module $M$, there exists finitely generated projective $\IZ\cala$-modules
  $P_0,P_1, \ldots, P_l$ and an exact sequence of $\IZ\cala$-modules 
    \[ 0 \to P_l \xrightarrow{f_l} P_{l-1} \xrightarrow{f_{l-1}} \cdots \xrightarrow{f_2}
      P_{1} \xrightarrow{f_1} P_0 \xrightarrow{f_0} M \to 0.
    \]
  \end{enumerate}
  The implication implication (b) $\implies$ (a) is obvious, since $\cok(f_1)$ is a
  finitely presented $\IZ\cala$-module.  It remains to prove the implication (a)
  $\implies$ (b).  Let $f_1 \colon P_1 \to P_0$ be a $\IZ\cala$-homomorphism of finitely
  generated projective $\IZ\cala$-modules.  By assumption we can find an exact sequence of
  $\IZ\cala$-modules
\[
0 \to Q_l \xrightarrow{c_l} Q_{l-1} \xrightarrow{c_{n-1}} \cdots   \xrightarrow{c_2} Q_{1} 
\xrightarrow{c_1} Q_0 \xrightarrow{c_0} \cok(f_1) \to 0.
\]
Let $P_*$ be the $1$-dimensional $\IZ\cala$-chain complex whose first differential is $f_1$. Let
$Q_*$ be the $l$-dimensional $\IZ\cala$-chain complex whose $i$th chain module is $Q_i$ for 
$0 \le i \le l$ and whose $i$th differential is $c_i \colon Q_i \to Q_{i-1}$ for $1 \le i \le
l$.  One easily constructs a $\IZ\cala$-chain map $u_* \colon P_* \to Q_*$ such that $H_0(u_*)$
is an isomorphism. Let $\cone(u_*)$ be the mapping cone. We conclude $H_i(\cone(u_*))  = 0$
for $i \not= 2$ from the long exact homology sequence associated to the exact sequence $0
\to P_* \xrightarrow{i_*} \cyl(u_*) \xrightarrow{p_*} \cone(u_*) \to 0$ and the fact that
the canonical projection $q_* \colon \cyl(u_*) \to Q_*$ is a $\IZ\cala$-chain homotopy
equivalence with $q_* \circ i_* = u_*$.  Let $D_* \subseteq \cone(u_*)$ be the
$\IZ\cala$-subchain complex, whose $i$-th chain module is $\cone(u_*)$ for $i \ge 3$, the kernel
of the second differential of $\cone(u_*)$ for $i = 2$ and $\{0\}$ for $i = 0,1$.  Then
$D_i$ is finitely generated projective for $i \ge 0$ and the inclusion $k_* \colon D_* \to
\cone(u_*)$ induces isomorphisms on homology groups. Define the $\IZ\cala$-chain complex $C_*$ by
the pullback
\[
\xymatrix{C_* \ar[r]^{\overline{p_*}} \ar[d]^{\overline{k_*}}
& D_* \ar[d]^{k_*}
\\
\cyl(u_*) \ar[r]^{p_*}
&
\cone(u_*)
}
\]
This can be extended to a commutative diagram of $\IZ\cala$-chain complexes with exact
rows
\[
\xymatrix{0 \ar[r]
&
P_* \ar[r]^-{\overline{i_*}} \ar[d]^{\id}
&
C_* \ar[r]^-{\overline{p_*}} \ar[d]^{\overline{k_*}}
& D_* \ar[d]^{k_*} \ar[r]
& 
0
\\
0 \ar[r]
&
P_* \ar[r]^-{i_*}
&
\cyl(u_*) \ar[r]^-{p_*}
&
\cone(u_*)\ar[r]
& 
0
}
\]
Then $C_*$ is an $l$-dimensional $\IZ\cala$-chain complex whose $\IZ\cala$-chain
modules are finitely generated projective.  Since $D_i = 0$ for $i = 0,1$, we can identify
$P_1 = C_1$ and $P_0 = C_0$ and the first differentials of $P_*$ and $C_*$.  Since $k_*$
induces isomorphisms on homology, the same is true for $\overline{k_*}$.  Hence $C_*$
yields the desired extension of $f_1$ to an exact sequence
\[
0 \to C_l \to C_{l-1} \to \cdots \to C_2 \to P_1 \xrightarrow{f_1} P_0 
\]
This finishes the proof of assertion~\ref{lem:intrinsic_reformulation_of_regular_coherent:uniform_l_greater_equal_2}.
\\[1mm]~\ref{lem:intrinsic_reformulation_of_regular_coherent:uniform_l_is_1} Suppose that
$\cala$  is $1$-uniformly regular coherent.  Consider a morphism $f \colon A_1 \to A_0$. Let
$M$ be the finitely presented $\IZ\cala$-module given by the cokernel of the $\IZ\cala$-homomorphism
$\iota(f) \colon \iota(A_1) \to \iota(A_0)$.  By assumption we can find an exact sequence
$0 \to P_1 \to P_0 \to M \to 0$ of $\IZ\cala$-modules, where $P_1$ and $P_0$ are
finitely generated projective.  We conclude from
Lemma~\ref{lem:IZ_cala_modules}~\ref{lem:IZ_cala_modules:finite_dimension} that the image
of $\iota(f)$ is finitely generated projective. 
Hence we obtain a factorization of $\iota(f)$
as a composite $\iota(f) \colon \iota(A_1) \xrightarrow{f_1'} \im(\iota(f)) \xrightarrow{f_0'} \iota(A_0)$ such that
$\im(\iota(f))$ is a finitely generated projective $\IZ\cala$-module, $f_1'$ is surjective, and
$f_0'$ is injective. We conclude from Lemma~\ref{lem:Yoneda_equivalence}
and Lemma~\ref{lem:idempotent_complete_and_Zcala_modules} that $\im(f)$ can be identified with
$\iota(B)$ for some object $B$ in $\cala$ and there are morphisms $f_1 \colon A_1 \to B$ and
$f_0 \colon B \to A_1$ such that $f_1' = \iota(f_1)$ and $f_0' = \iota(f_1)$. Moreover,
$f_1$ is surjective, $f_0$ is injective and $f = f_0 \circ f_1$.

Suppose that for every morphism $f\colon A_1 \to A_0$ we can find a factorization
$A_1 \xrightarrow{f_1} B \xrightarrow{f_0} A_0$ of $f$ such that $f_1$ is surjective and
$f_0$ is injective.  Consider any finitely presented $\IZ\cala$-module $M$. We conclude
from Lemma~\ref{lem:Yoneda_equivalence} that there is a morphism $f \colon A_1 \to A_0$ in
$\cala$ and a morphism $p \colon \iota(A_0) \to M$ of $\IZ\cala$-modules such that the
sequence $\iota(A_1) \xrightarrow{\iota(f)} \iota(A_0) \xrightarrow{p} M \to 0$ is exact.
Choose a factorization $f = f_1 \circ f_0$ such that $f_1$ is surjective and $f_0$ is
injective. Let $B$ be the domain of $f_1$. We conclude from Lemma~\ref{lem:Yoneda_equivalence}
that we obtain a short exact sequence
$0 \to \iota(B) \xrightarrow{\iota(f_1)} \iota(A_0) \xrightarrow{p} M \to 0$. This is a
$1$-dimensional finite projection $\IZ\cala$-resolution of $M$.
This finishes the proof of assertion~\ref{lem:intrinsic_reformulation_of_regular_coherent:uniform_l_is_1}.
\\[1mm]~\ref{lem:intrinsic_reformulation_of_regular_coherent:uniform_l_is_0} We first
show~\ref{lem:intrinsic_reformulation_of_regular_coherent:uniform_l_is_0:(1)}
$\implies$~\ref{lem:intrinsic_reformulation_of_regular_coherent:uniform_l_is_0:(3)}.  Consider a
morphism $f \colon A_1 \to A_0$. Let $M$ be the finitely presented $\IZ\cala$-module given
by the cokernel of $\iota(f) \colon \iota(A_1) \to \iota(A_0)$. We obtain an exact sequence
of $\IZ\cala$-modules
$\iota(A_1) \xrightarrow{\iota(f)} \iota(A_0) \xrightarrow{p} M \to 0$.  
By assumption $M$
is a finitely generated projective $\IZ\cala$-module.  
Let
$\iota(f) \colon \iota(A_1) \xrightarrow{q} \im(\iota(f)) \xrightarrow{j} \iota(A_0)$ be
  the obvious factorization of $\iota(f)$.  
  Since $M$ is projective, $\im(f)$ is a direct
  summand in $\iota(A_0)$. We conclude from Lemma~\ref{lem:Yoneda_equivalence} and
  Lemma~\ref{lem:idempotent_complete_and_Zcala_modules} that we can identify
  $\im(\iota(f))$ with $\iota(B)$ for an appropriate object $B$ in $\cala$ and can find
  morphisms $r \colon A_0 \to B$ and $s \colon B \to A_1$ in $\cala$ such that
  $\iota(r) \circ j = \id_{\iota(B)}$ and $q \circ \iota(s) = \id_{\iota(B)}$.  Define
  $g \colon A_0 \to A_1$ by $g = s \circ r$. One easily checks that
  $\iota(f) \circ \iota(g) \circ \iota(f) = \iota(f)$. Hence $f \circ g \circ f = f$.

  Next we show~\ref{lem:intrinsic_reformulation_of_regular_coherent:uniform_l_is_0:(3)}
  $\implies$~\ref{lem:intrinsic_reformulation_of_regular_coherent:uniform_l_is_0:(2)}.
  Let $f \colon A_1 \to A_0$ be a morphism in $\cala$. Choose a morphism
  $h \colon A_0 \to A_1$ with $f \circ h \circ f = f$.  Then
  $f \circ h  \colon A_0 \to A_0$ is an idempotent. Since $\cala$ is
  idempotent complete, we can find  objects $A_{-1}$ and $A_{-1}^{\perp}$ and an
  isomorphism $u \colon A_0 \xrightarrow{\cong} A_{-1} \oplus A_{-1}^{\perp} $ in $\cala$ such that
  $u \circ (\id_{A_0}  -f \circ h) \circ u^{-1}$ is $\begin{pmatrix} \id_{A_{-1}} & 0 \\ 0 & 0 \end{pmatrix}$.
  Define $g \colon A_0 \to A_{-1}$ by the composite
  $A_0 \xrightarrow{u} A_{-1} \oplus A_{-1}^{\perp} \xrightarrow{\pr_{A_{-1}}} A_{-1}$.
  One easily checks that the sequence
  $A_1 \xrightarrow{f_1} A_0 \xrightarrow{g} A_{-1} \to 0$ is exact.

  Finally we
  show~\ref{lem:intrinsic_reformulation_of_regular_coherent:uniform_l_is_0:(2)}
  $\implies$~\ref{lem:intrinsic_reformulation_of_regular_coherent:uniform_l_is_0:(1)}.
  Consider a finitely presented $\IZ\cala$-module $M$. We conclude from
  Lemma~\ref{lem:Yoneda_equivalence} and that we an find a morphism
  $f_1 \colon A_1 \to A_0$ together with an exact sequence of $\IZ\cala$-modules
  $\iota(A_1) \xrightarrow{\iota(f_1)} \iota(A_0) \xrightarrow{p} M \to 0$.  Choose a
  morphism $f_0 \colon A_0 \to A_{-1}$ such that the sequence
  $A_1 \xrightarrow{f_1} A_0 \xrightarrow{f_0} A_{-1} \to 0$ is exact in $\cala$. Then we
  obtain an exact sequence of $\IZ\cala$-modules
  $\iota(A_1) \xrightarrow{\iota(f_1)} \iota(A_0) \xrightarrow{\iota(f_0)}
  \iota(A_{-1})\to 0$ by Lemma~\ref{lem:Yoneda_equivalence}.  This implies that $M$ is
  $\IZ\cala$-isomorphic to $\iota(A_{-1})$ and hence finitely generated projective.  This
  finishes the proof of
  assertion~\ref{lem:intrinsic_reformulation_of_regular_coherent:uniform_l_is_0}.
  \\[1mm]~\ref{lem:intrinsic_reformulation_of_regular_coherent:not_uniform} This follows
  from
  assertion~\ref{lem:intrinsic_reformulation_of_regular_coherent:uniform_l_greater_equal_2}.
  This finishes the proof of Lemma~\ref{lem:intrinsic_reformulation_of_regular_coherent}.
\end{proof}

Next we deal with the property Noetherian.  Consider two morphisms $f \colon A \to B$ and
$f' \colon A' \to B$. We write $f \subseteq f'$ if there exists a morphism
$g \colon A \to A'$ with $f = f' \circ g$.  Obviously we have
\begin{multline}
  f \subseteq f'\;\Longleftrightarrow  \\
  \im(f_* \colon \mor_{\cala}(?,A) \to \mor_{\cala}(?,B))  \subseteq \im(f'_* \colon \mor_{\cala}(?,A') \to \mor_{\cala}(?,B)).
  \\ 
\label{equivalence_f_le_f'_and_images}  
\end{multline}

\begin{lemma}[Intrinsic Reformulation of Noetherian]\label{lem:intrinsic_reformulation_of_Noetherian}
  Let $\cala$ be an additive category. Then the following assertions are equivalent:

  \begin{enumerate}
  \item\label{lem:intrinsic_reformulation_of_Noetherian:plain}
   $\cala$ is Noetherian;
 \item\label{lem:intrinsic_reformulation_of_Noetherian:chain} Each object $A$ has the
   following property: Consider a sequence of morphisms $f_n \colon A_n \to A$ with fixed
   target $A$ and $f_n \subseteq f_{n+1}$ for $n \ge 0$. Then there exists $n_0$ such that
   $f_{n} \subseteq f_{n_0}$ holds for all $n \in \IN$ with $n \ge n_0$.
  \end{enumerate}
\end{lemma}
\begin{proof} Let $N$ be a finitely generated $\IZ\cala$-module and $M \subseteq N$ a
  $\IZ\cala$-submodule.  Then there exists an object $A$ in $\cala$ together with an
  epimorphism of $\IZ\cala$-module $u \colon \mor_{\cala}(?,A) \to M$, see
  Lemma~\ref{lem:compatibility_with_direct_sums}.  Obviously  $M$ is finitely generated if
  $u^{-1}(M)$ is finitely generated. Hence $\cala$ is Noetherian if and only if for any
  object $A$ in $\cala$ the $\IZ\cala$-module $\mor_{\cala}(?,A)$ is Noetherian, i.e., any
  $\IZ\cala$-submodule $M$ of $\mor_{\cala}(?,A)$ is finitely generated.  By the usual
  argument $\mor_{\cala}(?,A)$ is Noetherian if and only if for any nested sequence
  $M_0 \subseteq M_1 \subseteq M_2 \subseteq M_3 \subseteq \cdots $ of finitely generated
  $\IZ\cala$-submodules of $\mor_{\cala}(?,A)$ there exists a natural number $n_0$ such
  that $M_n \subseteq M_{n_0}$ holds for all $n \geq n_0$,  see for instance~\cite[Proposition~0.2.17 on
  page~18]{Rowen(1991)}.

  Consider finitely generated $\IZ\cala$-submodules $M_0$, $M_1$, $M_2$, $\ldots$ of $\mor_{\cala}(?,A)$.
  We can find for every natural number $n$ an object $A_n$ together with an epimorphism
  $\mor_{\IZ\cala}(?,A_n) \to M_n$, see Lemma~\ref{lem:compatibility_with_direct_sums}. 
  By the Yoneda Lemma~\ref{lem:Yoneda_lemma} there is a morphism $f_n \colon A_n \to A$ such
  that the image of $(f_n)_* \colon \mor_{\cala}(?,A_n) \to \mor_{\cala}(?,A)$ is
  $M_n$.  Hence we get $M_m \subseteq M_n$ if and and only if $(f_m) \subseteq (f_n)$  holds. Now Lemma~\ref{lem:intrinsic_reformulation_of_Noetherian} follows.
\end{proof}

\begin{lemma}\label{lem:passage_to_full_subcategrories}
  Let $\cala$ be a full additive subcategory of the additive category $\calb$.
  If $\calb$ is Noetherian, $0$-uniformly regular coherent, or $0$-uniformly regular, then
  $\cala$ has the same property.
\end{lemma}
\begin{proof}
  This follows from
  Lemma~\ref{lem:inclusion_of_full_cofinal_subcategory}~%
\ref{lem:inclusion_of_full_cofinal_subcategory:regular_coherent}, 
Lemma~\ref{lem:intrinsic_reformulation_of_regular_coherent}~%
\ref{lem:intrinsic_reformulation_of_regular_coherent:uniform_l_is_0}
  and Lemma~\ref{lem:intrinsic_reformulation_of_Noetherian}.
\end{proof}

%%%%%%%%%%%%%%%%%%%%%%%%%%%%%%%%%%%%%%%%%%%%%%%%%%%%%%%%%%%%%%%%%%%%
%%%%%%%%%%%%%%%%%%%%%%%%%%%%% Section 6 %%%%%%%%%%%%%%%%%%%%%%%%%%%%%%%%
%%%%%%%%%%%%%%%%%%%%%%%%%%%%%%%%%%%%%%%%%%%%%%%%%%%%%%%%%%%%%%%%%%%%

\typeout{------------------- Section:  Vanishing of Nil-terms -----------------}

\section{Vanishing of Nil-terms}%
\label{sec:Vanishing_of_Nil-terms}

%%%%%%%%%%%%%%%%%%%%%%%%%%%%%%%%%%%%%%%%%%%%%%%%%%%%%%%%%%%%%%%%%%%%%

\subsection{Nil-categories}%
\label{subsec:Nil-categories}

The next definition is taken from~\cite[Definition~7.1]{Lueck-Steimle(2016BHS)}.

\begin{definition}[Nilpotent morphisms and Nil-categories]%
\label{def:Nilpotent_morphisms_and_Nil-categories_Steimle}
  Let $\cala$ be an additive category and $\Phi$ be an automorphism of
  $\cala$.
  \begin{enumerate}
  \item A morphism $f\colon \Phi(A)\to A$ of $\cala$ is called \emph{$\Phi$-nilpotent}, if
    for some $n \ge 1 $, the $n$-fold composite
    \[f^{(n)}:=f \circ \Phi(f) \circ \cdots \circ \Phi^{n-1}(f) \colon \Phi^n(A)\to A.
    \]
    is trivial;
 
  \item The category $\Nil(\cala, \Phi)$ has as objects pairs $(A, \phi)$ where $\phi\colon
    \Phi(A)\to A$ is a $\Phi$-nilpotent morphism in $\cala$. A morphism from $(A, \phi)$ to
    $(B, \mu)$ is a morphism $u\colon A\to B$ in $\cala$ such that the following diagram is
    commutative:
    \[\xymatrix{
      {\Phi(A)} \ar[rr]^{\phi} \ar[d]^{\Phi(u)} && A \ar[d]^u\\
      {\Phi(B)} \ar[rr]^{\mu} && B. }\]
  \end{enumerate}
\end{definition}

The category $\Nil(\cala, \Phi)$ inherits the structure of an exact category from $\cala$,
a sequence in $\Nil(\cala,\Phi)$ is declared to be exact if the underlying sequence in
$\cala$ is split exact.

Let $\Phi \colon \cala \xrightarrow{\cong} \cala$ be an automorphism of an additive
category $\cala$.  It induces an automorphism ${\Phi^{-1}}^* \colon \MODcatr{\IZ\cala}
\xrightarrow{\cong} \MODcatr{\IZ\cala}$ of abelian categories by precomposition with
$\Phi^{-1} \colon \cala \xrightarrow{\cong} \cala$. It sends $\MODcatr{\IZ\cala}_{\fgf}$ to
itself, since ${\Phi^{-1}}^* \mor_{\cala}(?,A)$ is isomorphic to
$\mor_{\cala}(?,\Phi(A))$. Thus we obtain an automorphism of additive categories 
${\Phi^{-1}}^* \colon\MODcatr{\IZ\cala}_{\fgf} \xrightarrow{\cong} \MODcatr{\IZ\cala}_{\fgf}$

\begin{lemma}\label{lem:Nil_and_Yoneda}
There is an equivalence of exact categories
\[
\iota \colon \Nil(\cala;\Phi) \xrightarrow{\simeq} \Nil(\MODcatr{\IZ\cala}_{\fgf};{\Phi^{-1}}^*)
\]
\end{lemma}
\begin{proof}
  The desired functors $\iota$ sends an object $(A,f)$ in $\Nil(\cala;\phi)$ given by a
  morphism $f \colon \Phi(A) \to A$ to the object in
  $\Nil(\MODcatr{\IZ\cala}_{\fgf};{\Phi^{-1}}^*)$ given by the composite
\[
{\Phi^{-1}}^*\mor_{\cala}(?,A) = \mor_{\cala}(\Phi^{-1}(?),A) \xrightarrow{\Phi} \mor_{\cala}(?,\Phi(A)) 
\xrightarrow{\mor_{\cala}(?,f)} \mor_{\cala}(?,A).
\] 
A morphism $u \colon (A,f) \to (A',f')$ in $\Nil(\cala;\Phi)$, which given by a morphism
$u \colon A \to A'$ in $\cala$ satisfying $f' \circ \Phi(u) = u \circ f$, is sent to the
morphism in $\Nil(\MODcatr{\IZ\cala}_{\fgf};{\Phi^{-1}}^*)$ given by the morphism
$u_*  \colon \mor_{\cala}(?,A) \to \mor_{\cala}(?,A')$. It defines indeed a
morphism from $\iota(A,f)$ to $\iota(A',f')$ by the commutativity of the following diagram
\[
\xymatrix@!C=11em{
\mor_{\cala}(\Phi^{-1}(?),A)  \ar[r]^-{\Phi} \ar[d]^{\mor_{\cala}(\Phi^{-1}(?),u)}
&
\mor_{\cala}(?,\Phi(A)) \ar[r]^-{\mor_{\cala}(?,f)} \ar[d]^{\mor_{\cala}(?,\Phi(u))}
&
\mor_{\cala}(?,A) \ar[d]^{\mor_{\cala}(?,u)}
\\
\mor_{\cala}(\Phi^{-1}(?),A')  \ar[r]^-{\Phi} 
&
\mor_{\cala}(?,\Phi(A')) \ar[r]^-{\mor_{\cala}(?,f')} 
&
\mor_{\cala}(?,A')
}
\]
It is an equivalence of additive categories by Lemma~\ref{lem:Yoneda_equivalence}.
\end{proof}

%%%%%%%%%%%%%%%%%%%%%%%%%%%%%%%%%%%%%%%%%%%%%%%%%%%%%%%%%%%%%%%%%%%%%

\subsection{Connective $K$-theory}%
\label{subsec:Connective-K-theory}

\begin{lemma}%
\label{lem:The_connective_K-theory_of_NIL-categories_for_additive_categories}
Let $\cala$ be an idempotent complete additive category. Suppose that $\cala$ is regular
coherent.  Let $\Phi \colon \cala \xrightarrow{\cong} \cala$ be any automorphism of
additive categories.  Denote by $J \colon \cala \to \Nil(\cala,\phi)$ the inclusion
sending an object $A$ to the object $(A,0)$.

Then the induced map on connective $K$-theory
\[
\bfK(J) \colon \bfK(\cala) \to \bfK(\Nil(\cala,\Phi))
\]
is a weak homotopy equivalence.
\end{lemma}
\begin{proof}
   We abbreviate $\Psi = {\Phi^{-1}}^*$ We have the following commutative diagram
  \[
  \xymatrix{\cala \ar[r]^-J \ar[d]_{\iota}
    &
   \Nil(\cala,\Phi) \ar[d]^{\iota}
   \\
   \MODcatr{\IZ\cala}_{\fgf} \ar[r]_-J
   &
   \Nil(\MODcatr{\IZ\cala}_{\fgf},\Psi)
   }
   \]
   where the  vertical arrows are equivalences of exact categories given by Yoneda embeddings,
   see Lemma~\ref{lem:Yoneda_equivalence}  and Lemma~\ref{lem:Nil_and_Yoneda},
   and the lower horizontal arrow is the obvious analogue of the upper horizontal  arrow.
    Hence it suffices to show that the map
   \[
  \bfK(J) \colon \bfK(\MODcatr{\IZ\cala}_{\fgf}) \to \bfK(\Nil(\MODcatr{\IZ\cala}_{\fgf},\Psi))
  \]
  is a weak homotopy equivalence. 

  Denote by $\MODcatr{\IZ\cala}_{\FL}$ the full subcategory of
  $\MODcatr{\IZ\cala}$ consisting of $\IZ\cala$-modules which are of type $\FL$, i.e.,
  possess a finite dimensional resolution by finitely generated free $\IZ\cala$-modules. 
  
  Consider the following commutative diagram 
  \begin{equation}
   \xymatrix{\bfK(\MODcatr{\IZ\cala}_{\fgf}) \ar[r]\ar[d]
   &
   \bfK(\Nil(\MODcatr{\IZ\cala}_{\fgf},\Psi)) \ar[d]
   \\
   \bfK(\MODcatr{\IZ\cala}_{\FL}) \ar[r]
   &
   \bfK(\Nil(\MODcatr{\IZ\cala}_{\FL},\Psi)) 
   }
  \label{Costa}
 \end{equation}
 where all arrows are induced by the obvious inclusions of categories.
  
 The left vertical arrow in the diagram~\eqref{Costa} is a weak homotopy equivalence by
 the Resolution Theorem, see~\cite[Theorem~4.6 on page~41]{Srinivas(1991)}.

 Next we show that the lower horizontal arrow in the diagram~\eqref{Costa} is a weak
 homotopy equivalence.  Consider an object $(M,f)$ in
 $\Nil(\MODcatr{\IZ\cala}_{\FL},\Psi)$.
  
  Recall that nilpotent means that for some  natural number $n\ge  0$ the composite
  \[
  f^{(n)} \colon \Psi^n(M) \xrightarrow{\Psi^{n-1}(f)} \Psi^{n-1}(M) 
  \xrightarrow{\Psi^{n-2}(f)} \cdots \xrightarrow{\Psi(f)} \Psi(M) \xrightarrow{f} M
  \]
  is trivial.  We get a filtration of (M,f) by subobjects
  \begin{multline*}
   (M,f) \supseteq (\im(f),f|_{\Psi(\im(f))}) \supseteq  (\im(f^{(2)}),f|_{\Psi(\im(f^{(2)}))}) 
   \\
    \supseteq \cdots \supseteq (\im(f^{(n-1)}),f|_{\Psi(\im(f^{(n-1)}))})  
    \\
     \supseteq (\im(f^{(n)}),f|_{\Psi(\im(f^{(n)}))}) = (\{0\},\id_{\{0\}}).
\end{multline*}
Here we consider $\Psi(\im(f^{(i)}))$ as a $\IZ\cala$-submodule of $\Psi(M)$ by the injective
map $\Psi(\im(f^{(i)})) \to \Psi(M)$, which is obtained by applying $\Psi$ to the inclusion
$\im(f^{(i)}) \to M$.  We get exact sequences of $\IZ\cala$-modules
\begin{eqnarray*}
&0 \to \im(f^{(i)}) \to M \to M/\im(f^{(i)}) \to 0; &
\\
& 0 \to \im(f^{i+1}) \to \im(f^{i}) \to \im(f^{i})/\im(f^{i+1}) \to 0. &
\end{eqnarray*}
Since $M$ is finitely presented and $\im(f^{i})$ is finitely generated, $M/\im(f^{(i)})$
is finitely presented.  Since $\cala$ is regular coherent and idempotent complete by assumption, $M$ and
$M/\im(f^{(i)})$ for all $i$ are of type $\FL$.  We conclude by induction over
$i = 0,1, \ldots $ from Lemma~\ref{lem:IZ_cala_modules}~%
\ref{lem:IZ_cala_modules:exact_sequences_FL} that $\im(f^{(i)})$
and $\im(f^{(i)})/\im(f^{(i+1)})$ belong to $\MODcatr{\IZ\cala}_{\FL}$ again. The
quotient of $(\im(f^{(i)}),f|_{\Psi(\im(f^{(i)}))})$ by
$(\im(f^{(i+1)}),f|_{\Psi(\im(f^{(i+1)}))})$ is given by
$(\im(f^{(i)})/\im(f^{(i+1)}),0)$, and hence belongs to $\MODcatr{\IZ\cala}_{\FL}$ for
all $i$. Now the lower horizontal arrow in diagram~\eqref{Costa} is a weak homotopy
equivalence by the Devissage Theorem, see~\cite[Theorem~4.8 on page~42]{Srinivas(1991)}.

Next we show that the right vertical arrow in the diagram~\eqref{Costa} induces split
injections on homotopy groups.  For this purpose we consider the following commutative
diagram of exact categories
\[\xymatrix@!C=15em{
\Nil(\MODcatr{\IZ\cala}_{\fgf}, \Psi) \ar[r]^-{I_1} \ar[rd]^{I_2} \ar[rdd]^{I_3}
&
\HNil(\Chcat(\MODcatr{\IZ\cala}_{\fgf}), \Psi)
\\
 & 
\HNil(\Chcat_{\res}(\MODcatr{\IZ\cala}_{\fgf}), \Psi) \ar[u]_{I_4} \ar[d]^{H_0}
\\
& \Nil(\MODcatr{\IZ\cala}_{\FL},\Psi)
}
\]
The category $\HNil(\Chcat(\MODcatr{\IZ\cala}_{\fgf}), \Psi)$ is given by
finite-dimensional chain complexes $C_*$ over $\MODcatr{\IZ\cala}_{\fgf}$ (with $C_i = 0$
for $i \le -1$) together with chain maps $\phi \colon C_* \to C_*$, which are homotopy
nilpotent, and $\HNil(\Chcat_{\res}(\MODcatr{\IZ\cala}_{\fgf}), \Psi)$ is the full
subcategory of $\HNil(\Chcat(\MODcatr{\IZ\cala}_{\fgf}), \Psi)$ consisting of those chain
complexes, for which $H_i(C_*) = 0$ for $i \ge 1$. The maps $I_k$ for $k = 1,2,3,4$ are the
obvious inclusions, the functor $H_0$ is given by taking the zeroth homology group. The
upper horizontal arrow induces a weak homotopy equivalence on connective $K$-theory
by~\cite[page~173]{Lueck-Steimle(2016BHS)}. The functor $H_0$ induces a weak homotopy
equivalence on connective $K$-theory by the Approximation Theorem of Waldhausen, see for
instance~\cite[Theorem~4.18]{Lueck-Steimle(2016BHS)}. Hence the map induced by $I_3$ on
connective $K$-theory, which is the right vertical arrow in the diagram~\eqref{Costa},
induces split injections on homotopy groups.

We conclude that all arrows appearing in the diagram~\eqref{Costa} induce weak homotopy
equivalences on connective algebraic $K$-theory.  This finishes the proof of
Lemma~\ref{lem:The_connective_K-theory_of_NIL-categories_for_additive_categories}.
\end{proof}

\begin{theorem}[The connective $K$-theory of additive categories]%
\label{the:The_connective_K-theory_of_additive_categories}
Let $\cala$ be an  additive category, which is idempotent complete and regular coherent.
Consider any  automorphism $\Phi \colon \cala \xrightarrow{\cong}  \cala$ of additive categories.

Then we get  a map of connective spectra
\[
\bfa  \colon \bfT_{\bfK(\Phi^{-1})}  \to  \bfK(\cala_{\Phi}[t,t^{-1}])
\]
such that $\pi_n(\bfa)$ is bijective for $n \ge 1$.
\end{theorem}
\begin{proof} This follows from Theorem~\ref{the:BHS_decomposition_for_connective_K-theory},
since Lemma~\ref{lem:The_connective_K-theory_of_NIL-categories_for_additive_categories}
implies $\pi_n(\bfE(R,\Phi)) = 0$ for $n \ge 0$ and hence 
$\pi_n(\bfNK(\cala_{\Phi}[t]))  = \pi_n(\bfNK(\cala_{\Phi}[t^{-1}])) =0$ for all $n \ge 1$.
\end{proof}

We will need later the following consequence of
Lemma~\ref{lem:The_connective_K-theory_of_NIL-categories_for_additive_categories},
where we can drop the assumption that $\cala$ is idempotent complete.

\begin{lemma}%
\label{lem:The_connective_K-theory_of_NIL-categories_for_additive_categories_without_idempotent_complete}
Let $\cala$ be an additive category. Suppose that $\cala$ is regular
coherent.  Let $\Phi \colon \cala \xrightarrow{\cong} \cala$ be any automorphism of
additive categories.  Denote by $J \colon \cala \to \Nil(\cala,\phi)$ the inclusion
sending an object $A$ to the object $(A,0)$.

Then the induced map 
\[
\pi_n(\bfK(J)) \colon \pi_n(\bfK(\cala)) \to \pi_n(\bfK(\Nil(\cala,\Phi)))
\]
is bijective for $n \ge 1$.
\end{lemma}
\begin{proof}
We have the obvious commutative diagram coming from the inclusion $\cala \to \Idem(\cala)$.
\[\xymatrix{\pi_n(\bfK(\cala)) \ar[r]  \ar[d]
  &
  \pi_n(\bfK(\Nil(\cala,\Phi))) \ar[d]
  \\
  \pi_n(\bfK(\Idem(\cala))) \ar[r]  
  &
  \pi_n(\bfK(\Nil(\Idem(\cala)),\Idem(\Phi))).
}
\]
The left vertical arrow is bijective for $n \ge 1$ by
Lemma~\ref{K_and_cofinal}~\ref{K_and_cofinal:connective}.
The lower horizontal arrow is bijective for $n \ge 1$ by
Lemma~\ref{lem:The_connective_K-theory_of_NIL-categories_for_additive_categories}, since
$\Idem(\cala)$ is regular coherent by
Lemma~\ref{lem:inclusion_of_full_cofinal_subcategory}~%
\ref{lem:inclusion_of_full_cofinal_subcategory:regular_coherent}.  Hence we have to
show that the right vertical arrow is bijective for $n \ge 1$. For this purpose it
suffices to show because of Lemma~\ref{K_and_cofinal}~\ref{K_and_cofinal:connective} 
that $\Nil(\cala,\Phi)$ is a cofinal full subcategory of
$\Nil(\Idem(\cala),\Idem(\Phi))$.  This follows from the fact that $\cala$ is a cofinal
full subcategory of $\Idem(\cala)$.  
\end{proof}

%%%%%%%%%%%%%%%%%%%%%%%%%%%%%%%%%%%%%%%%%%%%%%%%%%%%%%%%%%%%%%%%%%%%%

\subsection{Non-connective $K$-theory}%
\label{subsec:Non-connective-K-theory}

In the sequel define $\cala[\IZ^m]$ 
inductively over $m$ by $\cala[\IZ^m] := \cala[\IZ^{m-1}]_{\id}[t,t^{-1}]$,
where $\cala[\IZ^{m-1}]_{\id}[t,t^{-1}]$ is the (untwisted) finite Laurent category associated to $\cala[\IZ^{m-1}]$ 
and the automorphism given by the identity, see Subsection~\ref{subsec:Twisted_finite_Laurent_category}.

\begin{lemma}%
\label{lem:The_non-connective_K-theory_of_NIL-categories_for_additive_categories}
Let $\cala$ be an  additive category.
Suppose that $\cala[\IZ^m]$  is regular coherent for every $m \ge 0$.
Consider any  automorphism $\Phi \colon \cala \xrightarrow{\cong}  \cala$ of additive categories.
Denote by $J \colon \cala \to \Nil(\cala,\Phi)$ the inclusion sending an object $A$ to the object $(A,0)$.

Then the induced map on non-connective $K$-theory
\[
\bfKinfty(J) \colon \bfKinfty(\cala) \to \bfKNilinfty(\Nil(\cala,\Phi))
\]
is a weak homotopy equivalence.
\end{lemma} 
\begin{proof} 
Fix $n \in \IZ$. We have to show  that $\pi_n(\bfKinfty(J))$ is bijective.
This follows from  Lemma~\ref{lem:The_connective_K-theory_of_NIL-categories_for_additive_categories_without_idempotent_complete}
for $n \ge 1$ and is proved in general as follows.

From the definitions and the construction in~\cite[Section~6]{Lueck-Steimle(2014delooping)},
one obtains for every $n \in \IZ$ a commutative diagram
\[
  \xymatrix{\pi_{n}(\bfKinfty(\cala)) \ar[d]_{i} \ar[r] 
  & 
  \pi_{n}(\bfKNilinfty(\cala,\Phi))\ar[d]^{j}
   \\
   \pi_{n+1}(\bfKinfty(\cala[\IZ])) \ar[d]_r \ar[r]
   &
   \pi_{n+1}(\bfKNilinfty(\cala[\IZ],\Phi[\IZ]))\ar[d]^{s}
  \\
  \pi_{n}(\bfKinfty(\cala)) \ar[r] 
   & 
   \pi_{n}(\bfKNilinfty(\cala,\Phi))
}
\]
where $r \circ i = \id$ and $j \circ s = \id$ and these maps are part of the corresponding
(untwisted) Bass-Heller-Swan decompositions.  Iterating this, one obtains for every $m \ge 0$ a commutative diagram
\[
  \xymatrix{\pi_{n}(\bfKinfty(\cala)) \ar[d]_{i} \ar[r] 
  & 
  \pi_{n}(\bfKNilinfty(\cala,\Phi))\ar[d]^{j}
   \\
   \pi_{n+m}(\bfKinfty(\cala[\IZ^m])) \ar[d]_r \ar[r]
   &
   \pi_{n+m}(\bfKNilinfty(\cala[\IZ^m],\Phi[\IZ^m]))\ar[d]^{s}
  \\
  \pi_{n}(\bfKinfty(\cala)) \ar[r] 
   & 
   \pi_{n}(\bfKNilinfty(\cala,\Phi))
}
\]
where $r \circ i = \id$ and $j \circ s = \id$ holds. Now choose $m$ such that $n + m \ge 1$ holds.
Then the middle horizontal arrow can be identified by construction with its connective  version
\[
\pi_{n+m}(\bfK(\cala[\IZ^m])) \to \pi_{n+m}(\bfK(\Nil(\cala[\IZ^m],\Phi[\IZ^m]))).
\]  
Since this map is a bijection by
Lemma~\ref{lem:The_connective_K-theory_of_NIL-categories_for_additive_categories}
the upper horizontal arrow is a retract of an isomorphism and hence itself an  isomorphism.
\end{proof}

\begin{theorem}[The non-connective $K$-theory of additive categories]%
\label{the:The_non_connective_K-theory_of_additive_categories}
Let $\cala$ be an  additive category.
Suppose that $\cala[\IZ^m]$ 
is regular coherent for every $m \ge 0$.
Consider any  automorphism $\Phi \colon \cala \xrightarrow{\cong}  \cala$ of additive categories.

Then we get  a weak homotopy equivalence of non-connective spectra
\[
\bfa^{\infty}  \colon \bfT_{\bfKinfty(\Phi^{-1})}  \xrightarrow{\simeq}  \bfKinfty(\cala_{\Phi}[t,t^{-1}]).
\]
\end{theorem}
\begin{proof} This follows from Theorem~\ref{the:BHS_decomposition_for_non-connective_K-theory},
since Lemma~\ref{lem:The_non-connective_K-theory_of_NIL-categories_for_additive_categories}
implies $\pi_n(\bfE^{\infty}(R,\Phi)) = 0$ and hence 
$\pi_n(\bfNKinfty(R_{\Phi}[t]))  = \pi_n(\bfNKinfty(\cala_{\Phi}[t^{-1}])) =0$ for all $n \in \IZ$.
\end{proof}

%%%%%%%%%%%%%%%%%%%%%%%%%%%%%%%%%%%%%%%%%%%%%%%%%%%%%%%%%%%%%%%%%%%%
%%%%%%%%%%%%%%%%%%%%%%%%%% Section 7 %%%%%%%%%%%%%%%%%%%%%%%%%%%%%%%%
%%%%%%%%%%%%%%%%%%%%%%%%%%%%%%%%%%%%%%%%%%%%%%%%%%%%%%%%%%%%%%%%%%%%

\typeout{------------------- Section: Noetherian additive categories  -----------------}

\section{Noetherian additive categories}%
\label{sec:Noetherian_additive_categories}

  \begin{theorem}[Hilbert Basis Theorem for additive categories]%
~\label{the:Hilbert_Basis_Theorem_for_additive_categories}

    Consider an additive category $\cala$ together with an automorphism
$\Phi \colon \cala \xrightarrow{\cong} \cala$.

    \begin{enumerate}
    \item\label{the:Hilbert_Basis_Theorem_for_additive_categories:cala_to_cala[t]}
    If the additive category $\cala$ is Noetherian, then the additive
    categories $\cala_{\Phi}[t]$, $\cala_{\Phi}[t^{-1}]$, and $\cala_{\Phi}[t,t^{-1}]$ are
    Noetherian; 

    \item\label{the:Hilbert_Basis_Theorem_for_additive_categories:cala[t]_to_cala[t,t_upper_-1]}
    If the additive category $\cala_{\Phi}[t]$ is Noetherian, then the additive
    category  $\cala_{\Phi}[t,t^{-1}]$ is  Noetherian. 
  \end{enumerate}
\end{theorem}

\begin{proof}~\ref{the:Hilbert_Basis_Theorem_for_additive_categories:cala_to_cala[t]} We
  only treat $\cala_{\Phi}[t]$, the proof for $\cala_{\Phi}[t^{-1}]$ is analogous. For
  $\cala_\Phi[t,t^{-1}]$ the claim will follow then
  from~\ref{the:Hilbert_Basis_Theorem_for_additive_categories:cala[t]_to_cala[t,t_upper_-1]}.
    
  We  translate the usual proof of the Hilbert Basis Theorem for rings to additive categories. 
  Consider a finitely generated $\IZ\cala_{\Phi}[t]$-module $N$ and a
  $\IZ\cala_{\Phi}[t]$-submodule $M \subseteq N$.  We have to show that $M$ is finitely
  generated.  Lemma~\ref{lem:IZ_cala_modules}~\ref{lem:IZ_cala_modules:additive} implies
  that there is an epimorphisms $\phi \colon \mor_{\cala_{\Phi}[t]}(?,A) \to N$ for some
  object $A$.  If $\phi^{-1}(M)$ is finitely generated, then $M$ is finitely generated,
  since $f$ induces an epimorphism $f^{-1}(M) \to M$.  Hence we can assume without loss of
  generality $N = \mor_{\cala_{\Phi}[t]}(?,A)$.

  Fix an object $Z$ in $\cala$.  Consider a non-trivial element $f \colon Z\to A$ in
  $N(Z)$.  We can write it as a finite sum $\sum_{k= 0}^{d(f)} f_k \cdot t^k$, where
  $f_k \colon \Phi^k(Z) \to A$ is a morphism in $\cala$ and $f_{d(f)} \not= 0$.  We call
  the natural number $d(f)$ the \emph{degree} of $f$ and
  $R(f) = f_{d(f)} \colon \Phi^{d(f)}(Z) \to A$ the \emph{leading coefficient} of $f$.  We
  put $d(0 \colon Z \to A) = -\infty$ and $R(0 \colon Z \to A) = 0$.
  
  We define now $I_d$ as the $\IZ\cala$-submodule of $\mor_\cala(?,A)$ that is generated by all
  $R(f)$ with $f \in M(Z)$ and $d(f) = d$ for some object $Z$ from $\cala$.
  We have $I_0 \subseteq I_1 \subseteq I_2 \subseteq \cdots$ and
  define $I$  to be the  $\IZ\cala$-submodule $\bigcup_{d \ge 0} I_d$.  As $\cala$ is by
  assumption Noetherian, $I$ and all the $I_d$ are finitely generated.  Therefore we find a
  finite collection of morphisms
  $f_i \in M(Z_i) \subseteq \mor_{\cala_\Phi[t]}(Z_i,\cala)$ such that the $R(f_i)$
  generate $I$. We abbreviate $d_i := d(f_i)$.  Since each $f_i$ lies in of the $I_d$-s,
  we can find a natural number $d_0$ such that 
  $I = I_{d_0} = I_d$ holds for $d \ge d_0$.  Hence we can also arrange for the
  $f_i$ to have the following property: for each $d$ the $R(f_i)$ with $d_i \leq d$
  generate $I_d$.  We record that
  $R(f_i) \in I_{d_i}(\Phi^{d_i}(Z_i)) \subseteq \mor_\cala(\Phi^{d_i}(Z_i),A)$.

  We will show that the $f_i$ generate $M$.
  Let $f \in M(Z)$, $f \neq 0$. We abbreviate $d := d(f)$.
  We have $R(f) \in I_{d}(\Phi^{d} (Z)) \subseteq \mor_\cala(\Phi^{d}(Z),A)$.
  We can write 
  \begin{equation*} 
    R(f) = \sum_{i} R(f_i) \circ \varphi_i
  \end{equation*}
   with $\varphi_i \in \mor_\cala(\Phi^{d}(Z),\Phi^{d_i}(Z_i)$ and $\varphi_i = 0$ whenever $d(f_i) > d(f)$.
  Set 
  \begin{equation*}
  	\tilde \varphi_i := {\Phi}^{-d_i}(\varphi_i) \cdot t^{d-d_i} \in \mor_{\cala_\phi[t]}(Z,Z_i).
  \end{equation*} 
  Then  
  \begin{equation*}
    R\Big(\sum_i f_i \circ \tilde \varphi_i \Big)
    = \sum_i R(f_i) \circ \Phi^{d_i}(\Phi^{-d_i}(\varphi_i)) = \sum_i R(f_i) \circ \varphi_i.
  \end{equation*}
  Thus $d(f - \sum_i f_i \circ \tilde \varphi) < d$.  Now we can repeat the argument for
  $f' := f - \sum_i f_i \circ \tilde \varphi$.  By induction on $d(f)$ we now find that
  $f$ belongs to the submodule of $\mor_{\cala_\phi[t](?,A)}$ generated by the $f_i$.
  Hence $M$ is a finitely generated $\IZ\cala_{\Phi}[t]$-module.
  \\[1mm]\ref{the:Hilbert_Basis_Theorem_for_additive_categories:cala[t]_to_cala[t,t_upper_-1]}
  It suffices to show for a $\cala_{\Phi}[t,t^{-1}]$-submodule $M$ of
  $\mor_{\cala_{\Phi}[t,t^{-1}]}(?,A)$ that $M$ is finitely generated as
  $\cala_{\Phi}[t,t^{-1}]$-module.  For $Z \in \cala$ we have
  $\mor_{\cala_\Phi[t]}(Z,A) \subseteq \mor_{\cala_\Phi[t,t^{-1}]}(Z,A)$.  We define the
  $\IZ\cala_{\Phi}[t]$-module $M'$ by
   \begin{equation*}
   	  M'(Z) := M(Z) \cap \mor_{\cala_\Phi[t]}(Z,A).
   \end{equation*}
   Since $\cala_{\Phi}[t]$ is Noetherian, we find a finite collection of morphisms
   $f_i \in M'(Z_i) \subseteq M(Z_i) \subseteq \mor_{\cala_\Phi[t,t^{-1}]}(Z_i,A)$ that
   generate $M'$ as an $\IZ\cala_{\Phi}[t]$-module.  We claim that the $f_i$ also generate
   $M$ as an $\IZ\cala_{\Phi}[t,t^{-1}]$-module.  Let
   $f \in M(Z) \subseteq \mor_{\cala_{\Phi}[t,t^{-1}]}(Z,A)$.  For $d \geq 0$ we have
   $\id_Z \cdot t^d \in \mor_{\cala_\Phi[t,t^{-1}]}(\Phi^{-d}(Z),Z)$.  For sufficiently
   large $d$ we have
   $f \circ (\id_Z \cdot t^d) \in \mor_{\cala_{\Phi}[t]}(\Phi^{-d}(Z),A) \cap
   M(\Phi^{-d}(Z)) = M'(\Phi^{-d}(Z))$.  Thus $f \circ (\id_Z \cdot t^d)$ belongs to the
   $\IZ\cala_{\Phi}[t,t^{-1}]$-submodule of $M$ generated by the $f_i$.  As
   $(\id_Z \cdot t^d)$ is an isomorphism in $\cala_{\Phi}[t,t^{-1}]$, $f$ also belongs to
   the $\IZ\cala_{\Phi}[t,t^{-1}]$-submodule of $M$ generated by the $f_i$.
 \end{proof}

%%%%%%%%%%%%%%%%%%%%%%%%%%%%%%%%%%%%%%%%%%%%%%%%%%%%%%%%%%%%%%%%%%%%
%%%%%%%%%%%%%%%%%%%%%%%%%% Section 8 %%%%%%%%%%%%%%%%%%%%%%%%%%%%%%%%
%%%%%%%%%%%%%%%%%%%%%%%%%%%%%%%%%%%%%%%%%%%%%%%%%%%%%%%%%%%%%%%%%%%%

\typeout{------------------- Section:  Additive categories with finite global dimension -----------------}

\section{Additive categories with finite global dimension}%
\label{sec:Additive_categories_with_finite_global_dimension}

Let $\Phi \colon \cala \to \cala$ be an automorphism of the additive category $\cala$.
Let $\overline{\Phi} \colon \cala_{\phi}[t] \to \cala_{\Phi}[t]$ be the automorphism of
additive categories induced by $\Phi$, which sends the morphisms
$\sum_{k = 0}^{\infty} f_k \cdot t^k \colon A \to B$ to the morphism
$\sum_{k = 0}^{\infty} \Phi(f_k) \cdot t^k \colon \Phi(A) \to \Phi(B)$.  Denote
by $i \colon \cala \to \cala_{\Phi}[t]$ the inclusion sending $f \colon A \to B$ to
$(f \cdot t^0) \colon A \to B$.  Obviously we have
$\overline{\Phi} \circ i = i \circ \Phi$.

%%%%%%%%%%%%%%%%%%%%%%%%%%%%%%%%%%%%%%%%%%%%%%%%%%%%%%%%%%%

\subsection{The characteristic sequence}%
\label{subsec:The_characteristic_sequence}

Consider a  $\IZ\cala_{\Phi}[t]$-module $M$. Let 
\[
e \colon i_*i^*M \to M
\]
be the $\IZ\cala_{\Phi}[t]$-morphism, which is the adjoint of the $\IZ\cala$-homomorphism
$\id\colon i^*M \to i^*M$ under the adjunction~\eqref{adjunction_(F_ast,F_upper_ast)}.  We
get for every object $A$ in $\cala$ a morphism 
$\id_{\Phi(A)} \cdot t \colon A \to \Phi(A)$ in $\cala_{\Phi}[t]$.  It induces a
$\IZ\cala_{\Phi}[t]$-morphism  $M(\id_{\phi(A)} \cdot t) \colon M(\Phi(A)) \to M(A)$. Since for a morphisms
$u \colon A \to B$ in $\cala$ we have
\begin{multline*}
(\id_{\phi(B)} \cdot t) \circ i(u) =(\id_{\phi(B)} \cdot t) \circ (u\cdot t^0)  
=  \Phi(u) \cdot t 
\\
= (\Phi(u) \cdot t^0) \circ (\id_{\Phi(A)} \cdot t)  = i(\Phi(u)) \circ (\id_{\Phi(A)} \cdot t), 
\end{multline*}
we obtain a morphism of $\IZ\cala$-modules 
\begin{equation} 
\alpha' \colon \Phi^*i^*M \xrightarrow{\cong}  i^*M.
\label{iso_alpha_prime}
\end{equation}
By applying $i_*$ we
obtain a morphism  of $\IZ\cala_{\Phi}[t]$-modules
\[
\alpha \colon i_*\Phi^*i^*M \to i_*i^*M.
\]
The morphism $\id_{\Phi(A)} \cdot t \colon A \to \Phi(A)$ in $\cala_{\Phi}[t]$ induces also a
$\IZ$-map 
\[
\beta(A) \colon  i_*\Phi^*i^*M(A) = i_*i^*M(\Phi(A)) \to i_*i^*M(A).
\]
Since for any morphism 
$v = \sum_{k = 0}^{\infty} f_k \cdot t^k \colon A \to B$
in $\cala_{\Phi}[t]$ we have
\begin{eqnarray*}
(\id_{\Phi(B)} \cdot t) \circ v
& = & 
(\id_{\Phi(B)} \cdot t)  \circ \biggl(\sum_{k = 0}^{\infty} f_k \cdot t^k\biggr)
\\
& = &
\sum_{k = 0}^{\infty}  (\id_{\Phi(B)} \cdot t)  \circ (f_k \cdot t^k)
\\
& = &
\sum_{k = 0}^{\infty}  \Phi(f_k) \cdot t^{k+1}
\\
& = &
\sum_{k = 0}^{\infty}  (\Phi(f_k) \cdot t^k) \circ (\id_{\Phi(A)} \cdot t)
\\
& = &
\overline{\Phi}\biggl(\sum_{k = 0}^{\infty}  f_k \cdot t^k\biggr)  \circ (\id_{\Phi(A)} \cdot t)
\\
& = &
\overline{\Phi}(v)   \circ (\id_{\Phi(A)} \cdot t),
\end{eqnarray*}
we get a $\IZ\cala_{\Phi}[t]$-homomorphism
denoted by
\[
\beta \colon i_*\Phi^*i^*M \to i_*i^*M.
\]
Define the so called \emph{characteristic sequence} of
$\IZ\cala_{\phi}[t]$-modules by
\begin{equation}
0 \to i_*\Phi^*i^*M \xrightarrow{\alpha - \beta}  i_*i^*M \xrightarrow{e}  M \to  0.
\label{characteristic_sequence}
\end{equation}
Given an object $A \in \cala$,  $(\alpha -\beta)(A)$ is explicitly given by
\begin{multline*}
 M(\Phi(?)) \otimes_{\IZ\cala} \mor_{\cala_{\phi}[t]}(A,{?})  \to  M(?) \otimes_{\IZ\cala} \mor_{\cala_{\phi}[t]}(A,{?}) , \\
x \otimes (f_k \cdot t^k \colon A \to {?}) 
\mapsto 
 M(\id_{\Phi(?)} \cdot t \colon {?} \to \Phi(?))(x) \otimes (f_k \cdot t^k \colon A \to {?}) 
\\- x \otimes (\Phi(f_k) \cdot t^{k+1} \colon A \to \Phi(?)),
\end{multline*}
and $e(A)$ is explicitly given by
\[
M(?) \otimes_{\IZ\cala}  \mor_{\cala_{\phi}[t]}(A,?)  \to M(A), \quad 
x \otimes  (u\colon A \to {?})  \mapsto M(u)(x) = xu.
\]

\begin{lemma}\label{lem:characteristic_sequence}
The characteristic sequence~\eqref{characteristic_sequence}
is  natural in $M$ and exact.
\end{lemma}
\begin{proof} It is obviously natural in $M$. 
To prove exactness, it suffices to prove the exactness of the sequence of $\IZ\cala$-modules
\begin{equation}
0 \to i^*i_*\Phi^*i^*M \xrightarrow{\alpha - \beta}  i^*i_*i^*M \xrightarrow{e}  i^*M \to  0.
\label{characteristic_sequence_after_i_upper_ast}
\end{equation}
Let $N$ be a $\IZ\cala$-module. We obtain a $\IZ\cala$-isomorphism
\begin{equation}
S(N) \colon \bigoplus_{k = 0}^{\infty} \Phi^{k}(N) \xrightarrow{\cong} i^*i_*N,
\label{I_upper_ast_circ_i_ast_is_bigoplus}
\end{equation}
which is defined for an object $A$ in $\cala$ by the $\IZ$-isomorphism
\[
S(N)(A) \colon \bigoplus_{k = 0}^{\infty} N(\Phi^k(A)) 
\xrightarrow{\cong} i^*i_*N(A) =  N(?) \otimes_{\IZ\cala} \mor_{\cala_{\Phi}{t}}(i(A),?) 
\]
sending $(x_k)_{k \ge 0}$ to $\sum_{k=0}^{\infty} x \otimes \bigl(\id_{\Phi^k(A)} \cdot t^k \colon A  \to \Phi^k(A)\bigr) $.
The inverse of $S(N)(A)$  
sends $y \otimes \biggl(\sum_{k = 0}^{\infty} f_k \cdot t^k \colon A \to {?}\bigr)$ to
$\sum_{k = 0}^{\infty} N(f_k)(y)$. 
Applying this to $N = i^*\overline{\Phi}^*M = \Phi^*i^*M$ and $N = i^* M$, we get identifications
\begin{eqnarray*} 
i^*i_*i^*\overline{\Phi}^*M = \bigoplus_{k = 1} ^{\infty} (\Phi^{k})^*i^*M;
\\
i^*i_*i^*M  = \bigoplus_{k = 0} ^{\infty} (\Phi^{k})^*i^*M.
\end{eqnarray*} 
Consider natural numbers $m$ and $n$ with $m \ge n$. For an object $A$ let the map
$s_{m,n}(A) \colon (\overline{\Phi}^m)^*M(A) \to (\overline{\Phi}^n)^*M(A)$ be the map
obtained by applying $M$ to the morphism
$\id_{\Phi^m(A)} \cdot t^{m-n} \colon \Phi^n(A) \to \Phi^m(A)$ in $\cala_{\Phi}[t]$. This
yields a map of $\IZ\cala$-modules
\[
s_{m,n} \colon (\Phi^m)^*i^*M \to (\Phi^n)^*i^*M.
\]
Under these identifications the $\IZ\cala$-sequence~\eqref{characteristic_sequence_after_i_upper_ast}
becomes the sequence 
\begin{multline*}
0 
\to 
\bigoplus_{m= 1}^{\infty} (\Phi^m)^*i^*M 
\xrightarrow{\begin{pmatrix}
-s_{1,0} & 0 & 0 & 0 &\cdots
\\ 
\id  & -s_{2,1} & 0 & 0 & \cdots
\\
0 & \id & - s_{3,2} & 0 & \cdots 
\\
0 & 0 & \id & - s_{4,3} & \cdots
\\
\vdots & \vdots& \vdots & \vdots & \ddots & 
\end{pmatrix}
} 
\bigoplus_{n = 0}^{\infty} (\Phi^n)^*i^*M 
\\
\xrightarrow{\begin{pmatrix} \id&  s_{1,0} &  s_{2,0} &  \cdots
\end{pmatrix}}
i^*M
\to 
0.
\end{multline*}
Since $s_{m,n} \circ s_{l,m} = s_{l,n}$ for $l \ge m \ge n$ and $s_{m,m} = \id$ hold, this
sequence is split exact, with a splitting given by
\[ 
\bigoplus_{m = 1}^{\infty} (\Phi^m)^*M
\xleftarrow{
\begin{pmatrix}
0  & \id  & s_{2,1}  & s_{3,1} & s_{4,1} & \cdots
\\ 
0 & 0 & \id  & s_{3,2} &  s_{4,2} & \cdots
\\
0 & 0 & 0 & \id & s_{4,3} & \cdots 
\\
0 & 0 & 0 & 0 & \id & \cdots
\\
0 & 0 & 0 & 0 & 0  & \cdots
\\
\vdots & \vdots& \vdots & \vdots & \vdots & \ddots & 
\end{pmatrix}
} 
\bigoplus_{n= 0}^{\infty} (\Phi^n
)^*M 
\xleftarrow{\begin{pmatrix} \id \\ 0 \\ 0 \\ 0 \\ 0 \\ \vdots
\end{pmatrix}}
M.
\]
\end{proof}

%%%%%%%%%%%%%%%%%%%%%%%%%%%%%%%%%%%%%%%%%%%%%%%%%%%%%%%%%%%

\subsection{Localization}%
\label{subsec:localization}

\begin{definition}[Local module]\label{def:local_module}
  We call a $\IZ\cala_{\Phi}[t]$-module $M$ \emph{local}, if for any object $A$ in $\cala$
  and any natural number $k \in \IN$ the map
  \[
    M(\id_{\Phi^k(A)} \cdot t^k) \colon M(A) \to M(\Phi^k(A))
  \]
  induced by the morphism $\id_{\Phi^k(A)} \cdot t^k \colon A \to \Phi^k(A)$ in
  $\cala_{\Phi}[t]$ is bijective.
\end{definition}

 Let $j \colon \cala_{\Phi}[t] \to \cala_{\Phi}[t,t^{-1}]$ be the inclusion.

\begin{lemma}\label{lem:characterization_of_local} A $\IZ\cala_{\Phi}[t]$-module $M$
  is local, if and only if there is a $\IZ\cala_{\Phi}[t,t^{-1}]$-module $N$
  such that $M$ and $j^*N$ are isomorphic as $\IZ\cala_{\Phi}[t]$-modules.
\end{lemma}
\begin{proof} Since the morphism $\id_{\Phi^k(A)} \cdot t^k \colon A \to \Phi^k(A)$ in
$\cala_{\Phi}[t]$ becomes invertible when considered in $\cala_{\Phi}[t,t^{-1}]$,
a $\IZ\cala_{\Phi}[t]$-module $M$ is local, if there is a $\IZ\cala_{\Phi}[t,t^{-1}]$-module $N$
such that $M$ and $j^*N$ are isomorphic as $\IZ\cala_{\Phi}[t,t^{-1}]$-modules.

Now consider a local $\IZ\cala_{\Phi}[t]$-module $M$. We have to explain how the
$\IZ\cala_{\Phi}[t]$-structure extends to a $\IZ\cala_{\Phi}[t,t^{-1}]$-structure.
Consider a morphism $u \colon A \to B$ in $\cala_{\Phi}[t,t^{-1}]$. Then we can choose a
natural number $m$ such that the composite
$A \xrightarrow{u} B \xrightarrow{\id_{\Phi^m(B)} \cdot t^m} \Phi^m(B)$ is a morphism in
$\cala_{\Phi}[t]$. Hence we have the $\IZ$-map
$M((\id_{\Phi^m(B)} \cdot t^m)\circ u) \colon M(\Phi^m(B)) \to M(A)$.  Since $M$ is local,
the $\IZ$-map $M(\id_{\Phi^m(B)} \cdot t^m) \colon M(\Phi^m(B)) \to M(B)$ is a
isomorphism.  Now define
\[
M(u) \colon M(B) \xrightarrow{M(\id_{\Phi^m(B)} \cdot t^m)^{-1}} M(\Phi^m(B)) 
\xrightarrow{M((\id_{\Phi^m(B)} \cdot t^m)\circ u)} M(A).
\]
We leave it to the reader to check that the definition of $M(u)$ is independent of the
choice of $m$ and that we obtain the desired $\IZ\cala_{\Phi}[t,t^{-1}]$-structure on $M$
extending the given $\IZ\cala_{\Phi}[t]$-structure.
\end{proof}

Let $M$ be a $\IZ\cala_{\Phi}[t]$-module. We want to assign to it a
$\IZ\cala_{\Phi}[t,t^{-1}]$-module $S^{-1}M$ as follows.  Consider an object $A$ in
$\cala$. Define the abelian group
\[
S^{-1}M(A) := \{(l,x) \mid  l \in \IZ, x \in M(\Phi^l(A))\}/\sim
\]
for the equivalence relation $\sim$, where $(l_0,x_0)$ and $(l_1,x_1)$ are equivalent, if
and only if there is an integer $l \in \IZ$ with $l \le l_0,l_1$ such that the elements
$M(\id_{\Phi^{l_0}(A)} \cdot t^{l_0-l})(x_0)$ and
$M(\id_{\Phi^{l_1}(A)} \cdot t^{l_1-l})(x_0)$ of $M(\Phi^l(A))$ agree.  Given a morphism
$u \colon A \to B$ in $\cala_{\Phi}[t,t^{-1}]$, we can choose a natural number $m$ such
that the composite $A \xrightarrow{u} B \xrightarrow{\id_{\Phi^m(B)} \cdot t^m} \Phi^m(B)$
is a morphism in $\cala_{\Phi}[t]$.  Define
$S^{-1}(M)(u) \colon S^{-1}(M)(B) \to S^{-1}M(A)$ by sending $[l,x]$ 
to the class of
$\bigl(l-m, M\bigl(\overline{\Phi}^{l-m}((\id_{\Phi^m(B)} \cdot t^m) \circ
u)\bigr)(x)\bigr)$.  This is independent of the choice of the representative of $[l,x]$,
since we get for the different representative $(l-1,M(\id_{\Phi^{l}(B)} \cdot t)(x))$ 
\begin{eqnarray*}
\lefteqn{S^{-1}(M)([l-1,M(\id_{\Phi^{l}(B)} \cdot t)(x)])}
& & 
\\
& = &
\bigl[l-1-m, M\bigl(\overline{\Phi}^{l-1-m}((\id_{\Phi^m(B)} \cdot t^m) \circ u)\bigr) \circ M(\id_{\Phi^{l}(B)} \cdot t)(x)\bigr]
\\
& = &
\bigl[l-1-m, M\bigl((\id_{\Phi^{l}(B)} \cdot t) \circ \overline{\Phi}^{l-1-m}((\id_{\Phi^m(B)} \cdot t^m) \circ u)\bigr)(x)\bigr]
\\
& = &
\bigl[l-1-m, M\bigl((\id_{\Phi^{l}(B)} \cdot t) \circ (\id_{\Phi^{l-1(B)}} \cdot t^m) \circ \overline{\Phi}^{l-1-m}(u)\bigr)(x)\bigr]
\\
& = &
\bigl[l-1-m, M\bigl((\id_{\Phi^{l}(B)} \cdot t^{m} \circ \id_{\phi^{l-m}(B)} \cdot t) \circ \overline{\Phi}^{l-1-m}(u)\bigr)(x)\bigr]
\\
& = &
\bigl[l-1-m, M\bigl((\id_{\Phi^{l}(B)} \cdot t^{m}) \circ  \overline{\Phi}^{l-m}(u)\bigr)(x)\bigr]
\\
& = &
S^{-1}(M)([l,x]).
\end{eqnarray*}
This is independent of the choice of $m$ by the following calculation
\begin{eqnarray*}
\lefteqn{\bigl[l-(m+1), M\bigl(\overline{\Phi}^{l-(m+1)}((\id_{\Phi^{m+1}(B)} \cdot t^{m+1}) \circ u)\bigr)(x)\bigr]}
& & 
\\
& = &
\bigl[l-(m+1), M\bigl(\overline{\Phi}^{l-(m+1)}(\id_{\Phi^{m+1}(B)} \cdot t^{m+1})  \circ \overline{\Phi}^{l-(m+1)}(u)\bigr)(x)\bigr]
\\
& = &
\bigl[l-(m+1), M\bigl((\id_{\Phi^{l}(B)} \cdot t^{m+1})  \circ \overline{\Phi}^{l-(m+1)}(u)\bigr)(x)\bigr]
\\
& = & 
\bigl[l-m-1 , M\bigl((\id_{\Phi^l(B)} \cdot t^m)  \circ (\id_{\Phi^{l-m}(B)} \cdot t) \circ \overline{\Phi}^{l-m-1}(u)\bigr)(x)\bigr]
\\
& = & 
\bigl[l-m-1 , M\bigl((\id_{\Phi^l(B)} \cdot t^m) \circ \overline{\Phi}^{l-m}(u) \circ (\id_{\Phi^{l-m}(B)} \cdot t)\bigr)(x)\bigr]
\\
& = & 
\bigl[l-m-1 , M\bigl(\overline{\Phi}^{l-m}(\id_{\Phi^m(B)} \cdot t^m) 
\circ \overline{\Phi}^{l-m}(u) \circ (\id_{\Phi^{l-m}(B)} \cdot t)\bigr)(x)\bigr]
\\
& = & 
\bigl[l-m-1 , M\bigl(\overline{\Phi}^{l-m}((\id_{\Phi^m(B)} \cdot t^m) \circ u) \circ (\id_{\Phi^{l-m}(B)} \cdot t)\bigr)(x)\bigr]
\\
& = & 
\bigl[l-m-1 , M(\id_{\phi^{l-m}(B)} \cdot t) \circ M\bigl(\overline{\Phi}^{l-m}((\id_{\Phi^m(B)} \cdot t^m) \circ u)\bigr)(x)\bigr]
\\
& = &
\bigl[l-m, M\bigl(\overline{\Phi}^{l-m}((\id_{\Phi^m(B)} \cdot t^m) \circ u)\bigr)(x)\bigr].
\end{eqnarray*}
We leave it to the reader to check that $S^{-1}M(v \circ u) = S^{-1}M(u) \circ S^{-1}M(v)$
holds for any two composable morphisms $u \colon A \to B$ and $v \colon B \to C$ in
$\cala_{\Phi}[t,t^{-1}]$ and $S^{-1}M(\id_A) = \id_{S^{-1}M(A)}$ holds for any object $A$
in $\cala$. Note that the $\cala_{\phi}[t]$-module $j^*S^{-1}M$ is local by
Lemma~\ref{lem:characterization_of_local}

There is a natural map of $\cala_{\Phi}[t]$-modules
\begin{equation*}
I \colon M \to j^*S^{-1}M,
\end{equation*}
which is given for an object $A$ of $\cala$ by the map $I(A) \colon M(A) \to S^{-1}M(A)$
sending $x$ to $(0,x)$. We claim that $I$ is \emph{a localization} in the sense that for any
local $\cala_{\phi}[t]$-module $N$ and any $\cala_{\phi}[t]$-homomorphism
$f \colon M \to N$ there exists precisely one $\cala_{\phi}[t]$-homomorphism
$S^{-1}f \colon S^{-1}M \to N$.

Firstly we explain that there is at most one such map $S^{-1}f$ with these properties.
Namely, consider an object $A \in \cala$ and an element $[m,x] \in S^{-1}(M)(A)$. If
$m \ge 0$, then we compute
\begin{eqnarray}
S^{-1}f(A)([m,x])
& = &
S^{-1}(A)([0,M(\id_{\Phi^m(A)} \cdot t^m)(x)])
\label{S_upper_(-1)f(A)([m,x])_m_le_0}
\\
& = & 
S^{-1}(A) \circ I(A) \circ M(\id_{\Phi^m(A)} \cdot t^m)(x)
\nonumber
\\
& = &
f(A) \circ M(\id_{\Phi^m(A)} \cdot t^m)(x).
\nonumber
\end{eqnarray}
Suppose $m \le 0$. 
Since we have $S^{-1}(M)(\id_{A} \cdot t^{-m})([m,x]) = [0,x]$,
we compute for $[m,x] \in S^{-1}M(A)$
\begin{eqnarray*}
\lefteqn{S^{-1}(N)(\id_{A} \cdot t^{-m}) \circ S^{-1}f(A)([m,x])}
& & 
\\
& = & 
S^{-1}f(\Phi^{m}(A)) \circ S^{-1}(M)(\id_{A} \cdot t^{-m})([m,x])
\\
& = & 
S^{-1}f(\Phi^{m}(A))([0,x])
\\
& = & 
S^{-1}f(\Phi^{m}(A)) \circ I(A)(x)
\\
& = & 
f(\Phi^m(A))(x).
\end{eqnarray*}
Since the locality of $N$ implies that
$S^{-1}(N)(\id_{\phi^{m}(A)} \cdot t^m)$ is an isomorphism, we conclude
\begin{equation}
 S^{-1}f(A)([m,x])  =  S^{-1}(N)(\id_{A} \cdot t^{-m})^{-1} \circ f(\Phi^m(A))(x).
\label{S_upper_(-1)f(A)([m,x])_m_ge_0}
\end{equation}
Hence $S^{-1}f(A)$ is determined by the equations~\eqref{S_upper_(-1)f(A)([m,x])_m_le_0}
and~\eqref{S_upper_(-1)f(A)([m,x])_m_ge_0}.  We leave it to the reader to check that it makes
sense to define the desired $\IZ_{\cala}[t]$-homomorphism $S^{-1}f(A)$ by the
equations~\eqref{S_upper_(-1)f(A)([m,x])_m_le_0} and~\eqref{S_upper_(-1)f(A)([m,x])_m_ge_0}.

The adjoint of $I \colon M \to j^*S^{-1}M$ under the adjunction~\eqref{adjunction_(F_ast,F_upper_ast)} 
is denoted by
\begin{equation}
\alpha \colon  j_*M \to S^{-1}M.
\label{alpha_adjoint_of_I}
\end{equation}
The adjoint of $\id_{j_*M}$ under the adjunction~\eqref{adjunction_(F_ast,F_upper_ast)}
  is the $\IZ\cala_{\Phi}[t]$-homomorphism
  \begin{equation}
    \lambda \colon M \to j^*j_*M,
   \label{lambda_M_to_j_upper_astJ_astM}
  \end{equation}
which is explicitly given by
$M(??) \to \mor_{\cala_{\Phi}[t,t^{-1}]}(??,?) \otimes_{\IZ\cala_{\Phi}[t]} M(?) $
sending $u \in M(??) $ to $\id_{??} \otimes u$.  
Given  an $\IZ\cala_{\Phi}[t,t^{-1}]$-module $N$, the adjoint of $\id_{j^*N}$ under the 
adjunction~\eqref{adjunction_(F_ast,F_upper_ast)} is the
$\IZ\cala_{\Phi}[t,t^{-1}]$-homomorphism
\begin{equation}
\rho \colon j_*j^*N \to N,
\label{rho_j_astj_upper_astN_to_N}
\end{equation}
which is explicitly given by 
$ N(?) \otimes_{\IZ\cala_{\Phi}[t]} \mor_{\cala_{\Phi}[t,t^{-1}]}(??,?)  \to N(??)$ sending
$x \otimes u$ to $N(u)(x)=xu$. 

\begin{lemma}\label{lem:localization}

  \begin{enumerate}

  \item\label{lem:localization:lambda_is_localization} The
    $\IZ\cala_{\Phi}[t]$-homomorphism $\lambda \colon M \to j^*j_*M$
    of~\eqref{lambda_M_to_j_upper_astJ_astM} is a localization;

  \item\label{lem:localization:alpha_is_iso} The $\IZ\cala_{\phi}[t,t^{-1}]$-homomorphism
    $\alpha \colon j_*M \to S^{-1}M$ of~\eqref{alpha_adjoint_of_I} is an
    isomorphism, which is natural in $M$;

  \item\label{lem:localization:j_astj_upper_ast_is_id} Let $N$ be a
    $\IZ\cala_{\Phi}[t,t^{-1}]$-module. Then the $\IZ\cala_{\Phi}[t,t^{-1}]$-map
    $\rho \colon j_*j^*N \to N$ of~\eqref{rho_j_astj_upper_astN_to_N} is an isomorphism.

  \end{enumerate}
\end{lemma}
\begin{proof}~\ref{lem:localization:lambda_is_localization} 
 Let $f \colon M \to N$ be a $\IZ\cala_{\Phi}[t]$-map with a local
  $\IZ\cala_{\Phi}[t]$-module as target. Because of
  Lemma~\ref{lem:characterization_of_local} there is a $\IZ\cala_{\Phi}[t,t^{-1}]$-module
  $N'$ and a $\IZ\cala_{\Phi}[t]$-isomorphism $u \colon N \to j^*N'$.  Let the
  $\IZ\cala_{\Phi}[t,t^{-1}]$-map $v_0 \colon j_*M \to N$ be the adjoint of $u \circ f$
  under the adjunction~\eqref{adjunction_(F_ast,F_upper_ast)}. Because of the naturality
  of the adjunction~\eqref{adjunction_(F_ast,F_upper_ast)} we get for the composite
  $\overline{f} \colon j^*j_*M \xrightarrow{j^*v_0} j^*N' \xrightarrow{u^{-1}} N$ that
  $\overline{f} \circ \lambda = f$ holds. We conclude that
  $\overline{f}$ is uniquely determined by $\overline{f} \circ \lambda = f$
  from the explicite description of
  $\lambda$ and from the fact that for any morphism $u \colon A \to B$ in
  $\cala_{\Phi}[t,t^{-1}]$ there is a natural number such that the composite of $u$ with
  $\id_{\Phi^m(B)} \cdot t^m \colon \Phi(M) \to \Phi^m(B)$ lies in $\cala_{\Phi}[t]$.
\\[1mm]~\ref{lem:localization:alpha_is_iso}
Obviously $\alpha$ is natural in $M$. The naturality of the adjunction~\eqref{adjunction_(F_ast,F_upper_ast)}  implies
\[
j^*\alpha \circ \lambda = I.
\]
Since both $I \colon M \to j^*S^{-1}M$ and
$\lambda \colon M \to j^*j_* M$ a localizations,
$j^*\alpha$ and hence $\alpha$ are bijective.
\\[1mm]~\ref{lem:localization:j_astj_upper_ast_is_id} It suffices to show that
$j^*\rho \colon j^*j_*j^*N \to j^*N$ is bijective.
Assertion~\ref{lem:localization:lambda_is_localization} applied to $j^*N$ and the
naturality of the adjunction~\eqref{adjunction_(F_ast,F_upper_ast)} imply that
$j^*\rho \colon j^*N \to j^*j_*j^*N$ is a localization.  Since
$\id_{j^*N} \colon j^*N \to j^*N$ is a localization, $j^*\rho$ is an isomorphism.
\end{proof}

\begin{lemma}\label{lem:j_ast_is_flat}
The functor $j_* \colon \MODcatr{\IZ\cala_{\phi}[t]} \to \MODcatr{\IZ\cala_{\phi}[t,t^{-1}]}$ is flat.
\end{lemma}
\begin{proof}
Because of the adjunction~\eqref{adjunction_(F_ast,F_upper_ast)}  the functor $j_*$ is 
right exact by a general argument, see~\cite[Theorem~2.6.1. on page~51]{Weibel(1994)}.
Hence it remains to show that for an injective $\IZ\cala_{\phi}[t]$-map $i \colon M \to N$
the $\IZ\cala_{\phi}[t,t^{-1}]$-map $j_*i \colon j_*M \to j_*N$ is injective. In view of 
Lemma~\ref{lem:localization}~\ref{lem:localization:alpha_is_iso}
it suffices to show that $S^{-1}i \colon S^{-1}M \to S^{-1}N$ is injective.
Consider an object $A$ in $\cala$ and an element $[l,x]$ in the kernel of $S^{-1}(i)(A)$. Since 
$S^{-1}i([l,x] )= [l,i(\Phi^l(A))(x)]$, there is a natural number $m \le l$
such that $N(\id_{\Phi^{l}(A)} \cdot t^{m-l})(i(\Phi^l(A))(x)) = 0$. 
Since $N(\id_{\Phi^{l}(A)} \cdot t^{m-l}) \circ i(\Phi^l(A)) = i(\Phi^{m-l}(A)) \circ M(\id_{\Phi^{l}(A)} \cdot t^{m-l})$
and $i(\Phi^{m-l}(A))$ is by assumption injective, $M(\id_{\Phi^{l}(A)} \cdot t^{m-l})(x) = 0$.
This implies $[l,x] = 0$. 
\end{proof}

%%%%%%%%%%%%%%%%%%%%%%%%%%%%%%%%%%%%%%%%%%%%%%%%%%%%%%%%%%%

\subsection{Global dimension}%
\label{subsec:Global_dimension}

Recall that an additive category $\cala$ has \emph{global dimension $\le d$}, if 
the abelian category $\MODcatr{\IZ\cala}$  has global dimension $\le d$, i.e.,
if every $\IZ\cala$-module has a projective $d$-dimensional resolution,
see Definition~\ref{def:Regularity_properties_of_additive_categories}~%
\ref{def:Regularity_properties_of_additive_categories:global_homological_dimension}.

\begin{theorem}[Global dimension and the passage from $\cala$ to $\cala_{\Phi}{[t]}$]%
\label{the:global_dimension_for_cala_and_cala[t]}
  Let $\cala$ be an additive category $\cala$ and $\Phi \colon \cala \to \cala$ be an
  automorphism of additive categories. 

  \begin{enumerate}
  \item\label{the:global_dimension_for_cala_and_cala[t]_modules} Let $M$ be a
    $\IZ\cala_{\Phi}[t]$-module. If      $\pdim_{\cala}(i^*M) \le d$, then $\pdim_{\cala[t]}(M) \le d+1$;   
  \item\label{the:global_dimension_for_cala_and_cala[t]:global}
  If  $\cala$ has global dimension $\le d$, then $\cala_{\Phi}[t]$ has global  dimension $\le (d+1)$.
\end{enumerate}
\end{theorem}

Theorem~\ref{the:global_dimension_for_cala_and_cala[t]} is a version
of the Hilbert syzygy theorem. Its the proof is not much
different from the classical syzygy Theorem for rings.  For a more general version
see~\cite[Corollary~31.1 on page~119]{Mitchell(1972)}.  

\begin{proof}[Proof of Theorem~\ref{the:global_dimension_for_cala_and_cala[t]}]
  Obviously $i^* \colon \MODcatr{\IZ\cala_{\Phi}[t]} \to \MODcatr{\IZ\cala}$ 
  is faithfully flat and is compatible with direct sums  over arbitrary index sets.  Next we show that
 $i^*$ sends projective $\IZ\cala_{\phi}[t]$-modules to projective
 $\IZ\cala$-modules. It suffices to show that
 $i^*\mor_{\cala_{\Phi}[t]}(?,A) \cong  i^*i_*\mor_{\cala}(?,A)$ is free as a $\IZ\cala$-module for any object 
$A$.  This  follows from the $\IZ\cala$-isomorphism~\eqref{I_upper_ast_circ_i_ast_is_bigoplus},
since $(\Phi^k)^*\mor_{\cala}(?,A) \cong \mor_{\cala}(?,\Phi^{-k}(A))$. 

The functor $i_* \colon \MODcatr{\IZ\cala} \to \MODcatr{\IZ\cala_{\Phi}[t]}$ is compatible
with direct sums over arbitrary index sets, is right exact and sends $\mor_{\cala}(?,A)$
to $\mor_{\cala_{\Phi}[t]}(?,A)$.  In particular $i_*$ respects the properties finitely
generated, free, and projective. Next we want to show that $i_*$ is faithfully flat. For
this purpose it suffices to show that $i^* \circ i_*$ is faithfully flat. This is obvious
since $i^* \circ i_*$ is the functor sending a morphism $f \colon M \to N$ to the morphism
$\bigoplus_{k \in \IN} (\Phi^{k})^*(f) \colon \bigoplus_{k \in \IN} (\Phi^{k})^*(M) \to
\bigoplus_{k \in \IN} (\Phi^{k})^*(f)$ under the identification~\eqref{I_upper_ast_circ_i_ast_is_bigoplus}.

Now consider a $\IZ_{\Phi}[t]$-module $M$ with $\pdim_{\cala}(i^*M) \le d$.
Since the $\IZ\cala$-modules $i^*M$ and $\Phi^*i^*M$ are isomorphic, see~\eqref{iso_alpha_prime},
we get $\pdim_{\cala}(\phi^*i^*M) \le d$. Since $i_*$ is faithfully flat and
respects projective modules, we conclude $\pdim_{\cala_{\Phi}[t]}(i_*i^*M) \le d$ and 
$\pdim_{\cala_{\Phi}[t]}(i_*\Phi^*i^*M) \le d$. Now 
Lemma~\ref{lem:IZ_cala_modules}~%
\ref{lem:IZ_cala_modules:exact_sequences_homological_dimension} and
Lemma~\ref{lem:characteristic_sequence} together imply $\pdim_{\cala_{\Phi}[t]}(M) \le (d+1)$.
\\[1mm]~\ref{the:global_dimension_for_cala_and_cala[t]:global}
This follows directly from assertion~\ref{the:global_dimension_for_cala_and_cala[t]_modules}.
\end{proof}

\begin{theorem}[Global dimension and the passage from $\cala_{\phi}{[t]}$ to $\cala_{\Phi}{[t,t^{-1}]}$]%
\label{the:global_dimension_for_cala[t]_and_cala[t,t_upper_-1]}
Let $\cala$ be an additive category  and let $\Phi \colon \cala \to \cala$ be an
automorphism of additive categories.

\begin{enumerate}

\item\label{the:global_dimension_for_cala[t]_and_cala[t,t_upper_-]:modules}
Let $M$ be a $\IZ\cala_{\Phi}[t,t^{-1}]$-module. 
If we have $\pdim_{\cala[t]}(j^*M) \le d$, then  we get $\pdim_{\cala[t,t^{-1}]}(M) \le d$;

\item\label{the:global_dimension_for_cala[t]_and_cala[t,t_upper_-1]:global} If
  $\cala_{\Phi}[t]$ has global dimension $\le d$, then $\cala_{\Phi}[t,t^{-1}]$ has global
  dimension $\le d$.
\end{enumerate}
\end{theorem}
\begin{proof}~\ref{the:global_dimension_for_cala[t]_and_cala[t,t_upper_-]:modules}
  Let $M$ be a $\IZ[\cala]_{\Phi}[t,t^{-1}]$-module satisfying $\pdim_{\cala[t]}(j^*M) \le d$.   The
  functor $j_* \colon \MODcatr{\IZ\cala_{\phi}[t]} \to \MODcatr{\IZ\cala_{\phi}[t,t^{-1}]}$
  is flat by Lemma~\ref{lem:j_ast_is_flat}.  Since it respects the property projective,
  we get $\pdim_{\cala_{\Phi}[t,t^{-1}]}(j_*j^*M) \le d$. 
  Lemma~\ref{lem:localization}~\ref{lem:localization:j_astj_upper_ast_is_id} implies
  $\pdim_{\cala[t,t^{-1}]}(M) \le d$.
  \\[1mm]~\ref{the:global_dimension_for_cala[t]_and_cala[t,t_upper_-1]:global}
 This follows from assertion~\ref{the:global_dimension_for_cala[t]_and_cala[t,t_upper_-]:modules}.
\end{proof}

%%%%%%%%%%%%%%%%%%%%%%%%%%%%%%%%%%%%%%%%%%%%%%%%%%%%%%%%%%%%%%%%%%%%
%%%%%%%%%%%%%%%%%%%%%%%%%% Section 9 %% %%%%%%%%%%%%%%%%%%%%%%%%%%%%%%%%
%%%%%%%%%%%%%%%%%%%%%%%%%%%%%%%%%%%%%%%%%%%%%%%%%%%%%%%%%%%%%%%%%%%%

\typeout{------------------- Section:  Regular additive categories -----------------}

\section{Regular additive categories}%
\label{sec:Regular_additive_categories}

Regularity for additive categories $\cala$ requires finite resolutions of finitely
presented modules, but not for arbitrary modules.  In particular, regularity has no
consequence for global dimension and we cannot use
Theorem~\ref{the:global_dimension_for_cala_and_cala[t]} in the following result.

\begin{theorem}[Regularity and  the passage from $\cala$ to $\cala_{\Phi}{[t]}$]%
\label{the:Regularity_for_cala_and_cala[t]}
Let $\cala$ be an additive category $\cala$ and let $\Phi \colon \cala \to \cala$ be an
automorphism of additive categories.  Let $l$ be a natural number.

\begin{enumerate}
\item\label{the:Regularity_for_cala_and_cala[t]:first}
Suppose that $\cala$ is regular or $l$-uniformly regular respectively.
Then $\cala_{\Phi}[t]$ is regular or $(l+2)$-uniformly regular respectively;

\item\label{the:Regularity_for_cala_and_cala[t]:second}
Suppose that $\cala[t]$ is regular or $l$-uniformly regular respectively.
Then $\cala_{\Phi}[t,t^{-1}]$ is regular or $l$-uniformly regular respectively.

\end{enumerate}
\end{theorem}

\begin{proof}~\ref{the:Regularity_for_cala_and_cala[t]:first} We know already that
  $\cala_{\phi[t]}$ is Noetherian because of
  Theorem~\ref{the:Hilbert_Basis_Theorem_for_additive_categories}~%
\ref{the:Hilbert_Basis_Theorem_for_additive_categories:cala_to_cala[t]}.
  Let $M$ be a finitely generated $\cala_{\phi}[t]$-module.  We have to show that it has a
  finitely generated projective resolution, which is finite-dimensional or
  $(l+1)$-dimensional.  Since $\cala_{\phi}[t]$ is Noetherian, there exists a finitely
  generated projective resolution of $M$, which may be infinite-dimensional.  We conclude
  from Theorem~\ref{lem:IZ_cala_modules}~\ref{lem:IZ_cala_modules:finite_dimension} that
  it suffices to show the projective dimension of $M$ is finite or bounded by $(l+1)$
  respectively.  As $M$ is finitely generated, we find a finite collection of elements
  $x_j \in M(Z_j)$ with objects $Z_j$ from $\cala$ such that the $x_j$ generate $M$ as an
  $\IZ\cala_\Phi[t]$-module.  For $d \geq 0$ consider the morphism
  $\id_{Z_j}\cdot t^d \colon \Phi^{-d}(Z_j) \to Z_j$ in $\cala_\Phi[t]$ and set
  $x_j[d] := M(\id_{Z_j}\cdot t^d)(x_j) \in M(\Phi^{-d}(Z_j))$.  Let $M_n$ be the
  $\IZ\cala$-submodule of $i^*M$ generated by all $x_j[d]$ with $d \leq n$.  We obtain an
  increasing subsequence $M_0 \subseteq M_0 \subseteq M_1 \subseteq M_2 \subseteq \cdots $
  of $\IZ\cala$-submodules of $i^*M$ with $i^*M = \bigcup_{n \ge 0} M_n$.  Let
  $T_n \colon i^*M \to \Phi^*i^*M$ be the following $\IZ\cala$-morphism.  For an object
  $Z$ from $\cala$ consider
  $\id_{\Phi^n(Z)} \cdot t^n \in \mor_{\cala_\Phi[t]}(Z,\Phi^n(Z))$ and define
  $T_Z \colon i^*M(Z) = M(Z) \to \Phi^*i^*M(Z) = M(\Phi(Z))$ to be
  $M(\id_{\Phi^n(Z)} \cdot t^n)$.  Let
  $\pr_n \colon (\Phi^n)^*(M_n) \to (\Phi^n)^*(M_n) / (\Phi^n)^*(M_{n-1})$ be the
  projection.  The composite
 \begin{equation*}
 f_n \colon 	M_0 \xrightarrow{\overline{T_n}} (\Phi^n)^*(M_n) \xrightarrow{\pr_n} (\Phi^n)^*(M_n) / (\Phi^n)^*(M_{n-1})
 \end{equation*}
 is surjective and we write $K_n$ for its kernel.  We obtain an increasing sequence of
 $\cala$-submodules $K_0 \subseteq K_1 \subseteq K_2 \subseteq \cdots $ of $M_0$.  Since
 $\cala$ is Noetherian and $M_0$ is finitely generated, there exists an integer
 $n_0 \ge 1$ such that $K_n = K_{n_0}$ holds for $n \ge n_0$.  Define
 $g_n \colon (\Phi^{n_0})^*M_{n_0}/(\Phi^{n_0})^*M_{n_0-1} \to
 (\Phi^n)^*M_n/(\Phi^n)^*M_{n-1}$ for $n \ge n_0$ to be the map induced by
 $\Phi^*({T_{n-n_0}})$ for $n \ge n_0$.  We obtain for every natural number $n$ with
 $n \ge n_0$ a commutative diagram of $\IZ\cala$-modules with exact rows
   \[
   \xymatrix{0 \ar[r]
    & 
    K_{n_0} \ar[r] \ar[d]^{\cong}
    & 
    M_0 \ar[r]^-{f_{n_0}} \ar[d]^{\id_{M_0}}
    & 
    (\Phi^{n_0})^*M_{n_0}/(\Phi^{n_0})^*M_{n_0-1} \ar[r] \ar[d]^{g_n}
    &
    0
    \\
    0 \ar[r]
    & 
    K_{n} \ar[r] 
    & 
    M_0 \ar[r]^-{f_n} 
    & 
    (\Phi^n)^*M_n/(\Phi^n)^*M_{n-1} \ar[r] 
    &
    0
    }
    \]
    Hence $g_n$ is an isomorphism of $\IZ\cala$-module $n \ge n_0$.  
    As $\Phi^*$ is an isomorphism we have   
    \begin{equation*}
    	\pdim_{\IZ\cala} M_n/M_{n-1} = \pdim_{\IZ\cala} (\Phi^n)^*(M_n/M_{n-1}) = \pdim_{\IZ\cala} (\Phi^n)^*M_n/(\Phi^n)^*M_{n-1}). 
    \end{equation*}
    Thus for $n \geq n_0$ we have $\pdim_{\IZ\cala}(M_n/M_{n-1}) = \pdim_{\IZ\cala}(M_{n_0} / M_{n_0-1})$.
    We have the short exact sequence $0 \to M_{n-1} \to M_n \to M_n/M_{n-1} \to 0$ and hence get from 
    Lemma~\ref{lem:IZ_cala_modules}~\ref{lem:IZ_cala_modules:exact_sequences_homological_dimension} 
    \[
    \pdim_{\IZ\cala}(M_n) \le 
    \sup\{\pdim_{\IZ\cala}(M_{n-1}), \pdim_{\IZ\cala}(M_n/M_{n-1})\}.
    \]
    This implies by induction over $n \ge n_0$
    \[
    \pdim_{\IZ\cala}(M_n) \le 
    \sup\{\pdim_{\IZ\cala}(M_{n_0-1}), \pdim_{\IZ\cala}(M_{n_0}/M_{n_0-1})\}.
    \]
    Put 
    \[
    D := \sup\bigl\{\sup\{\pdim_{\IZ\cala}(M_k) \mid k = 0,1, \ldots, n_0-1\},\pdim_{\IZ\cala}(M_{n_0}/M_{n_0-1})\bigr\}.
    \]
    Note that $D < \infty$, if $\cala$ is regular, and $D \le l$, if $\cala$ is uniformly $l$-regular.
    We get 
    \[
    \pdim_{\IZ\cala}\biggl(\bigoplus_{n \in \IN} M_n\biggr)  \le \sup\{\pdim_{\IZ\cala}(M_n) \mid n \ge 0\} \le D.
     \]
     We have the  short exact sequence of $\cala$-modules 
     \[
     0 \to \bigoplus_{n \in \IN} M_n \to \bigoplus_{n \in \IN} M_n \to i^*M \to 0,
     \]
     where the first map is given by $(x_n)_{n \ge 0} \mapsto (x_0, x_1 -x_0. x_2, -x_1, \ldots )$ 
     and the second by  $(x_n)_{n \ge 0} \mapsto \sum_{n \ge 0} x_n$.  We conclude from
     Lemma~\ref{lem:IZ_cala_modules}~\ref{lem:IZ_cala_modules:exact_sequences_homological_dimension} 
    \[
    \pdim_{\IZ\cala}(i^*M) \le D +1.
    \]
    Now Theorem~\ref{the:global_dimension_for_cala_and_cala[t]}~\ref{the:global_dimension_for_cala_and_cala[t]_modules}
    implies 
    \[
    \pdim_{\IZ\cala_{\Phi}[t]}(M) \le D +2.
    \]
    This finishes the proof if assertion~\ref{the:Regularity_for_cala_and_cala[t]:first}.
    \\[1mm]~\ref{the:Regularity_for_cala_and_cala[t]:second}
    We know already that
    $\cala_{\phi}[t,t^{-1}]$ is Noetherian because of
    Theorem~\ref{the:Hilbert_Basis_Theorem_for_additive_categories}~%
\ref{the:Hilbert_Basis_Theorem_for_additive_categories:cala[t]_to_cala[t,t_upper_-1]}.
    Let $M$ be a finitely generated $\cala_{\phi}[t,t^{-1}]$-module.  We can find a
    finitely generated free $\IZ\cala_{\Phi}[t]$-module $F_0$ and a free
    $\IZ\cala_{\Phi}[t]$-module $F_1$ together with an exact sequence of
    $\IZ\cala_{\Phi}[t,t^{-1}]$-modules
    $j_* \F_0 \xrightarrow{f} j_*F_1 \xrightarrow{e} M \to 0$.  Here we write $j$ for the
    inclusion $\cala_{\Phi}[t] \to \cala_\phi[t,t^{-1}]$. By composing $f$ with an
    appropriate automorphism of $j_*F_0$ one can arrange that $f = j_*g$ for some
    $\IZ\cala_{\Phi}[t]$-homomorphism $g \colon F_0 \to F_1$.  The cokernel of $g$ is a
    finitely generated $\IZ\cala_{\Phi}[t]$-module $N$ and there is an obvious exact
    sequence of $\IZ\cala_{\Phi}[t]$-modules $F_1 \xrightarrow{g} F_1 \to N \to 0$.  Since
    the functor $j_*$ is flat by Lemma~\ref{lem:j_ast_is_flat} and respects the property
    projective, we obtain an $\IZ\cala_{\Phi}[t,t^{-1}]$-isomorphism
    $j_*N \xrightarrow{\cong} M$ and have
    $\dim_{\IZ\cala_{\Phi}[t,t^{-1}]}(j_*N) \le \dim_{\IZ\cala_{\Phi}[t]}(N)$. Hence we
    get $\dim_{\IZ\cala_{\Phi}[t,t^{-1}]}(M) \le \dim_{\IZ\cala_{\Phi}[t]}(N)$.  This
    finishes the proof of Theorem~\ref{the:Regularity_for_cala_and_cala[t]}.
\end{proof}

\begin{remark}\label{rem:regular_versus_regular_coherent}
We do not know whether Theorem~\ref{the:Regularity_for_cala_and_cala[t]}
remains true if we replace  regular by regular coherent.  To our knowledge it is an open problem,
whether for a regular coherent ring $R$ the rings $R[t]$ or $R[t,t^{-1}]$ are  regular coherent again.
\end{remark}

%%%%%%%%%%%%%%%%%%%%%%%%%%%%%%%%%%%%%%%%%%%%%%%%%%%%%%%%%%%%%%%%%%%%
%%%%%%%%%%%%%%%%%%%%%%%%%% Section 10 %%%%%%%%%%%%%%%%%%%%%%%%%%%%%%%%%%
%%%%%%%%%%%%%%%%%%%%%%%%%%%%%%%%%%%%%%%%%%%%%%%%%%%%%%%%%%%%%%%%%%%%

\typeout{----------- Directed union and infinite products of additive categories   ----------}

\section{Directed union and infinite products of additive categories}%
\label{sec:Directed_union_and_infinite_products_of_additive_categories}

A functor of additive categories $F \colon \cala \to \calb$ is called \emph{flat}, if for
every exact sequence $A_0 \xrightarrow{i} A_1 \xrightarrow{p} A_2$ in $\cala$ the sequence
in $F(A_0) \xrightarrow{F(i)} F(A_1) \xrightarrow{F(p)} F(A_2)$ in $\calb$ is exact.
It is called \emph{faithfully flat}, if 
a sequence $A_0 \xrightarrow{i} A_1 \xrightarrow{p} A_2$ in $\cala$ is exact, if and only if  the sequence
in $F(A_0) \xrightarrow{F(i)} F(A_1) \xrightarrow{F(p)} F(A_2)$ in $\calb$ is exact.

\begin{lemma}\label{lem:flatness_and_cofinality}
  Let $i \colon \cala \to \cala'$ and $j \colon \calb \to \calb'$ be inclusions of cofinal
  full additive subcategories.  Suppose that the following diagram of functors of additive
  categories commutes
  \[\xymatrix{\cala \ar[r]^{F} \ar[d]_{i}
     &
     \calb \ar[d]^{j}
     \\
     \cala' \ar[r]_{F'}
     &
     \calb'
   }
 \]
 Then
 \begin{enumerate}
 \item\label{lem:flatness_and_cofinality:cala_to_cala_Prime_faithfully_falt}
   The inclusion $i \colon \cala \to \cala'$ is faithfully flat;

 \item\label{lem:flatness_and_cofinality_F_versus_F_prime}
 $F$ is flat or faithfully flat respectively  if  and only if $F'$ is flat or faithfully flat respectively.
\end{enumerate}
\end{lemma}
\begin{proof}
  We first  show that $F'$ is exact or
  faithfully exact respectively, provided that $F$ is exact or faithfully exact respectively.

  Consider morphisms $f' \colon A'_0 \to A'_1$ and $g' \colon A_1' \to A_2'$ in $\cala$.
  Choose objects $A_k$ in $\cala$ and morphisms $i_k \colon A_k' \to A_k$ and
  $r_k \colon A_k \to A_k'$ in $\cala'$ satisfying $r_k \circ i_k = \id_{A_k}$ for
  $k = 0,1,2$. Define $f \colon A_0 \to A_1$ and $g \colon A_1\to A_2$ by
  $f = i_1 \circ f' \circ r_0$ and $g = i_2 \circ g' \circ r_1$. Then the following diagram of
  morphisms in $\cala'$ commutes
  \[
    \xymatrix@!C=8em{A_0' \ar[r]^-{f'} \ar[d]^-{i_0}
      &
      A_1' \ar[r]^-{g'}\ar[d]^-{i_1}
       &
       A_2' \ar[d]^-{i_2 \oplus 0}
       \\
       A_0 \ar[r]^-{f} \ar[d]^-{r_0}
      &
      A_1 \ar[r]^-{g \oplus (\id_{A_1} - i_1 \circ r_1)}\ar[d]^-{r_1}
       &
       A_2 \oplus A_1 \ar[d]^-{r_2 \oplus 0}
       \\
       A_0' \ar[r]^-{f'} 
      &
      A_1' \ar[r]^-{g'}
       &
       A_2'.
     }
   \]
   Next we check that the middle row is exact in $\cala$, if and only if the upper row is
   exact in $\cala'$.  Suppose that the middle row is exact in $\cala$. Consider a morphism
   $v' \colon B' \to A_1'$ in $\cala'$ such that $g' \circ v' = 0$.  Choose an object $B$ in $\cala$
   and maps $j \colon B' \to B$ and $s \colon B \to B'$ with $s \circ j = \id_{B'}$.  Then
   we have the morphism $i_1 \circ v' \circ s \colon B \to A_1$ whose composite with
   $g \oplus (\id_{A_1} - i_1 \circ r_1) \colon A_1 \to A_2 \oplus A_1$ is zero. Hence we can find a
   morphism $u_0 \colon B \to A_0$ with $f \circ u_0 = i_1 \circ v' \circ s$. Define
   $u' \colon B' \to A_0'$ by the composite $r_0 \circ u_0 \circ j$. One easily checks
   that $f' \circ u' = v'$. Hence the upper row is exact in $\cala'$.

   Suppose that the upper row is exact in $\cala'$. Consider a morphism $v \colon B \to A_1$ in $\cala$
   such that $g \oplus (\id_{A_1} - i_1 \circ r_1)  \circ v = 0$. Then $g \circ v = 0$ and
   $v = i_1 \circ r_1 \circ v$. We conclude
   \[
   g' \circ (r_1 \circ v) = r_2 \circ i_2 \circ g'  \circ r_1 \circ v 
= r_2 \circ g \circ i_1 \circ r_1  \circ v   = r_2 \circ g \circ v  = r_2 \circ 0 =0.
   \]
   Since the upper row is exact, we can find $u' \colon B \to A_0'$ satisfying $f' \circ u_0' = r_1 \circ v$.
   Define  $u \colon B \to A_0$ by $i_0 \circ u'$. Then
   \[
   f \circ u = f \circ i_0 \circ u' = i_1 \circ f' \circ u' = i_1 \circ r_1 \circ v = v.
   \]
   Hence the middle row is exact.
   
   If we apply $F'$ and put $i_k' = F'(i_k)$ and $r_k' = F'(r_k)$,
   we get  $r_k' \circ i_k' = \id_{F'(A_k')}$ and the commutative diagram
   \[
    \xymatrix@!C=13em{F'(A_0' )\ar[r]^-{F'(f')} \ar[d]^-{i_0'}
      &
      F'(A_1') \ar[r]^-{F'(g')}\ar[d]^-{i_1'}
       &
       F'(A_2') \ar[d]^-{i_2' \oplus 0}
       \\
       F(A_0) \ar[r]^-{F(f)} \ar[d]^-{r_0'}
      &
      F(A_1) \ar[r]^-{F(g) \oplus (\id_{F(A_1)} -i_1' \circ r_1')}\ar[d]^-{r_1'}
       &
       F'(A_2) \oplus F(A_1) \ar[d]^-{r_2 \oplus 0}
       \\
       F'(A_0') \ar[r]^-{F'(f')} 
      &
      F'(A_1') \ar[r]^-{F'(g')}
       &
       F'(A_2') 
     }
   \]
   and,  by the same argument as above, 
   the middle row is exact in $\calb$, if and only if the upper row is exact in $\calb'$. 
   We conclude that  the functor $F'$ is exact or
  faithfully exact respectively, provided  that  $F$ is exact or faithfully exact respectively.

  Since both $\id_{\cala}$ and $\id_{\calb}$ are faithfully flat, this special case shows
  that both $i \colon \cala \to \cala'$ and $j \colon \calb \to \calb'$ are faithfully
  flat.

   Suppose that $F'$ is flat or faithfully flat respectively. Then
   $j \circ F = F' \circ i$ is flat or faithfully flat respectively.  This implies that
   $F$ is flat or faithfully flat respectively.  This finishes the proof of
   Lemma~\ref{lem:flatness_and_cofinality}.
 \end{proof}   
 
\begin{lemma}\label{lem:regular_coherent_and_unions}
  Let $\cala = \bigcup_{i \in I} \cala_i$ be the directed union of 
  additive subcategories $\cala_i$ for an arbitrary directed set $I$.

  \begin{enumerate}
  \item\label{lem:regular_coherent_and_unions:idempotent_complete} The idempotent
    completion $\Idem(\cala)$ is the directed union of the idempotent completions
    $\Idem(\cala_i)$;

    \item\label{lem:regular_coherent_and_unions:regular_coherent_l_ge_1} Consider $l \ge 1$.

  Suppose that $\cala_i$ is regular coherent or $l$-uniformly regular
  coherent respectively  for every $i \in I$ and for every $i,j \in I$ with
  $i \le j$ the inclusion $\cala_i \to \cala_j$ is flat.
  Then the inclusion $\Idem(\cala_i) \to \Idem(\cala_j)$ is flat for every $i,j \in I$ with
  $i \le j$  and   both $\cala$  and $\Idem(\cala)$ are regular coherent or $l$-uniformly regular coherent respectively;

\item\label{lem:regular_coherent_and_unions:regular_coherent_l_is_0}
  Suppose that $\cala_i$ is $0$-uniformly regular
  coherent respectively  for every $i \in I$.
  Then  both $\cala$  and $\Idem(\cala)$ are $0$-uniformly regular coherent respectively.
\end{enumerate}
\end{lemma}
\begin{proof}~\ref{lem:regular_coherent_and_unions:idempotent_complete}
  This is obvious.
  \\[1mm]~\ref{lem:regular_coherent_and_unions:regular_coherent_l_ge_1} If 
  the inclusion $\cala_i \to \cala_j$ is flat, then also the
  inclusion $\Idem(\cala_i) \to \Idem(\cala_j)$ is flat by Lemma~\ref{lem:flatness_and_cofinality}.
  In view of Lemma~\ref{lem:inclusion_of_full_cofinal_subcategory}~%
\ref{lem:inclusion_of_full_cofinal_subcategory:regular_coherent}
  and assertion~\ref{lem:regular_coherent_and_unions:idempotent_complete},
  we can assume without loss of generality
  that each $\cala_i$ and $\cala$ are idempotent complete.
  Hence we can use  the criterion for regular coherent given in
  Lemma~\ref{lem:intrinsic_reformulation_of_regular_coherent} in the sequel. We treat only the case $l  \ge 2$,
  the case $l = 1$ is proved analogously

Consider a morphism $f_1 \colon A_1 \to A_0$ in $\cala$. Choose an index $i$ such that
$f_1$ belongs to $\cala_i$. Then we can find a sequence of morphisms
\[
0 \to A_n \xrightarrow{f_n} A_{n-1} \xrightarrow{f_{n-1}} \cdots   \xrightarrow{f_2} A_{1} \xrightarrow{f_1} A_0
\]
which is in $\cala_i$ exact at $A_k$ for $k = 1,2, \ldots, n$. It remains to show that
this sequence is exact at $\cala$ at $A_k$ for $k = 1,2, \ldots, n$. Fix
$k \in \{1,2, \ldots, n\}$.  It remains to show for any object $A \in \cala$ and morphism
$g \colon A \to A_k$ with $f_k \circ g = 0$ that there exists a morphism
$\overline{g} \colon A \to A_{k+1}$ with $f_{k+1} \circ \overline{g} = g$.  We can choose
$j \in I$ with $i \le j$ such that $g$ belongs to $\cala_j$. Since
$A_{k+1} \xrightarrow{f_{k+1}} A_{k} \xrightarrow{f_{k}} A_{k-1}$ is exact in $\cala_i$,
we conclude from the assumptions that it is also exact in $\cala_j$ and hence we can
construct the desired lift $\overline{g}$ already in $\cala_j$.
\\[1mm]~\ref{lem:regular_coherent_and_unions:regular_coherent_l_is_0}
In view of Lemma~\ref{lem:inclusion_of_full_cofinal_subcategory}~%
\ref{lem:inclusion_of_full_cofinal_subcategory:regular_coherent}
  and assertion~\ref{lem:regular_coherent_and_unions:idempotent_complete},
  we can assume without loss of generality
  that each $\cala_i$ and $\cala$ are idempotent complete.
  Now the claim follows from the
  equivalence~\ref{lem:intrinsic_reformulation_of_regular_coherent:uniform_l_is_0:(1)}
  $\Longleftrightarrow$~\ref{lem:intrinsic_reformulation_of_regular_coherent:uniform_l_is_0:(3)}
 appearing in Lemma~\ref{lem:intrinsic_reformulation_of_regular_coherent}~%
\ref{lem:intrinsic_reformulation_of_regular_coherent:uniform_l_is_0}.
\end{proof}

\begin{lemma}\label{lem:regular_coherent_and_products}
  Let $l$ be a natural number. Let $\cala = \{\cala_i \mid i \in I\}$ be a
  collection of  additive categories
  $\cala_i$ for an arbitrary index set $I$.  Let $l$ be a natural number.

  \begin{enumerate}
  \item\label{lem:regular_coherent_and_products:products}
    Suppose that each $\cala_i$ is
    $l$-uniformly regular coherent. Then $\prod_{i \in I} \Idem(\cala_i)$ is
    $l$-uniformly regular coherent;

  \item\label{lem:regular_coherent_and_products:sums}
    Suppose that each $\cala_i$ is
    $l$-uniformly regular coherent, $l$-uniformly regular,  regular coherent, regular, or Noetherian.
    Then $\bigoplus_{i \in I} \cala_i$ has the same property.
  \end{enumerate}
\end{lemma}
\begin{proof}
    Obviously $\prod_{i \in I} \cala_i$ inherits the structure of an additive category.
  Recall that $\bigoplus_{i \in I} \cala_i$ is the full additive subcategory of
  $\prod_{i \in I}\cala_i$ consisting of those objects $A_i \mid i \in I\}$, for which only
  finitely many of the objects $A_i$ are different from zero. Obviously
  \begin{eqnarray*}
    \Idem(\bigoplus_{i \in I} \cala_i) & \cong &  \bigoplus_{i \in I} \Idem(\cala_i);
    \\
    \Idem(\prod_{i \in I} \cala_i)& \cong &  \prod_{i \in I} \Idem(\cala_i).
  \end{eqnarray*}
  Lemma~\ref{lem:intrinsic_reformulation_of_regular_coherent} implies
  that $\bigoplus_{i \in I} \Idem(\cala_i)$ and $\prod_{i \in I} \Idem(\cala_i)$ are $l$-uniformly
  regular coherent, if each $\Idem(\cala_i)$ is  $l$-uniformly
  regular coherent.  We conclude from  Lemma~\ref{lem:inclusion_of_full_cofinal_subcategory}~%
\ref{lem:inclusion_of_full_cofinal_subcategory:regular_coherent} 
  that $\bigoplus_{i \in I} \cala_i$ and $\prod_{i \in I} \cala_i$ are $l$-uniformly
  regular coherent, if each $\cala_i$ is  $l$-uniformly regular coherent.

  Analogously one shows that $\bigoplus_{i \in I} \cala_i$ is regular coherent
  if each $\cala_i$ is regular coherent. 

  Next we show that $\bigoplus_{i \in I} \cala_i$ is Noetherian
  if each $\cala_i$ is Noetherian.
  Consider any object $A$ in $\bigoplus_{i \in I} \cala_i$. Choose a finite index set
  $J \subseteq I$ such that $A$ belongs to $\bigoplus_{i \in J} \cala_i$. Let
  $\pr \colon \bigoplus_{i \in I} \cala_i \to \bigoplus_{i \in J} \cala_i$ be the
  projection. Consider two morphism $f_0 \colon A_0 \to A$ and $f_1 \colon A_1 \to A$ in
  $\bigoplus_{i \in I} \cala_i$. Then $f_0 = f_1$ holds if and only if $\pr(f_0) = \pr(f_1)$
  holds.  Hence $f_0 \subseteq f_1$ holds in $\bigoplus_{i \in I} \cala_i$ if and only if
  $\pr(f_0) \subseteq \pr(f_1)$ holds in $\bigoplus_{i \in J} \cala_i$.  We conclude from
  Lemma~\ref{lem:intrinsic_reformulation_of_Noetherian} that it suffices to show that
  $\bigoplus_{i \in J} \cala_i$ is Noetherian for any finite subset $J \subseteq I$.  But
  this is an easy consequence of Lemma~\ref{lem:intrinsic_reformulation_of_Noetherian}
  again. \end{proof}

Lemma~\ref{lem:regular_coherent_and_products}~\ref{lem:regular_coherent_and_products:products}
will be generalized in Lemma~\ref{lem:inverse_systems_and_regularity}.

\begin{remark}[Advantage of the notion $l$-uniformly regular coherent]%
  \label{lem:Advantage_of_the_notion_l-uniformly_regular_coherent}
  The decisive advantage of the notion $l$-uniformly regular coherent is that it satisfies
  both Lemma~\ref{lem:regular_coherent_and_unions} and
  Lemma~\ref{lem:regular_coherent_and_products}.
  Lemma~\ref{lem:regular_coherent_and_unions} and
  Lemma~\ref{lem:regular_coherent_and_products}~\ref{lem:regular_coherent_and_products:products}
  are not true, if one replaces $l$-uniformly regular coherent by any of the properties
  regular coherent, $l$-uniformly regular, regular or Noetherian, unless   $I$ is finite.
\end{remark}

    %%%%%%%%%%%%%%%%%%%%%%%%%%%%%%%%%%%%%%%%%%%%%%%%%%%%%%%%%%%%%%%%%%%%
%%%%%%%%%%%%%%%%%%%%%%%%%%%%% Section 11   %%%%%%%%%%%%%%%%%%%%%%%%%%%%%%
%%%%%%%%%%%%%%%%%%%%%%%%%%%%%%%%%%%%%%%%%%%%%%%%%%%%%%%%%%%%%%%%%%%%

\typeout{---------------------- Vanishing of negative K-groups  -----------------------}

\section{Vanishing of negative $K$-groups}%
\label{sec:Vanishing_of_negative_K-groups}

\begin{theorem}[Vanishing of negative $K$-groups]%
\label{the:Vanishing_of_negative_K-groups}
Let $\cala$ be an additive category, such that $\cala[t_1,t_2, \ldots, t_m]$ is regular
coherent for every $m \ge 0$.
  
Then $K_n(\cala) = 0$ holds for all $n \le -1$.
\end{theorem}

\begin{proof} 
   
  For an additive category $\calb$ define $G'_0(\IZ\calb)$ to be the abelian group with
  isomorphism classes $[M]$ of finitely presented $\IZ\calb$-modules $M$ as generators
  such that for each exact sequence of finitely presented $\IZ\calb$-modules
  $0 \to M_0 \to M_1 \to M_2 \to 0$ we have the relation $[M_0] -[M_1] + [M_2] =
  0$. Define $K_0(\IZ\calb)$ analogously but with finitely presented replaced by finitely
  generated projective. A functor of additive categories $F\colon \calb \to \calb'$
  induces a homomorphism $F_* \colon K_0(\IZ\calb) \to K_0(\IZ\calb')$ by sending $[M]$ to
  $[F_*M]$.  It induces a homomorphism $F_* \colon G'_0(\IZ\calb) \to G'_0(\IZ\calb')$ by
  sending $[M]$ to $[F_*M]$, if
  $F_* \colon \MODcatr{\IZ\cala}_{\calb} \to \MODcatr{\IZ\cala}_{\calb'}$ is flat. There
  is the forgetful functor $U \colon K_0(\IZ\calb) \to G'_0(\IZ\calb')$.  If $\calb$ is
  regular coherent, then $U$ is a bijection by the Resolution Theorem,
  see~\cite[Theorem~4.6 on page~41]{Srinivas(1991)}.  The Yoneda embedding induces an
  isomorphism $K_0(\calb) \xrightarrow{\cong} K_0(\IZ\calb)$, natural in $\calb$.
  
  Suppose $\cala[t]$ is regular coherent. We show that $\cala[t,t^{-1}]$ is regular
  coherent and $K_{-1}(\cala) = 0$.  The functor
  $j_* \colon \MODcatr{\IZ\cala[t]} \to \MODcatr{\IZ\cala{[t,t^{-1}]}}$ is flat by
  Lemma~\ref{lem:j_ast_is_flat}.  Let $M_*$ be a finitely presented
  $\IZ\cala[t,t^{-1}]$-module.  Then we can find a morphism $f \colon A \to A'$ in
  $\cala[t,t^{-1}]$ together with an exact sequence of $\IZ\cala[t,t^{1-}]$-modules
  \[\mor_{\cala[t,t^{-1}]}(?,A) \xrightarrow{f_*} \mor_{\cala[t,t^{-1}]}(?,A') \to M \to
    0.
  \]
  Choose a natural number $s$ and a morphism $g \colon A \to A'$ in $\cala[t]$ such that
  $(\id_{A'} \cdot t^s) \circ f = g$ holds in $\cala[t,t^{-1}]$. Since
  $\id_{A'} \cdot t^s \colon A' \xrightarrow{\cong} A'$ is an isomorphism in
  $\cala[t,t^{-1}]$, we obtain an exact sequence of $\IZ\cala[t,t^{-1}]$-modules
  \[
    j_*(\mor_{\cala[t]}(?,A)) \xrightarrow{j_*(g_*)} j_*(\mor_{\cala[t]}(?,A')) \to M \to
    0.
  \]
  Let $N$ be the finitely presented $\IZ\cala[t]$-module, which is the cokernel of the
  $\IZ\cala[t]$-homomorphism $g_* \colon \mor_{\cala[t]}(?,A) \to \mor_{\cala[t]}(?,A')$.
  Since $j_*$ is flat and in particular right exact, we obtain an isomorphism of finitely
  presented $\IZ\cala[t,t^{-1}]$-modules $j_*N \xrightarrow{\cong} M$. This implies that
  the homomorphism $j_* \colon G'_0(\IZ\cala[t]) \to G'_0(\IZ\cala[t,t^{-1}])$ is
  surjective and that $\cala[t,t^{-1}]$ is regular coherent since $\cala[t]$ is regular
  coherent by assumption.

  Hence we obtain a commutative diagram
  \[\xymatrix{K_0(\cala[t]) \ar[r] \ar[d]_{\cong} & K_0(\cala[t,t^{-1}]) \ar[d]^{\cong}
      \\
      K_0(\IZ\cala[t]) \ar[r] \ar[d]_{\cong} & K_0(\IZ\cala[t,t^{-1}]) \ar[d]^{\cong}
      \\
      G'_0(\IZ\cala[t]) \ar[r] & G'_0(\IZ\cala[t,t^{-1}]),}
  \]
  whose vertical arrows are bijections and whose lowermost horizontal arrow is surjective.
  Hence the uppermost horizontal arrow is surjective. We conclude from
  Theorem~\ref{the:Fundamental_sequence_of_K-groups} that $K_{-1}(\cala)$ vanishes.

  Next we show by induction for $n = 1,2 \ldots$ that $K_{-m}(\cala)$ vanishes for
  $m = 1,2, \ldots, n$.  The induction beginning $n = 1$ has been taken care of above. The
  induction step from $n \ge 1 $ to $n+1$ is done as follows. One shows using the claim
  above by induction for $i = 1, 2, \ldots, n$ that $\cala[\IZ^i][t_{i+1},\dots, t_{n+1}]$
  is regular coherent.  In particular $\cala[\IZ^n][t_{n+1}]$ is regular coherent.

  We conclude from the $n$-times iterated Bass-Heller-Swan isomorphism, see
  Theorem~\ref{the:BHS_decomposition_for_non-connective_K-theory} that $K_{-n-1}(\cala)$
  is a direct summand in $K_{-1}(\cala[\IZ^n])$. Hence it suffices to show that
  $K_{-1}(\cala[\IZ^n])$ is trivial.  This follows from the induction beginning applied to
  $\cala[\IZ^n]$.
\end{proof}

We conclude from Theorem~\ref{the:Regularity_for_cala_and_cala[t]}
and Theorem~\ref{the:Vanishing_of_negative_K-groups}

\begin{corollary}[Vanishing of negative $K$-groups of regular additive categories]%
\label{the:Vanishing_of_negative_K-groups_of_regular_additive_categories}
Let $\cala$ be an additive category which is regular.

Then $K_n(\cala) = 0$ holds for all $n \le -1$.
\end{corollary}

\begin{remark}\label{rem:Schlichting}
  As noted in Lemma~\ref{lem:Yoneda_equivalence} the Yoneda embedding identifies $\cala$
  with the category of finitely generated free $\IZ\cala$-modules.  If $\cala$ is
  Noetherian, then the category of finitely generated $\IZ\cala$-modules is abelian.  If
  it is in addition regular coherent (i.e., if $\cala$ is regular), then $\cala$ is
  derived equivalent to this abelian category.  Schlichting showed in~\cite[Theorem~6 on
  page~117]{Schlichting(2006)} that $K_{-1}$ of abelian categories is trivial.  It follows
  that for regular $\cala$ we can obtain Theorem~\ref{the:Vanishing_of_negative_K-groups}
  from Schlichting's result.  Similarly,
  Corollary~\ref{the:Vanishing_of_negative_K-groups_of_regular_additive_categories} can
  alternatively be deduced from~\cite[Theorem~7 on page~118]{Schlichting(2006)}.
\end{remark}

%%%%%%%%%%%%%%%%%%%%%%%%%%%%%%%%%%%%%%%%%%%%%%%%%%%%%%%%%%%%%%%%%%%%
%%%%%%%%%%%%%%%%%%%%%%%%%%%%% Section 12  %%%%%%%%%%%%%%%%%%%%%%%%%%%%%%%
%%%%%%%%%%%%%%%%%%%%%%%%%%%%%%%%%%%%%%%%%%%%%%%%%%%%%%%%%%%%%%%%%%%%

\typeout{-----   Section 12: Nested sequences, the associated categories, and their K-theory ---------}

\section{Nested sequences, the associated categories, and their $K$-theory}%
\label{sec:Nested_sequences,the_associated_categories,and_their_K-theory}

%%%%%%%%%%%%%%%%%%%%%%%%%%%%%%%%%%%%%%%%%%%%%%%%%%%%%%%%%%%%%%%%%%%%

\subsection{Nested sequences and the associated categories}%
\label{subsec:inverse_systems_and_the_associated_sequence_category}

\begin{definition}[Nested sequences of additive categories]%
\label{def:inverse_system_of_additive_categories}
A \emph{nested sequence of additive categories} $\cala_*$ is a decreasing 
sequence of additive subcategories
\[
\cala_0 \supseteq \cala_1 \supseteq \cala_2 \supseteq \cdots.
\]
\end{definition}

We have two notions of morphisms.
\begin{definition}[(Pro-)morphisms of nested sequences of additive categories]%
\label{def:(Pro-)morphisms_of_nested_sequences_of_additive_categories}
A \emph{morphism of nested sequences of additive categories}
$F \colon \cala_* \to \cala_*'$ is a sequence of functors of additive categories
$F_m \colon \cala_m \to \cala_m'$ for $m \in \IN$ such that $F_m$ restricted to
$\cala_{m+1}$ is $F_{m+1}$.

A \emph{pro-morphism of nested sequences of additive categories} $F \colon \cala_* \to \cala'_*$
is a functor of additive categories  $F \colon \cala_0 \xrightarrow{\cong} \cala'_0$ such that there is a
function $N \colon \IN \to \IN$ with the property that for every $m,n \in \IN$ with
$m \ge N(n)$ we have $F(\cala_m) \subseteq \cala'_n$.  
\end{definition}

Obviously a morphism $F_*$ defines a pro-morphisms by taking $F_0$, but
not every pro-morphism comes from a morphism in this way. The
composite of two morphisms and of two pro-morphisms is defined in
the obvious way.

Note that a pro-automorphism $\phi \colon \cala_* \to \cala_*$ is the same as an
automorphism of additive categories $\Phi \colon \cala_0 \xrightarrow{\cong} \cala_0$ such
that there is a function $N \colon \IN \to \IN$ with the property that for every
$m,n \in \IN$ with $m \ge N(n)$ we have $\Phi(\cala_m) \subseteq \cala_n$ and
$\cala_{m} \subseteq \Phi(\cala_{n})$.

\begin{definition}[The sequence category $\cals(\cala_*)$  and
  the limit category $\call(\cala_*)$]\label{def:calo_upper_0(cals_ast)}
  Define the additive category $\cals(\cala_*)$, called \emph{sequence category},
  associated to the nested sequence of additive categories $\cala_*$ as follows:

\begin{itemize}
\item An object in $\cals(\cala_*)$ is a sequence
  $\underline{A} = (A_m)_{m \in \IN}$ of objects in $\cala_0$ such that
  there exists a function (depending on $\underline{A}$)
  $L \colon \IN \to \IN$ with the property that $A_m$ belongs to $\cala_l$ for
  $l,m \in \IN$ with $m \ge L(l)$;

\item A morphism $\underline{\varphi} \colon \underline{A} \to \underline{A'}$ in
  $\cals(\cala_*)$ consists of a sequence of morphisms $\varphi_m \colon A_m \to A'_m$ in
  $\cala_0$ for $m \in \IN$ such that there exists a function $L \colon \IN \to \IN$ with the property
  that $\varphi_m \colon A_m \to A_m'$ belongs to $\cala_l$ for $m,l \in \IN$ with
  $m \ge L(l)$;

\item Composition and the structure of an additive category on $\cals(\cala_*)$
comes from the corresponding structure on $\cala_0$.

\end{itemize}

Let $\calt(\cala_*)$ be the full subcategory of $\cals(\cala_*)$ consisting of objects
$\underline{A}$, for which there exists a natural number $M$ (depending on $\underline{A}$)
with $A_m = 0$ for $m \ge M$.

The additive category $\call(\cala_*)$, called \emph{limit category}, is defined to be the
additive quotient category $\cals(\cala_*)/\calt(\cala_*)$.  Recall that for a full additive subcategory
$\calu \subseteq \calb$ of an additive category $\calb$  the additive quotient category $\calb/\calu$ has
the same objects as $\calb$ and that morphisms in $\calb/\calu$ are equivalence classes of
morphisms in $\calb$, where two morphisms $\varphi,\varphi'$ are identified, if their difference
$\varphi-\varphi'$ can be factorized through an object of $\calu$.  For $\call(\cala_*)$ this
means that morphisms $\varphi,\varphi'$ from $\cals(\cala_*)$ are identified in $\call(\cala_*)$,
if and only if  $\varphi_m=\varphi'_m$ for all but finitely many $m$.
\end{definition}

Obviously $\cals(\cala_*)$ is an additive subcategory of $\prod_{m \in \IN} \cala_0$ and
is equal to it, if the nested sequence is constant, i.e., $\cala_0 = \cala_m$ for
$m \in \IN$.  Obviously $\calt(\cala_*)$ can be identified with
$\bigoplus_{m \in \IN}\cala_m$.  Controlled categories appear for instance in proofs of
the Farrell-Jones Conjecture.  In for us important cases controlled categories correspond
to nested sequences of additive categories, where one should think of the control
condition to become sharper the larger the index $m$ gets.

For the notion of a Karoubi filtration and the associated  weak homotopy fibration sequence
we refer for instance to~\cite{Cardenas-Pedersen(1997)},
and~\cite[Definition~5.4]{Kasprowski(2015findeccom)} or~\cite[Section~12.2]{Lueck(2022book)}.
One easily checks

\begin{lemma}\label{lem:Karoubi_filtration_sequence_categories}
  The inclusion $\calt(\cala_*) \subseteq \cals(\cala_*)$ is a Karoubi filtration and we
  have the  weak homotopy fibration sequence 
\[
  \bfKinfty(\calt(\cala_*)) \to \bfKinfty(\cals(\cala_*)) \to \bfKinfty(\call(\cala_*)).
\]
\end{lemma}

Let $F \colon \cala_* \to \cala'_*$ be a pro-morphisms. It induces functors of additive
categories
\begin{eqnarray}
  \cals(F) \colon \cals(\cala_*)  & \to & \cals(\cala'_*);
  \label{cals(F)}
  \\
  \calt(F) \colon \calt(\cala_*) & \to & \calt(\cala'_*);
  \label{calt(F)}
  \\
  \call(F) \colon \call(\cala_*) & \to & \call(\cala'_*),
  \label{call(F)}
\end{eqnarray}
as follows. We begin with $\cals(F)$. It sends an object
$\underline{A} = (A_m)_{m \in \IN}$ to the object
$\cals(F)(\underline{A}) = (F(A_m))_{m \in \IN}$. We have to check that the latter
collection defines an element in $\cals(\cala'_*)$. Recall that there is a function
$L \colon \IN \to \IN$ with the property that $A_m$ belongs to $\cala_l$ for $l,m \in \IN$
with $m \ge L(l)$ and that there is a function $N \colon \IN \to \IN$ with the property
that for every $m,n \in \IN$ with $m \ge N(n)$ we have $F(\cala_m) \subseteq \cala_n$.
Consider the function $L \circ N \colon \IN \to \IN$.  For $m,n \in \IN$ with
$m \ge L \circ N(n)$ we conclude $A_m \in \cala_{N(n)}$ and hence
$F(A_m) \in F(\cala_{N(n)}) \subseteq \cala_n$. Hence $(F(A_m))_{m \in \IN}$ is a
well-defined object in $\cals(\cala_*)$.

Given a morphism $\underline{\varphi} \colon \underline{A} \to \underline{A'}$ in
$\cals(\cala_*)$, we can define
$\cals(F)(\underline{\varphi}) \colon \cals(F)(\underline{A}) \to
\cals(F)(\underline{A'})$ in $\cals(\cala'_*)$ by the collection
$(F(\varphi_m))_{m \in \IN}$. Obviously $\cals(F' \circ F) = \cals(F') \circ \cals(F)$ and
$\cals(\id_{\underline{A}}) = \id_{\underline{A}}$ holds.  By construction $\cals(F)$
induces a functor of additive categories
$\calt(F) \colon \calt(\cala_*) \to \calt(\cala'_*)$. By passing to the quotients we also
get a functor of additive categories $\call(F) \colon \call(\cala_*) \to \call(\cala'_*)$.

Note that  a pro-automorphism
$\Phi \colon \cala_* \to \cala_*$ induces automorphisms of additive categories
\begin{eqnarray*}
\cals(\Phi) \colon \cals(\cala_*) &\xrightarrow{\cong} & \cals(\cala_*);
  \\
  \calt(\Phi) \colon \calt(\cala_*) &\xrightarrow{\cong} & \calt(\cala_*);
  \\
  \call(\Phi) \colon \call(\cala_*) &\xrightarrow{\cong} & \call(\cala_*).
\end{eqnarray*}

\begin{definition}\label{def:admissible_function}
  We call a function $I \colon \IN \to \IN$ \emph{admissible}, if it has the following properties
  \[
   \begin{array}{rclcl}
    I(m) & \le & m & & \text{for}\; m \in \IN;
    \\
    I(m) & \le & I(m+1)  & & \text{for}\; m \in \IN;
    \\
    \lim_{m \to \infty} I(m) & =   & \infty.
   \end{array}
   \]
   Let $\cali$ be the set of admissible functions. It becomes a directed set by defining for $I,J \in \cali$
   \[I \le J \Longleftrightarrow I(m) \ge J(m) \; \text{for all}\; m \in \IN.
   \]
 \end{definition}

 Note that $\cali$ is indeed directed. For $I,J \in \cali$ define
 $K \colon \IN \to \IN$ by $K(m) = \min\{I(m),J(m)\}$.
 Then $K \in \cali$ and $I,J \le K$ holds.

 \begin{lemma}\label{lem:admissible_functions_exists_for_calo_and_its_variants}
   \begin{enumerate}
   \item\label{lem:admissible_functions_exists_for_calo_and_its_variants:phi_in_calo_upper_0(A_ast)_implies_I_exists}
     Let $\underline{\phi}$ be a morphism in $\cals(\cala_*)$. Then there exists an admissible function
     $I \in \cali$ such that $\phi_m \in \cala_{I(m)}$ holds for all $m \in \IN$;

     \item\label{lem:admissible_functions_exists_for_calo_and_its_variants:I_exists_implies_phi_in_calo_upper_0(A_ast)}
       Let $\underline{\phi}$ be a sequence of morphisms $\phi_m \colon A_m \to A_m'$ in $A_0$.
       Suppose that there exists an admissible function
       $I \in \cali$ such that $\phi_m \in \cala_{I(m)}$ holds for all $m \in \IN$. Then $\underline{\phi}$ belongs to
       $\cals(\cala_*)$.
     \end{enumerate}
   \end{lemma}
   \begin{proof}~\ref{lem:admissible_functions_exists_for_calo_and_its_variants:phi_in_calo_upper_0(A_ast)_implies_I_exists}
    Choose a function $L \colon \IN \to \IN$ such that $\phi_m$ belongs to $\cala_l$ if
    $m \ge L(l)$. Define a new function $I' \colon \IN \to \IN$ by
  \[
    I'(m) = \max\{i \in \{0,1,\ldots, m\} \mid 
   \phi_m \in  \cala_i\}.
  \]
  It satisfies
  \[
   \begin{array}{lclcl}
    I'(m) & \le & m & & \text{for}\; m \in \IN;
    \\
    l & \le  & I'(m)  & & \text{for}\; l,m \in \IN, m \ge L(l), m \ge l;
    \\
    \phi_m  & \in & \cala_{I'(m)} & & \text{for}\;m \in \IN.
   \end{array}
   \]
     Define the function $I \colon \IN \to \IN$ by
   \[I(m) = \min\{I'(j) \mid j \in \IN, m \le j\}.
   \]
   Then we get for all $n \in \IN$
  \[
   \begin{array}{lclcl}
    I(m) & \le & m & & \text{for}\; m \in \IN;
    \\
    I(m) & \le & I(m+1)  & & \text{for}\; m \in \IN;
    \\
    l & \le  & I(m)  & & \text{for}\; l,m, \in \IN, m \ge L(l), m \ge l;
    \\
    \phi_m  & \in & \cala_{I(m)} & & \text{for}\; m \in \IN.
   \end{array}
   \]
   The first three properties imply $I \in \cali$.
   \\[1mm]~\ref{lem:admissible_functions_exists_for_calo_and_its_variants:I_exists_implies_phi_in_calo_upper_0(A_ast)}
   Suppose that there exists $I \in \cali$ satisfying $\phi_m \in \cala_{I(m)}$ for all $m \in \IN$.
   Define the desired function $L \colon \IN \to \IN$ by $L(l) = \min\{m \in \IN \mid l \le I(m)\}$.
    \end{proof}

 %%%%%%%%%%%%%%%%%%%%%%%%%%%%%%%%%%%%%%%%%%%%%%%%%%%%%%%%%%%%%%%%%%%% 

\subsection{Uniform regular coherence}%
\label{subsec:Uniform_regular_coherence}

\begin{lemma}\label{lem:inverse_systems_and_regularity}
Consider the nested sequence $\cala_*$ of additive  categories
$\cala_0 \supseteq \cala_1 \supseteq \cala_2 \supseteq \cdots$.
Suppose that for the natural number $l \ge 2$ each of the additive categories $\cala_m$ is $l$-uniformly
regular coherent and that the inclusion $\cala_m \to \cala_{m+1}$ is flat for $m \in \IN$.

Then $\cals(\cala_*)$ and $\call(\cala_*)$ are $l$-uniformly regular coherent.
\end{lemma} 
\begin{proof}
We first treat $\cals(\cala_*)$.
Let $\underline{\phi}^1 \colon \underline{A}^1 \to \underline{A}^0$ be a morphism in $\cals(\cala_*)$.
Because of Lemma~\ref{lem:admissible_functions_exists_for_calo_and_its_variants}~%
\ref{lem:admissible_functions_exists_for_calo_and_its_variants:phi_in_calo_upper_0(A_ast)_implies_I_exists}
we can choose $I \in \cali$ with $A^{1}_m, A^{0}_m, \phi^{1}_m  \in  \cala_{I(m)}$.
By assumption we can find for each $m\in \IN$ an exact sequence
  \[
    0 \to A^{l}_m \xrightarrow{\phi^{l}_m} A^{l-1}_m  \xrightarrow{\phi^{l}_m} \cdots
    \xrightarrow{\phi^{2}_m} A^{1}_m \xrightarrow{\phi^{1}_1} A^{0}_m
  \]
  in $\cala_{I(m)}$. We conclude from
  Lemma~\ref{lem:admissible_functions_exists_for_calo_and_its_variants}~%
\ref{lem:admissible_functions_exists_for_calo_and_its_variants:I_exists_implies_phi_in_calo_upper_0(A_ast)}
  that the collection of these sequences for $m = 0,1,2 \ldots $
  defines a sequence in $\cals(\cala_*)$
  \begin{equation}
    0 \to \underline{A}^{l} \xrightarrow{\underline{\phi}^{l}} \underline{A}^{l-1}
    \xrightarrow{\underline{\phi}^{l-1}}  \cdots \xrightarrow{\underline{\phi}^{2}} \underline{A}^{1}
    \xrightarrow{\underline{\phi}^{1}} \underline{A}^{0}.
   \label{some_resolution_of_phi_upper_1}
  \end{equation}
Finally we show that the sequence~\eqref{some_resolution_of_phi_upper_1} is exact as a
  sequence in $\cals(\cala_*)$.  We have to solve for every
  $ j \in \{1, \ldots, l\}$ the following lifting problem in $\cals(\cala_*)$.
  \begin{equation}
   \label{Wagner}
    \xymatrix{\underline{A}^{j+1} \ar[r]^-{\underline{\phi}^{j+1}}
      &
      \underline{A}^{j}\ar[r]^-{\underline{\phi}^{j}}
      &
      \underline{A}^{j-1}
      \\
      & 
      \underline{B}. \ar@{-->}[ul]^{\underline{\nu}} \ar[u]^{\underline{\mu}} \ar[ru]_{0}
      &
    }
  \end{equation}
  Because of Lemma~\ref{lem:admissible_functions_exists_for_calo_and_its_variants}~%
\ref{lem:admissible_functions_exists_for_calo_and_its_variants:phi_in_calo_upper_0(A_ast)_implies_I_exists}
  we can choose  $J \in \cali$  such $B_m, \mu_m   \in  \cala_{J(m)}$ holds.
  Choose $K \in \cali$ with $I, J \le K$.
  Now consider the following lifting problem in $\cala_{K(m)}$
  \begin{equation}
    \label{Neuer}
    \xymatrix{A^{j+1}_m  \ar[r]^-{\phi^{j+1}_m}
      &
      A^{j}_m\ar[r]^-{\phi^{j}_m}
      &
      A^{j-1}_m
      \\
      & 
      B_m.\ar@{-->}[ul]^{\nu_m} \ar[u]^{\mu_m} \ar[ru]_{0}
      &
    }
  \end{equation}
  As the inclusion
  $\cala_{I(m)} \to \cala_{K(m)}$ is flat by assumption, and the sequence
  $A^{j+1}_m \xrightarrow{\phi^{j+1}_m} A^{j}_m
  \xrightarrow{\phi^{j}_m} A^{j-1}_m$ is by construction exact at
  $A^{j}_m$, when considered in $\cala_{I(m)}$, it is exact at
  $A^{j}_m$, when considered in $\cala_{K(m)}$.
  Hence~\eqref{Neuer} has a solution $\nu_m \colon B_m \to A^{j+1}$
  when considered in $\cala_{K(m)}$.  We conclude from
  Lemma~\ref{lem:admissible_functions_exists_for_calo_and_its_variants}~%
\ref{lem:admissible_functions_exists_for_calo_and_its_variants:I_exists_implies_phi_in_calo_upper_0(A_ast)}
  that the collection of the morphisms $\nu_m$ yields a morphism
  $\underline{\nu} \colon \underline{B}\to \underline{A}^{j+1}$ 
 in $\cals(\cala_*)$.   Therefore $\overline{\nu} $ is a solution to the lifting
  problem~\eqref{Wagner} in $\cals(\cala_*)$.  We conclude that~\eqref{some_resolution_of_phi_upper_1} is an
  exact sequence in $\cals(\cala_*)$. This finishes the proof of
  Lemma~\ref{lem:inverse_systems_and_regularity} for $\cals(\cala_*)$.

The proof for $\call(\cala_*)$ is the following modification of the one for $\cals(\cala_*)$.
Let $\overline{\underline{\phi}^1} \colon \underline{A}^1 \to \underline{A}^0$ be a 
morphism in $\call(\cala_*)$. Choose a representative
$\underline{\phi}^1 \colon \underline{A}^1 \to \underline{A}^0$ in $\cals(\cala_*)$.
Now one proceeds as above and constructs the sequence~\eqref{some_resolution_of_phi_upper_1}
in $\cals(\cala_*)$. However, instead of solving the lifting problem~\eqref{Wagner} in
$\cals(\cala_*)$ we have to solve the lifting problem 
  \begin{equation}
   \label{Wagner_cald}
    \xymatrix{\underline{A}^{j+1} \ar[r]^-{\overline{\underline{\phi}^{j+1}}}
      &
      \underline{A}^{j}\ar[r]^-{\overline{\underline{\phi}^{j}}}
      &
      \underline{A}^{j-1}
      \\
      & 
      \underline{B} \ar@{-->}[ul]^{\underline{\nu}} \ar[u]^{\overline{\underline{\mu}}} \ar[ru]_{0}
      &
    }
  \end{equation}
in $\call(\cala_*)$.  Choose a representative $\underline{\mu}$ for $\overline{\underline{\mu}}$.
There is a natural number $M$ such that $\phi^{j}_m \circ \mu_m = 0$ holds for
$m \ge M$. We can change the representative $\underline{\mu}$
by putting $\mu_m = 0$ for $m < M$ and by leaving $\mu_m$ unchanged for $m \ge M$.
Now we choose a solution $\nu_m$ to the lifting problem~\eqref{Neuer} in $\cala_{K(m)}$ for $m \ge M$.
Put $\nu_m = 0$ for $m < M$. Then we get a morphism 
$\underline{\nu} \colon \underline{B} \to \underline{A}^{j+1}$
in $\cals(\cala_*)$ such that its class $\overline{\underline{\nu}} \colon \underline{B} 
\to \underline{A}^{j+1}$ in $\call(\cala_*)$
is a solution to the lifting problem~\eqref{Wagner_cald}. This finishes the proof of
Lemma~\ref{lem:inverse_systems_and_regularity}.
\end{proof}

\begin{example}[The property Noetherian does not pass to the sequence category]%
\label{exa:The_property_Noetherian_does_not_pass_to_the_sequence_category}
The analogue of Lemma~\ref{lem:inverse_systems_and_regularity} for the properties
Noetherian, regular,  or $l$-uniformly regular instead of uniformly $l$-regular coherent does not hold as the
following example shows.  Suppose that none of the $\cala_m$ is the trivial additive
category.  Consider an object $\underline{A}$ of $\cals(\cala_*)$ such that
$A_m \not= \{0\}$ for $m \in \IN$, and the $\IZ\cals(\cala_*)$-module
\[
  F = \mor_{\cals(\cala_*)}(\underline{?},\underline{A}).
\]
Define a $\IZ\cals(\cala_*)$-submodule $V$ of $F$ by
\[
  V(?) = \{\underline{\phi} \in F(?) \mid \exists M(\underline{\phi}) \in \IN \:
  \text{with}\; \phi_m = 0 \;\text{for}\; m \ge M(\underline{\phi})\}.
\]
Suppose that there exists an object $\underline{B}$ and an epimorphism
$f \colon \mor_{\cals(\cala_*)}(\underline{?},\underline{B}) \to V$.  If we write
$f(\id_{\underline{B}}) = \underline{\psi} \in V(\underline{B})$, then there must be a natural
number $M$ with $\psi_m = 0$ for $m \ge M$. This implies that for any
$\underline{\phi} \in V(?)$ we have $\phi_m = 0$ for $m \ge M$.  This is a contradiction,
since $M$ does not depend on $\underline{\phi}$.  Hence $V$ is not finitely generated and
$\cals(\cala_*)$ is not Noetherian. This construction yields also a counterexample for
$\call(\cala_*)$.
\end{example}

   %%%%%%%%%%%%%%%%%%%%%%%%%%%%%%%%%%%%%%%%%%%%%%%%%%%%%%%%%%%%%%%%%%%% 

\subsection{The algebraic $K$-theory of the sequence categories 
$\cals(\cala_*)$, $\calt(\cala_*)$ and $\call(\cala_*)$}%
\label{subsec:The_algebraic_K-theory_of_the_sequence_categories_with_upper_0}

Given $I \in \cali$, define a subcategory
    $\cals(\cala_*)_I$ of $\cals(\cala_*)$ as follows.  An object $\underline{A}$ in
    $\cals(\cala_*)$ belongs to $\cals(\cala_*)_I$, if $A_m \in \cala_{I(m)}$ holds for
    $m \in \IN$.  A morphism $\underline{\phi} \colon \underline{A} \to \underline{T}$ in
    $\cals(\cala_*)$ belongs to $\cals(\cala_*)_I$, if and only if $\phi_m \in \cala_{I(m)}$ 
    holds for $m \in \IN$.

    Next we show
    \begin{eqnarray}
    \cals(\cala_*)_I &\subseteq & \cals(\cala_*)_J \quad \text{for}\; I,J \in \cali, I \le J;
    \label{calo_upper_0(A_ast)_I_subseteq_calo_upper_0(A_ast)_J_for_I_le_J}
      \\
      \cals(\cala_*) & = & \bigcup_{I\in \cali} \cals(\cala_*)_I.
      \label{calo_upper_0(A_ast)_is_union_of_the_calo_upper_0(A_ast)_I}
    \end{eqnarray}
    The first equation is obvious since $\cala_i \subseteq \cala_j$ for $i \ge j$.
    The second follows from Lemma~\ref{lem:admissible_functions_exists_for_calo_and_its_variants}.

    Define analogously the subcategory $\calt(\cala_*)_I$ of $\calt(\cala_*)$. Then we get
    \begin{eqnarray}
    \calt(\cala_*)_I &\subseteq & \calt(\cala_*)_J \quad \text{for}\; I,J \in \cali, I \le J;
    \label{calt_upper_0(A_ast)_I_subseteq_calt_upper_0(A_ast)_J_for_I_le_J}
      \\
      \calt(\cala_*) & = & \bigcup_{I\in \cali} \calt(\cala_*)_I.
      \label{calt_upper_0(A_ast)_is_union_of_the_calt_upper_0(A_ast)_I}
    \end{eqnarray}

    One easily checks that the inclusion $\calt(\cala_*)_I \subseteq \cals(\cala_*)_I$ is Karoubi filtration. Define
    \begin{eqnarray}
        \call(\cala_*)_I & = & \cals(\cala_*)_I/\calt(\cala_*)_I.
      \label{cald_upper_0(cals_ast)_I}
    \end{eqnarray}
    Then we get
    \begin{eqnarray}
    \call(\cala_*)_I &\subseteq & \call(\cala_*)_J \quad \text{for}\; I,J \in \cali, I \le J;
    \label{call_upper_0(A_ast)_I_subseteq_call_upper_0(A_ast)_J_for_I_le_J}
      \\
      \call(\cala_*) & = & \bigcup_{I\in \cali} \call(\cala_*)_I.
      \label{call_upper_0(A_ast)_is_union_of_the_call_upper_0(A_ast)_I}
    \end{eqnarray}

    \begin{lemma}\label{lem:hocolim_I_K(cald(A_ast)_I)_simeq_K(cald(A_ast))}
      \
      \begin{enumerate}
      \item\label{lem:hocolim_I_K(cald(A_ast)_I)_simeq_K(cald(A_ast)):functors}
      We get for $I,J \in \cali$ with $I \le J$ functors
      \begin{eqnarray*}
        \cals(A_*)_I & \to & \call(A_*)_J;
        \\
        \cals(A_*)_I & \to & \call(A_*), 
      \end{eqnarray*}
      and analogously for $\calt$ and $\call$;

      \item\label{lem:hocolim_I_K(cald(A_ast)_I)_simeq_K(cald(A_ast)):hocolim} The functors appearing
      in assertion~\ref{lem:hocolim_I_K(cald(A_ast)_I)_simeq_K(cald(A_ast)):functors}
      induce  weak homotopy  equivalences, natural in $\cala_*$
      \begin{eqnarray*}
        \hocolim_{I \in \cali} \bfKinfty(\cals(\cala_*)_I) & \xrightarrow{\simeq} &\bfKinfty(\cals(\cala_*));
         \label{hocolim_I_in_cali_for_calo_upper_0(A_ast)}
         \\
        \hocolim_{I \in \cali} \bfKinfty(\calt(\cala_*)_I) & \xrightarrow{\simeq} &\bfKinfty(\calt(\cala_*));
         \label{hocolim_I_in_cali_for_calt_upper_0(A_ast)}
          \\
         \hocolim_{I \in \cali} \bfKinfty(\call(\cala_*)_I) & \xrightarrow{\simeq} &\bfKinfty(\call(\cala_*)).
        \label{hocolim_I_in_cali_for_cald_upper_0(A_ast)}
          \end{eqnarray*}
        \end{enumerate}
      \end{lemma}
 \begin{proof}~\ref{lem:hocolim_I_K(cald(A_ast)_I)_simeq_K(cald(A_ast)):functors}
   The desired functors come from~\eqref{calo_upper_0(A_ast)_I_subseteq_calo_upper_0(A_ast)_J_for_I_le_J},~%
\eqref{calt_upper_0(A_ast)_I_subseteq_calt_upper_0(A_ast)_J_for_I_le_J},   
and~\eqref{call_upper_0(A_ast)_I_subseteq_call_upper_0(A_ast)_J_for_I_le_J}.
\\[1mm]~\ref{lem:hocolim_I_K(cald(A_ast)_I)_simeq_K(cald(A_ast)):hocolim} 
This follows from 
of~\eqref{calo_upper_0(A_ast)_is_union_of_the_calo_upper_0(A_ast)_I}~%
\eqref{calt_upper_0(A_ast)_is_union_of_the_calt_upper_0(A_ast)_I},~%
\eqref{call_upper_0(A_ast)_is_union_of_the_call_upper_0(A_ast)_I}
and~\cite[Corollary~7.2]{Lueck-Steimle(2014delooping)}.
\end{proof}

Given $I \in \cali$, we define $\bigoplus_{m \in \IN} \cala_{I(m)}$ to be the full
subcategory of $\prod_{m \in \IN} \cala_{I(m)}$ consisting of those objects
$\underline{A}$ for which there exists a natural number $M$ (depending on $\underline{A}$)
satisfying $A_m = 0$ for $m \ge M$.  Let
$\bigl(\prod_{m \in \IN} \cala_{I(m)}\bigr) \bigl/ \bigl(\bigoplus _{m \in \IN} \cala_{I(m)}\bigr)$
be the quotient additive category.

\begin{lemma}\label{lem:proof_from_inside_to_outside}
Fix $I \in \cali$. There are weak homotopy equivalences, natural in $\cala_*$,
\begin{eqnarray*}
   \bfKinfty\bigl(\prod_{m \in \IN} \cala_{I(m)}\bigr) 
    & \xrightarrow{\simeq} & 
    \prod_{m \in \IN}  \bfKinfty\bigl(\cala_{I(m)}\bigr);
    \\
    \bigvee_{m \in \IN} \bfKinfty(\cala_{I(m)}\bigr) 
    & \xrightarrow{\simeq} & 
     \bfKinfty\bigl(\bigoplus_{m \in \IN} \cala_{I(m)}\bigr),
\end{eqnarray*}
and 
\begin{multline*}
  \hocofib\bigl(\bfKinfty\bigl(\bigoplus_{m \in \IN} \cala_{I(m)}\bigr) \to \bfKinfty\bigl(\prod_{m\in \IN} \cala_{I(m)}\bigr)\bigr)
  \\
  \xrightarrow{\simeq}
     \bfKinfty\left(\prod_{m \in \IN} \cala_{I(m)}\biggl/\bigoplus_{m \in \IN} \cala_{I(m)}\right).
   \end{multline*}
 \end{lemma}
\begin{proof}
  The first one is weak homotopy equivalence by~\cite{Carlsson(1995)}, see
  also~\cite[Theorem~1.2]{Kasprowski-Winges(2020)}, since the non-connective algebraic
  $K$-theory spectrum is indeed an $\Omega$-spectrum.  The second one is a weak homotopy
  equivalence, since $\bigoplus_{m \in \IN} \cala_m$ is the union of the subcategories
  $\bigoplus_{m = 0}^n \cala_i$ and hence we get a weak homotopy equivalence
  \[
\hocolim_{n \to \infty} \bfKinfty\bigl(\bigoplus_{m = 0}^n \cala_{I(m)}\bigr) 
\xrightarrow{\simeq} 
\bfKinfty\bigl(\bigoplus_{m \in \IN} \cala_{I(m)}\bigr),
\]
and the natural map 
\[
\bigvee_{m = 0}^n \bfK(\cala_{I(m)}) \xrightarrow{\simeq} \bfKinfty\bigl(\bigoplus_{m = 0}^n \cala_{I(m)}\bigr)
\]
is a weak homotopy equivalence. The third map is a weak homotopy equivalence, since
the inclusion  $\bigoplus_{m \in \IN} \cala_{I(m)} \subseteq \prod_{m\in \IN} \cala_{I(m)}$ is a Karoubi filtration.
\end{proof}

\begin{lemma}\label{lem:K_theory_of_calt_upper_0_I_calo_upper_0_I;cald_upper_0_I}
Given $I \in \cali$, there  are weak homotopy equivalences, natural in $\cala_*$,
\begin{eqnarray*}
   \bfKinfty\bigl(\cals(\cala_*)_I\bigr) 
    & \xrightarrow{\simeq} & 
    \prod_{m \in \IN}  \bfKinfty\bigl(\cala_{I(m)}\bigr);
    \\
    \bigvee_{m \in \IN} \bfKinfty(\cala_{I(m)}\bigr) 
    & \xrightarrow{\simeq} & 
    \bfKinfty\bigl(\calt(\cala_*)_I\bigr);
  \\
  \hocofib\bigl(\bfKinfty\bigl(\bigoplus_{m \in \IN} \cala_{I(m)}\bigr)
  \to \bfKinfty\bigl(\prod_{m\in \IN} \cala_{I(m)}\bigr)\bigr)
      & \xrightarrow{\simeq} & 
     \bfKinfty(\call(\cala)_I).
   \end{eqnarray*}
\end{lemma}
\begin{proof}
   The are obvious identifications
  \begin{eqnarray*}
    \prod_{m \in \IN} \cala_{I(m)} & = &  \cals(\cala_*)_I;
    \\
    \bigoplus_{m\in \IN} \cala_{I(m)} & = &  \calt(\cala_*)_I;
    \\
    \prod_{m \in \IN} \cala_{I(m)}\biggl/ \bigoplus _{m \in \IN} \cala_{I(m)} & = &  \call(\cala_*)_I.
  \end{eqnarray*}                                                                                                       
  Now the claim follows from Lema~\ref{lem:proof_from_inside_to_outside}.
\end{proof}

As a consequence of
Lemma~\ref{lem:hocolim_I_K(cald(A_ast)_I)_simeq_K(cald(A_ast))}~%
\ref{lem:hocolim_I_K(cald(A_ast)_I)_simeq_K(cald(A_ast)):hocolim} 
and 
Lemma~\ref{lem:K_theory_of_calt_upper_0_I_calo_upper_0_I;cald_upper_0_I}
we get

\begin{lemma}[$K$-groups of $\cals(\cala_*)$, $\calt(\cala_*)$, and $\call(\cala_*)$]%
\label{lemma:K-groups_of_cald_upper_0(cals_ast)}
  There are zigzags of weak homotopy equivalences of spectra, natural in $\cala_*$
  \begin{eqnarray*}
   \hocolim_{i \in \cali}  \prod_{m \in \IN}  \bfKinfty\bigl(\cala_{I(m)}\bigr)
    & \xleftrightarrow{\simeq}  & 
    \bfKinfty\bigl(\cals(\cala_*)\bigr);
    \\
    \hocolim_{i \in \cali}   \bigvee_{m \in \IN} \bfKinfty(\cala_{I(m)}\bigr) 
    & \xleftrightarrow{\simeq} & 
     \bfKinfty\bigl(\calt(\cala_*)\bigr);
     \\
    \hocolim_{i \in \cali}   \left(\hocofib\bigl(\bfKinfty\bigl(\bigoplus_{m \in \IN} \cala_{I(m)}\bigr)
    \to \bfKinfty\bigl(\prod_{m\in \IN} \cala_{I(m)}\bigr)\bigr)\right)
      & \xleftrightarrow{\simeq} & 
     \bfKinfty(\call(\cala)).
   \end{eqnarray*}
 \end{lemma}

 In particular we get from Lemma~\ref{lemma:K-groups_of_cald_upper_0(cals_ast)}
 for every $n \in \IZ$ an isomorphism
 \begin{eqnarray}
   \colim_{i\in I} \biggl(\prod_{m \in \IN} K_n(A_{I(m)})\biggl/\bigoplus_{m \in \IN} K_n(A_{I(m)})\biggr)
   & \cong &
   K_n(\call(\cala_*)).
  \label{K_n(call(cala_ast))}
 \end{eqnarray} 
 
 %%%%%%%%%%%%%%%%%%%%%%%%%%%%%%%%%%%%%%%%%%%%%%%%%%%%%%%%%%%%%%%%%%%%
%%%%%%%%%%%%%%%%%%%%%%%%%%%%% Section 13  %%%%%%%%%%%%%%%%%%%%%%%%%%%%%%%
%%%%%%%%%%%%%%%%%%%%%%%%%%%%%%%%%%%%%%%%%%%%%%%%%%%%%%%%%%%%%%%%%%%%

\typeout{-----   Section 13: The main technical result ---------}

\section{The main technical result}%
\label{sec:The_main_technical_result}

%%%%%%%%%%%%%%%%%%%%%%%%%%%%%%%%%%%%%%%%%%%%%%%%%%%%%%%%%%%%%%%%%%%%%

\subsection{The statement of the  main technical result}%
\label{subsec:The_statement_of_the_main_technical_result}

Fix a natural number $r$ and a nested sequence of additive categories $\cala_*$.  We have
defined the additive category $\call(\cala_*)$ in
Definition~\ref{def:calo_upper_0(cals_ast)}.  Suppose that each $\cala_*$ comes with a
$\IZ^r$-action $\Psi \colon \IZ^r \to \proaut(\cala_*)$ by pro-automorphisms in the sense
of Definition~\ref{def:(Pro-)morphisms_of_nested_sequences_of_additive_categories}.  Then
we obtain a $\IZ^r$-action $\Phi \colon \IZ^r \to \aut(\call(\cala_*))$ on
$\call(\cala_*)$ see~\eqref{call(F)}.  We obtain a covariant functor, see for
instance~\cite[Section~9]{Bartels-Lueck(2009coeff)},
\[
\bfKinfty_{\call(\cala_*)} \colon \OrG{\IZ^r} \to \Spectra.
\]
It determines a $\IZ^r$-homology theory $H_n^{\IZ^r}(-,\bfKinfty_{\call(\cala_*)} )$
with the property that for every subgroup $H \subseteq \IZ^r$ and $n \in \IZ$ we have the 
natural isomorphisms
\[
  H_n^{\IZ^r}(\IZ^r/H,\bfKinfty_{\call(\cala_*)}) \xrightarrow{\cong}
  K_n(\call(\cala_*) \rtimes_{\Phi|_{H}} H)
\]
as explained for instance in~\cite[Section~9]{Bartels-Lueck(2009coeff)}.  The
nested sequence of additive categories $\cala_*$ yields for any natural number $d$ another
nested sequence of additive categories $\cala_*[\IZ^d]$ by
$\cala_0[\IZ^d] \supseteq \cala_1[\IZ^d] \supseteq \cala_2[\IZ^d] \supseteq \cdots$, where
$\cala_m[\IZ]$ is the untwisted case of Definition~\ref{def:A_phi[t,t(-1)]} and we define
inductively $\cala_m[\IZ^d] = (\cala_m[\IZ^{d-1}])[\IZ]$.  Moreover,
$\call(\cala_*[\IZ^d])$ inherits a $\IZ^r$-action
$\Phi[\IZ^d] \colon \IZ^r \to \aut(\call(\cala_*[\IZ^d]))$.

The main result of this section is

\begin{theorem}\label{the:assembly_for_cald_upper_0(cals_ast)_with_Zr-action}
Suppose:

\begin{enumerate}

\item\label{the:assembly_for_cald_upper_0(cals_ast)_with_Zr-action:uniform_regular_coherence}
  For every natural number $d$ there exists a natural number $l(d)$ such that for any natural
  number $m$ the additive category $\cala_m[\IZ^d]$ is
  $l(d)$-uniformly regular coherent;

\item\label{the:assembly_for_cald_upper_0(cals_ast)_with_Zr-action:exactness}
The inclusion $\cala_{m+1}[\IZ^d] \to \cala_m[\IZ^d]$ is exact for any natural numbers $d$ and $m$.

\end{enumerate}

Then the map induced by the projection $E\IZ^r \to \pt$
\[
H_n^{\IZ^r}(E\IZ^r;\bfKinfty_{\call(\cala_*)}) \to 
H_n^{\IZ^r}(\pt;\bfKinfty_{\call(\cala_*)}) = K_n(\call(\cala_*) \rtimes_{\phi} \IZ^r)
\] 
is bijective for all $n \in \IZ$.
\end{theorem}

%%%%%%%%%%%%%%%%%%%%%%%%%%%%%%%%%%%%%%%%%%%%%%%%%%%%%%%%%%%%%%%%%%%%%

\subsection{Reduction to the case $r = 1$}%
\label{sec:Reduction_to_the_case_r_is_1}

\begin{lemma}\label{lem:reduction-to_the_case_r_is_1}
  If Theorem~\ref{the:assembly_for_cald_upper_0(cals_ast)_with_Zr-action} holds for $r =1 $,
  it is true for all $r \ge 1$.
\end{lemma}
\begin{proof}
The Farrell-Jones Conjecture is known to be true for $\IZ^r$
and implies that the map induced by the projection $\EGF{\IZ^r}{\VCYC} \to \pt$
\[
H_n^{\IZ^r}(\EGF{\IZ^r}{\VCYC};\bfKinfty_{\call(\cala_*)}) \to H_n^{\IZ^r}(\pt;\bfKinfty_{\call(\cala_*)})
\]
is an isomorphism for all $n \in \IZ$, where $\EGF{\IZ^r}{\VCYC}$ is the classifying space
of the family $\VCYC$ of virtually cyclic subgroups of $G$, see for
instance~\cite{Lueck(2005s)}.  We also have the map induced by the up to $\IZ^r$-homotopy
unique $\IZ^r$-map $E\IZ^r \to \EGF{\IZ^r}{\VCYC}$.
\begin{equation}
H_n^{\IZ^r}(E\IZ^r;\bfKinfty_{\call(\cala_*)}) \to H_n^{\IZ^r}(\EGF{\IZ^r}{\VCYC};\bfKinfty_{\call(\cala_*)}).
\label{rela_FJC_assembly_for_Zr}
\end{equation}
Obviously it suffices to show that~\eqref{rela_FJC_assembly_for_Zr} is bijective
for all $n \in \IZ$.  By the Transitivity Principle, see for instance~\cite[Theorem~65 on page~742]{Lueck-Reich(2005)},
this boils down to show that for any non-trivial virtually cyclic subgroup $V$ of $\IZ^m$ the map
induced by the projection $EV \to \pt$
\[
H_n^{V}(EV;\bfKinfty_{\call(\cala_*)}|_{V}) \to H_n^{V}(\pt;\bfKinfty_{\call(\cala_*)}|_V) 
\]
is bijective for all $n \in \IZ$. Since any non-trivial virtually cyclic subgroup $V$ of
$\IZ^r$ is isomorphic $\IZ$, we have reduced the proof of
Theorem~\ref{the:assembly_for_cald_upper_0(cals_ast)_with_Zr-action} to the special case $r = 1$.
\end{proof}

%%%%%%%%%%%%%%%%%%%%%%%%%%%%%%%%%%%%%%%%%%%%%%%%%%%%%%%%%%%%%%%%%%%%%

\subsection{Strategy of proof of Theorem~\ref{the:assembly_for_cald_upper_0(cals_ast)_with_Zr-action}}%
\label{subsec:Strategy_of_proof_of_Theorem_ref(the:assembly_for_cald_upper_0(cals_ast)_with_Zr-action)}%

We first explain, why the proof of
Theorem~\ref{the:assembly_for_cald_upper_0(cals_ast)_with_Zr-action} is  more
difficult than we had expected  and the reader may anticipate. There are are the following
reasons.

\begin{itemize}
\item There is  no proof of the  Devissage Theorem for non-connective $K$-theory.\\[1mm]
  We have used the Devissage Theorem for connective $K$-theory in the proof of
  Lemma~\ref{lem:The_connective_K-theory_of_NIL-categories_for_additive_categories}.  If
  Devissage would hold also in the non-connective setting, the proof of
  Lemma~\ref{lem:The_connective_K-theory_of_NIL-categories_for_additive_categories} could be extended to
  the non-connective $K$-theory spectrum and hence
  Lemma~\ref{lem:The_non-connective_K-theory_of_NIL-categories_for_additive_categories}
  and Theorem~\ref{the:The_non_connective_K-theory_of_additive_categories} would still
  be true, if we replace the assumption that $\cala[\IZ^m]$ is regular coherent for every
  $m \ge 0$ by the weaker assumption that that $\cala$ is regular coherent.

  Then the stronger version of
  Theorem~\ref{the:assembly_for_cald_upper_0(cals_ast)_with_Zr-action}, where one demands
  the conditions only for $d = 0$, would follow from the version of
  Theorem~\ref{the:The_non_connective_K-theory_of_additive_categories} mentioned above,
  and Lemma~\ref{lem:inverse_systems_and_regularity}, and
  Lemma~\ref{lem:reduction-to_the_case_r_is_1}.

  Since the known proofs of the Devissage Theorem for non-connective $K$-theory make
  strong assumptions on the underlying categories, which are not at all true in our case,
  we cannot argue like this. The consequence is that we need to make regularity
  assumptions about $\cala_m[\IZ^d]$ for all $d \ge 0$ and not only for $d = 0$.

\item If the additive category $\cala$ is regular coherent, it is unknown whether $\cala[\IZ]$
is regular coherent.\\[1mm]
  It is an open well-known problem, whether for a regular coherent ring $R$ the ring
  $R[\IZ^d]$ is again regular coherent. Hence we do not know, whether for a regular
  coherent additive category $\cala$, the additive category $\cala[\IZ^d]$ is regular
  coherent again.

\item The additive  category $\call(\cala_*)$ is not Noetherian and hence not regular.\\[1mm]
  We have shown that for a regular additive category $\cala$, the additive category
  $\cala[\IZ^d]$ is regular again, see Theorem~\ref{the:Regularity_for_cala_and_cala[t]}.
  However, the additive category $\call(\cala_*)$ is not Noetherian and hence not regular,
  see Example~\ref{exa:The_property_Noetherian_does_not_pass_to_the_sequence_category}.

\item The canonical inclusion $\call(\cala_*)[\IZ^d] \to \call(\cala_*[\IZ^d])$ is not
  an equivalence of additive categories for $d \ge 1$\footnote{The inclusion is bijective on objects, but not on morphisms.}. \\[1mm]
  The assumptions on the $\cala_m$ imply that $\call(\cala_*[\IZ^d])$ is regular coherent, but give us no information $\call(\cala_*)[\IZ^d]$.
  For this reason we cannot use the fact (coming from the Bass-Heller-Swan decomposition) that $K_{n} (\call(\cala_*))$ is a direct summand of $K_{n+d} (\call(\cala_*)[\IZ^d])$ to deduce Theorem~\ref{the:assembly_for_cald_upper_0(cals_ast)_with_Zr-action} for all $n$ from the case $n >> 0$.
  On the other hand for $\call(\cala_*[\IZ^d])$ we do not have a Bass-Heller-Swan decomposition\footnote{It seems not obvious how a Bass-Heller-Swan decomposition for $\call(\cala_*[\IZ^d])$ should look like and what the Nil terms should be.}.
  But we will exhibit $K_{n} (\call(\cala_*))$ also as a direct summand of $K_{n+d} (\call(\cala_*[\IZ^d]))$ and this will be the main work in the remainder of this paper.
  To this end we will check that enough of the arguments used in~\cite{Lueck-Steimle(2016BHS)} to construct the Bass-Heller-Swan decomposition can also be applied to $\call(\cala_*[\IZ^d])$ and yield the desired direct summand.\end{itemize}

%%%%%%%%%%%%%%%%%%%%%%%%%%%%%%%%%%%%%%%%%%%%%%%%%%%%%%%%%%%%%%%%%%%% 

\subsection{Twisted Laurent categories}\label{subsec:Twisted_Laurent_categories}

Recall the notion of a (twisted) finite Laurent category
from Definition~\ref{def:A_phi[t,t(-1)]}.  Given a sequences of additive subcategories $\cala_*$,
define the new sequences of additive subcategories $\cala_*[t^{\pm}]$ and $\cala_*[t,t^{\pm}]$
by replacing $\cala_m$  by the associated  untwisted finite
Laurent categories $\cala_m[t^{\pm}]$ and $\cala_m[t,t^{-1}]$.  A pro-automorphism
$\Phi \colon \cala_*\to \cala_*$ defines pro-automorphisms
$\Phi[t^{\pm}]  \colon \cala_*[t^{\pm}] \to \cala_*[t^{\pm}]$ and 
$\Phi[t,t^{{-1}}] \colon \cala_*[t,t^{{-1}}]\to \cala_*[t,t^{{-1}}]$. Define

  \begin{eqnarray*}
    \calc(\Phi)
    & = &
   \cals(\cala_*)_{\cals(\Phi)}[\IZ];
    \\
    \calc[t^{\pm}](\Phi)
    & = &
    \cals(\cala_*[t^{\pm}])_{\cals(\Phi[t^{\pm}])} [\IZ];
    \\
    \calc[t,t^{-1}](\Phi)
    & = &
    \cals(\cala_*[t,t^{-1}])_{\cals(\Phi[t,t^{-1}])} [\IZ].
  \end{eqnarray*}

  The passage from $\Phi$ to $\calc(\Phi)$, $\calc[t^{\pm}](\Phi)$, and
  $\calc[t,t^{-1}](\Phi)$ is natural.  Often we omit $\Phi$ from the
  notation. The set of objects of the categories $\calc(\Phi)$, $\calc[t^{\pm}](\Phi)$,
  and $\calc[t,t^{-1}](\Phi)$ can and will be identified with the set of
  objects of $\cals(\cala_*)$ and hence are independent of $\Phi$.

  %%%%%%%%%%%%%%%%%%%%%%%%%%%%%%%%%%%%%%%%%%%%%%%%%%%%%%%%%%%%%%%%

  \subsection{Induction functors}\label{subsec:Induction_functors}

  Next we define a commutative square of additive categories

    \begin{equation}
    \xymatrix@!C=8em{\calc \ar[r]^{i_+} \ar[d]_{i_-} \ar[rd]^{i_0}
        &
        \calc[t] \ar[d]^{j_+}
        \\
        \calc[t^{-1}] \ar[r]_{j_-} 
        &
        \calc[t,t^{-1}]}
      \label{square_of_induction_functors}
    \end{equation}
    where the functors $i_0$, $i_+$, $i_-$, $j_-$, and $j_+$ induce the identity on the set of
    objects. For each additive category $\cala_m$.we have a  diagram of functors of untwisted Laurent
    categories, where all functors are the canoncial ones, see~\cite[Section~1.4]{Lueck-Steimle(2016BHS)},
    \[
    \xymatrix@!C=8em{\cala_m \ar[r]^{i_+(\cala_m)} \ar[d]_{i_-(\cala_m)} \ar[rd]^{i_0(\cala_m)}
        &
        \cala_m[t] \ar[d]^{j_+(\cala_m)}
        \\
        \cala_m[t^{-1}] \ar[r]_{j_-(\cala_m)} 
        &
        \cala_m[t,t^{-1}].}
    \]
    This construction is natural with respect to  the inclusions $\cala_{m+1} \to \cala_{m}$.
    Hence we obtain a commutative diagram of additive categories
    \[
    \xymatrix@!C=11em{\cals(\cala_*) \ar[r]^{\cals(i_+(\cala_*))} \ar[d]_{\cals(i_-(\cala_*))} \ar[rd]^{\cals(i_0(\cala_*))}
        &
        \cals(\cala_*[t])\ar[d]^{\cals(j_+(\cala_*))}
        \\
        \cals(\cala_*)[t^{-1}]\ar[r]_{\cals(j_-(\cala_*))} 
        &
        \cals(\cala_*[t,t^{-1}]).}
    \]
    Since it is compatible  with the automorphisms of additive categories
    $\cals(\Phi)$ of $\cals(\cala_*)$,  $\cals(\Phi[t^{\pm}])$ of $\cals(\cala_*[t^{\pm}])$,
    and  $\cals(\Phi[t,t^{-1}])$ of $\cals(\cala_*[t,t^{-1}])$,
    it yield the desired diagram~\eqref{square_of_induction_functors} in the obvious way.

    Note that the diagram diagram~\eqref{square_of_induction_functors} is natural in $\Phi \colon \cala_*\to \cala_*$.

%%%%%%%%%%%%%%%%%%%%%%%%%%%%%%%%%%%%%%%%%%%%%%%%%%%%%%%%%%%%%%%%

    \subsection{Formally adjoining infinite direct sums}\label{subsec:formally_adjoining_infinite_formal_direct-sums}

    In~\cite[Section~1.3]{Lueck-Steimle(2016BHS)} a functorial extension
    $\cala \subseteq \cala^{\kappa}$ is constructed for every additive category $\cala$
    such that $\cala \subseteq \cala^{\kappa}$ is an inclusion of additive categories and
    in $\cala^{\kappa}$ the direct sums over a collection of objects over a countable set
    is defined. Now define
    for a nested sequence of additive categories a functor of additive categories
    $\cala_*$ given by $\cala_0 \supseteq \cala_1 \supseteq \cala_2 \supseteq \cdots$
    the  nested sequence of additive categories $\cala_*^{\kappa}$ by
    $\cala_0^{\kappa} \supseteq \cala_1^{\kappa}  \supseteq \cala_2^{\kappa}  \supseteq \cdots$.
    The given pro-automorphism $\phi \colon \cala_* \to \cala_*$ yields
    a pro-automorphism $\phi^{\kappa} \colon \cala_*^{\kappa} \to \cala_*^{\kappa}$
    by passing to $\phi^{\kappa} \colon \cala_0^{\kappa} \to \cala_0^{\kappa}$. Define
    \begin{eqnarray*}
    \calc^{\kappa}
    & = &
   \cals(\cala_*^{\kappa})_{\cals(\Phi^{\kappa})}[\IZ];
    \\
    \calc[t^{\pm}]^{\kappa}
    & = &
    \cals(\cala_*[t^{\pm}]^{\kappa})_{\cals(\Phi[t^{\pm}]^{\kappa})} [\IZ];
    \\
    \calc[t,t^{-1}]^{\kappa}
    & = &
    \cals(\cala_*[t,t^{-1}]^{\kappa})_{\cals(\Phi^{\kappa}[t,t^{-1}]^{\kappa})} [\IZ].
    \end{eqnarray*}

    Note that we are working with $\cala_m[t^{\pm}]^{\kappa} $, which is different from
        $(\cala_m^{\kappa})[t^{\pm}]$. The same comment applies to
       $\cala_m[t,t^{-1}]^{\kappa}$, which is different from  $(\cala_m^{\kappa})[t,t^{-1}]$.

  The diagram~\eqref{square_of_induction_functors} extends in the obvious way to
  the diagram, natural in $\Phi$.

    \begin{equation}
    \xymatrix@!C=8em{\calc^{\kappa} \ar[r]^{i_+^{\kappa}} \ar[d]_{i_-^{\kappa}} \ar[rd]^{i_0^{\kappa}}
        &
        \calc[t]^{\kappa} \ar[d]^{j_+^{\kappa}}
        \\
        \calc[t^{-1}]^{\kappa} \ar[r]_{j_-^{\kappa}} 
        &
        \calc[t,t^{-1}]^{\kappa}.}
      \label{square_of_induction_functors_with_kappa}
    \end{equation}

   %%%%%%%%%%%%%%%%%%%%%%%%%%%%%%%%%%%%%%%%%%%%%%%%%%%%%%%%%%%%%%%%

    \subsection{Restriction functors}\label{subsec:restriction_functors}

    In~\cite[Section~1.5]{Lueck-Steimle(2016BHS)} restriction functors are defined for
    each additive category $\cala_m$

    \begin{eqnarray*}
     i^0(\cala_m) \colon \cala_m[t,t^{-1}]^{\kappa} & \to & \cala_m^{\kappa};
      \\
      i^{\pm}(\cala_m) \colon \cala_m[t^{\pm 1}]^{\kappa} & \to & \cala_m^{\kappa},
    \end{eqnarray*}
    such that they come with  adjunctions $(i_0(\cala_m)^{\kappa},i^0(\cala_m))$,
    $(i_+(\cala_m)^{\kappa},i^+(\cala_m))$, and $(i_-(\cala_m)^{\kappa},i^-(\cala_m))$.
    Moreover, everything is natural and in particular compatible with the automorphisms
    $\Phi[t,t^{-1}]^{\kappa} \colon \cala_0[t,t^{-1}]^{\kappa} \to
    \cala_0[t,t^{-1}]^{\kappa}$ and
    $\Phi[t^{\pm}]^{\kappa} \colon \cala_0[t^{\pm}]^{\kappa} \to
    \cala_0[t^{\pm}]^{\kappa}$. Hence these data define analogously to the construction of
    Subsection~\ref{subsec:Induction_functors} restriction functors
    \begin{eqnarray}
      i^0 \colon \calc[t,t^{-1}]^{\kappa} & \to & \calc^{\kappa};
      \label{I_upper_0}
      \\
      i^{\pm}  \colon \calc[t^{ \pm 1}]^{\kappa} & \to & \calc^{\kappa},
      \label{I_upper_pm}
    \end{eqnarray}
    such that we have adjunctions $((i_0)^{\kappa},i^0)$, $((i_+)^{\kappa},i^+)$, and $((i_-)^{\kappa},i^-)$.
    Everything is natural in $\Phi$.

  %%%%%%%%%%%%%%%%%%%%%%%%%%%%%%%%%%%%%%%%%%%%%%%%%%%%%%%%%%%%%%%%%%%%

\subsection{Truncation functors}\label{subsec:truncation_functors}

Put
\[
\overline{\IZ} := \IZ \amalg \{-\infty,\infty\}.
\]

\begin{notation}[Truncation for objects]\label{not:truncation_for_objects}
  Let $\underline{A} = (A_m)_{m \in \IN}$ and $\underline{A'} = (A_m')_{m \in \IN}$ be
  objects in $\cals(\cala_*)$.  Consider elements $\underline{a} = (a_m)_{m \in \IN}$ and
  $\underline{b} = (b_m)_{m \in \IN}$ in $\prod_{m \in \IN} \overline{\IZ}$.  Define an
  object in $\calc^{\kappa}$ by
  \begin{equation*}
    \underline{A}[\underline{a}, \underline{b}] = \left\{\left.\bigoplus_{k_i = a_i}^{b_i} A_m \; \right|\; m \in \IN \right\},
  \end{equation*}
  where $\bigoplus_{k_i = a_i}^{b_i} A_m$ is to be defined to be zero if $a_i > b_i$ or if
  we have $a_i = b_i$ and $b_i \in \{\pm \infty\}$.  Note that the direct sum
  $\bigoplus_{k_i = a_i}^{b_i} A_m$ has as entry for $k_i$ always the same object, namely
  $A_m$.  Since we are working in $\calc^{\kappa}$, this definition makes sense also in
  the case where $a_i = -\infty$ or $b_i = \infty$.

Given an element $c \in \overline{\IZ}$, denote by $\underline{c}$ the
element in $\prod_{m \in \IN} \overline{\IZ}$,   whose value at every $m \in \IN$
is $c$. Then we get for any object $\underline{A}$ in $\calc[t,t^{-1}]$, which is the same
as an object in $\cals(\cala_*)$,
\[
  i^0\underline{A} = \underline{A}[\underline{-\infty},\underline{\infty}].
\]
Given a morphism $\underline{f} \colon \underline{A} \to \underline{A'}$ in
$\calc[t,t^{-1}]$ and $\underline{a}, \underline{b}$, $\underline{a'}$, and
$\underline{b'}$ in $\prod_{m \in \IN}  \overline{\IZ}$, define the $\calc^{\kappa}$-morphism
$f\trun \colon \underline{A}[\underline{a}, \underline{b}] \to \underline{A'}[\underline{a'}, \underline{b'}]$
to be the composite
\[
\underline{f}\trun \colon \underline{A}[\underline{a}, \underline{b}]
  \xrightarrow{\underline{i}}  \underline{A}[\underline{-\infty}, \underline{\infty}] = i^0\underline{A}
  \xrightarrow{i^0\underline{f}} \underline{A'}[\underline{-\infty}, \underline{\infty}] = i^0\underline{A'}
  \xrightarrow{\underline{p}}  \underline{A'}[\underline{a'}, \underline{b'}],
\]
 where $\underline{i}$ is the obvious inclusion and $\underline{p}$ is the obvious projection in $\calc^{\kappa}$.
\end{notation}

The morphism
$f\trun \colon i^0\underline{A} = A[\underline{-\infty},\underline{\infty}] \to
i^0\underline{A'} = A'[\underline{-\infty},\underline{\infty}]$ agrees with $i^0f$ for a
morphism $\underline{f} \colon \underline{A} \to \underline{A'}$ in $\calc[t,t^{-1}]$. If
$f$ belongs to $\calc[t^{\pm}]$, we abbreviate $(i_{\pm} f)\trun$ by $f \trun$ again.

Note that $(\underline{g} \circ \underline{f})\trun$ is in general \emph{not} equal to
$\underline{g}\trun \circ \underline{f}\trun$ and $\id\trun$ is in general \emph {not} the
identity. As a typical example, let $\underline{f^+}\colon \underline{A}\to \underline{A}$
be the morphism $\id_{\underline{A}}\cdot t$ and
$\underline{f^-}\colon \underline{A}\to \underline{A}$ be the morphism
$\id_{\underline{A}}\cdot t^{-1}$.Then
\[
(f^+ \circ f^-)\trun\colon A[0,\underline{\infty}]\to A[0,\underline{\infty}]
\]
is the identity. The map
\[
f^-\trun \colon \underline{A}[\underline{0},\underline{\infty}] \to \underline{A}[\underline{0},\underline{\infty}] 
\]
is given at each $m \in \IN $ by the map
$\bigoplus_{k = 0}^{\infty} A_m \to \bigoplus_{k =0}^{\infty} A_m$ sending
$\{u_k \mid k = 0, 1, 2, \ldots\}$ to $\{u_{k+1} \mid k = 0, 1, 2, \ldots\}$.
Since $f^-\trun$ is not injective, we do \emph{not} have
$f^+\trun \circ f^-\trun = (f^+ \circ f^-)\trun$.

As another example,
\[
  \id_{\underline{A} }\trun\colon \underline{A}[\underline{-\infty},\underline{\infty}]
  \to A[\underline{0},\underline{0}]
\]
is given at every $m \in \IN $ by the projection onto the $0$th summand
$\bigoplus_{k = -\infty}^{\infty} A_m \to A_m$.

\begin{notation}[Truncation for chain complexes]\label{not:truncation_for_chain_complexes}
  If $C^+$ is a $\calc[t]$-chain complex and
  $\underline{a}, \underline{b} \in \prod_{m \in \IN} \overline{\IZ}$, then
  we obtain a $\calc^\kappa$-chain complex $C^+[\underline{a},\underline{b}]$ by defining
  the $n$th chain object to be $C^+_n[\underline{a},\underline{b}]$ and the $n$-th
  differential to be
  $c_n\trun \colon C^+_n[\underline{a},\underline{b}]\to
  C^+_{n-1}[\underline{a},\underline{b}]$, if $c_n$ is the differential of $C^+$.  (One
  has to check that $c_n\trun \circ c_{n+1} \trun = 0$.)  A chain map
  $\underline{f} \colon C^+ \to D^+$ of $\calc[t]$-chain complexes induces a
  $\calc^\kappa$-chain map denoted by
  $f\trun \colon C^+[\underline{a},\underline{b}]\to D^+[\underline{a'},\underline{b'}]$,
  provided that $\underline{a'} \leq \underline{a}$, i.e., $a'_m \le a_m$ for all
  $m \in \IN$, and $\underline{b'}\leq \underline{b}$ hold.

  If $C^-$ is an $\calc[t^{-1}]$-chain complex and
  $\underline{a}, \underline{b} \in \prod_{n \in \IN} \overline{\IZ}$,
  define the $\calc^\kappa$-chain complex $C^-[\underline{a},\underline{b}]$ analogously.
  A chain map $f \colon C^- \to D^-$ of $\cala[t^{-1}]$-chain complexes induces a
  $\calc^\kappa$-chain map denoted by
  $f\trun \colon C^-[\underline{a},\underline{b}] \to D^-[\underline{a'},\underline{b'}]$,
  provided that $\underline{a'} \geq \underline{a}$, i.e., $a'_m \ge a_m$ for all
  $m \in \IN$, and $\underline{b'}\geq \underline{b}$ hold.
\end{notation}

Note that Notation~\ref{not:truncation_for_chain_complexes}  
(in contrast to Notation~\ref{not:truncation_for_objects})
does in this generality \emph{not} make sense 
for chain complexes in $\calc[t,t^{-1}]$, e.g.,
$c_n\trun \circ c_{n+1} \trun = 0$ does not hold anymore.

 %%%%%%%%%%%%%%%%%%%%%%%%%%%%%%%%%%%%%%%%%%%%%%%%%%%%%%%%%%%%%%%%%%%%

\subsection{Some basic tools for non-connective $K$-theory}%
\label{subsec:Some_basic_tools_for_non-connective_K-theory}

Recall that we have defined the negative $K$-theory of an additive category using the 
delooping construction based on the Bass-Heller-Swan decomposition
of~\cite{Lueck-Steimle(2014delooping)}. In this section we present another definition
based on the non-connective $K$-theory spectrum associated to appropriate Waldhausen
categories due to Bunke-Kasprowski-Winges~\cite{Bunke-Kasprowski-Winges(2021(split)}.

 The next definition is taken from~\cite[Definition~2.1]{Bunke-Kasprowski-Winges(2021(split)}.

 \begin{definition}\label{def:homotopical_and_the_factorization_property}
   \

   \begin{enumerate}

   \item~\label{def:homotopical_and_the_factorization_property:factorization} The
     Waldhausen category $\calw$ admits \emph{factorizations}, if every morphism in $\calw$
     can be factorized into a cofibration followed by a weak equivalence; no functoriality
     of this factorization is assumed;

   \item~\label{def:homotopical_and_the_factorization_property:factorization:homotopical}
     The Waldhausen category $\calw$ is \emph{homotopical}, if it admits factorizations
     and the weak equivalences satisfy the \emph{two-out-of-six property}, i.e., if for
     composable morphisms
     $C_0 \xrightarrow{f_1} C_1 \xrightarrow{f_2} C_2 \xrightarrow{f_3} C_3$ in $\calw$
     both $f_2 \circ f_1$ and $f_3 \circ f_2$ are weak equivalences, then also 
     $f_1$, $f_2$, $f_3$, and $f_3 \circ f_2 \circ f_1$ are weak equivalences.
   \end{enumerate}
 \end{definition}

 Let $\Waldho$ be the category of homotopical Waldhausen categories.
 In the sequel we denote by
 \begin{equation}
 \bfKinftyW \colon \Waldho  \to  \Spectra
 \label{bfKinftyW_colon_Waldho_to_Spectra}
\end{equation}
the non-connective $K$-theory functor constructed in~\cite[Definition~2.37]{Bunke-Kasprowski-Winges(2021(split)}.

\begin{remark}\label{rem:cala_and_Ch(cala)}
  Let $\cala$ be an additive category. Then $\cala$ becomes a Waldhausen category, if we
  define the weak equivalences to be the isomorphisms and the cofibration to be the
  morphisms $f \colon A \to B$, for which there exists an object $A^{\perp}$ and an
  isomorphism $u \colon A \oplus A^{\perp} \xrightarrow{\cong} {B}$ such that the
  composite of $u$ with the canonical inclusion $A \to A \oplus A^{\perp}$ is $f$.  Note that
  this Waldhausen category is \emph{not} homotopical, as it does not satisfy
  factorization. So we cannot apply~\eqref{bfKinftyW_colon_Waldho_to_Spectra} to the
  Waldhausen category $\cala$.
 
  Let $\Chcat(\cala)$ be the Waldhausen category of bounded chain complexes over $\cala$,
  where a cofibration $f_* \colon C_* \to D_*$ is a chain map such that
  $f_n \colon C_n \to D_n$ is a cofibration in $\cala$ and the weak equivalences are the
  chain homotopy equivalences.  Then $\Chcat(\cala)$ is homotopical thanks to the mapping
  cylinder construction. Hence we can apply~\eqref{bfKinftyW_colon_Waldho_to_Spectra}
  to the Waldhausen category $\Chcat(\cala)$ and can consider its non-connective $K$-theory
  spectrum $\bfKinftyW(\Chcat(\cala))$.

  More generally, if $\cala$ is an exact category, then the Waldhausen category $\Chcat(\cala)$ can
  be defined analogously and is homotopical.
\end{remark}

Suppose that $\calw$ is a category with cofibrations
and that $\calw$  is equipped with two categories of weak equivalences, one finer than
the other, $v\calw \subseteq w \calw$.
Thus $\calw$ becomes a  Waldhausen category in two ways.
Suppose that in both cases $\calw$ is a homotopical Walhausen category.
Let $\calw^w$ denote the  full subcategory  of $\calw$ given
by the objects $C$  in $\calw$ having the property that the map  $C \to \pt$ belongs to  $w\calw$. 
Then  $\calw^w$ inherits two Waldhausen structures,
if we put $v\calw^w = \calw^w \cap v\calw $  and $w\calw^w = \calw^w \cap w\calw$.
Both yield homotopical Waldhausen categories.

\begin{theorem}[Fibration Theorem]\label{the:Fibration_Theorem}
  Under the assumptions above we get a weak   homotopy fibration of spectra
\[
\bfKinftyW(\calw^w,v\calw^w) \to \bfKinftyW(\calw,v\calw) \to \bfKinftyW(\calw,w\calw).
\]
\end{theorem}
\begin{proof} 
  This follows from~\cite[Theorem~2.35]{Bunke-Kasprowski-Winges(2021(split)}.
  \end{proof}

 \begin{theorem}[Cisinski's Approximation Theorem]\label{the:Cisinki_Approximation_Theorem} 
  Let $F \colon \calw_0 \to \calw_1$ be
  an exact functor of homotopical Waldhausen categories. Assume:

   \begin{enumerate}

   \item\label{the:Cisinki_Approximation_Theorem:A1} An arrow in $\calw_0$ is a weak equivalence
     in $\calw_0$, if and only if its image in $\calw_1$ is a weak equivalence in
     $\calw_1$;

   \item\label{the:Cisinki_Approximation_Theorem:A2} Given any object $C_0$ in $\calw_0$ and any
     map $f\colon F(C_0) \to C_1$ in $\calw_1$, there exists a commutative diagram in $\calw_1$
    \[
    \xymatrix{F(C_0) \ar[r]^-{f} \ar[d]_{F(u)}
    &
    C_1 \ar[d]^{v}_{\simeq}
    \\
    F(D_0) \ar[r]_-{w}^{\simeq}
    & 
    D_1
    }
    \]
    for a morphism $u \colon C_0 \to D_0$ in $\calw_0$ and weak equivalences $v \colon C_1 \to D_1$ and
   $w \colon F(D_0) \to D_1$ in $\calw_1$.
   \end{enumerate}
   
   Then  the map of spectra $\bfKinftyW(F) \colon \bfKinftyW(\calw_0)    \xrightarrow{\simeq} \bfKinftyW(\calw_1)$ 
    is a weak homotopy equivalence.
 \end{theorem}
 \begin{proof}  
   This follows from~\cite[Theorem~2.16]{Bunke-Kasprowski-Winges(2021(split)}.
 \end{proof}

 \begin{theorem}[Cofinality Theorem]\label{the:Cofinality_Theorem}
   Let $I \colon \calw_0 \to \calw_1$ be the inclusion of a full homotopical Waldhausen
   subcategory $\calw_0$ into a homotopical Waldhausen category $\calw_1$. Assume:

   \begin{enumerate}
    \item\label{the:Cofinality_Theorem:mapping_cylinder_argument} The functor $F$ 
      admits a mapping cylinder argument, i.e., for every morphism $f \colon C_0 \to C_1$
      in $\calw_1$ such that $C_0$ belongs to $\calw_0$ and $C_1$ is the target of a weak equivalence
       with some object in $\calw_0$ as source, there is a factorization in $\calw_1$
      \[C_0 \xrightarrow{f'} C' \xrightarrow{f''} C_1
      \]
      such that $C'$ belongs to $\calw_0$ and $f''$ is a weak equivalence;

    \item\label{the:Cofinality_Theorem:domination} The category $\calw_1$ is dominated by $\calw_0$, i.e.,
      for any object $C_1$ in $\calw_1$ there exists an object $C_0$ in $\calw_0$ and an object $C_1'$ in $\calw_1$
      and morphisms
      $r \colon C_0 \to C_1$ and $i \colon C_1' \to C_0$ such that $r \circ i$ is a weak equivalence.
    \end{enumerate}
    
   Then $\bfKinftyW(I) \colon \bfKinftyW(\calw_0) \to \bfKinftyW(\calw_1)$ is a weak homotopy
   equivalence.
 \end{theorem}
 \begin{proof}
   This follows from~\cite[Theorem~2.30]{Bunke-Kasprowski-Winges(2021(split)} and the fact that
   on the level of stable $\infty$-categories non-connective $K$-theory is inverting the
   passage to the idempotent completion.
 \end{proof}

 For the  proof of the following result we refer to~\cite[Theorem~18.17]{Lueck(2022book)}.
 
\begin{theorem}[The weak homotopy fibration sequence of a stable Karoubi
  filtration for $K$-theory in the setting of Waldhausen categories]%
\label{the:The_weak_homotopy_fibration_sequence_of_a_stable_Karoubi_filtration_for_K-theory_Waldhausen}%
Let $\cala$ be a  additive category and $i \colon \calu \to \cala$ be the inclusion
of a full additive subcategory. If the  additive category $\cala$ is stably
$\calu$-filtered, then
  \[
    \bfKinftyW(\Chcat(\calu)) \to \bfKinftyW(\Chcat(\cala)) \to \bfKinftyW(\Chcat(\cala/\calu))
  \]
  is a weak homotopy fibration of non-connective spectra,
  where $\bfKinftyW$ has been defined in~\eqref{bfKinftyW_colon_Waldho_to_Spectra},
\end{theorem}

The proof of the next result can be found in~\cite[Theorem~18.33]{Lueck(2022book)}.

\begin{theorem}[Gillet-Waldhausen zigzag for non-connective $K$-theory]%
  \label{the:Gillet-Waldhausen_zigzag_for_non-connective_K-theory}%
    There is a zigzag of weak homotopy equivalences,
  natural in $\cala$, from the non-connective $K$-theory spectrum $\bfKinftyW(\Chcat(\cala))$ in
  the sense of Bunke-Kasprowski-Winges~\cite{Bunke-Kasprowski-Winges(2021(split)} to the
  non-connective $K$-theory spectrum $\bfKinfty(\cala)$ in the sense of
  L\"uck-Steimle~\cite{Lueck-Steimle(2014delooping)}.
\end{theorem}

%%%%%%%%%%%%%%%%%%%%%%%%%%%%%%%%%%%%%%%%%%%%%%%%%%%%%%%%%%%%%%%%%%%%

\subsection{The projective line}\label{subsec:The_projective_line}

\begin{definition}[Projective line]\label{def:projective_line}
 We define \emph{the projective line} $\calx$ to be the following  additive category. 
 Objects are
triples $(C^+,f,C^-)$ consisting of objects $C^+$ in $\calc[t^+]$ and $C^-$ in $\calc[t^-]$ and an
isomorphism $f \colon j_+(C^+) \to j_-(C^-)$ in $\calc[t,t^{-1}]$. A morphism 
$(\varphi^+,\varphi^-) \colon (C^+,f,C^-) \to (D^+,g,D^-)$ in $\calx$ consists of morphisms 
$\varphi^+ \colon C^+ \to D^+$ in $\calc[t]$ and  $\varphi^- \colon C^- \to D^-$ in
$\calc[t^{-1}]$ such that the following diagram commutes in $\calc[t,t^{-1}]$
\[
\xymatrix{j_+(C^+) \ar[r]^{f} \ar[d]_{j_+(\varphi^+)} 
& j_-(C^-) \ar[d]^{j_-(\varphi^-)}
\\
j_+(D^+) \ar[r]_{g} 
& j_-(D^-).
}
\]
\end{definition}

Let 
\begin{eqnarray}
k^{\pm} \colon \calx & \to & \calc[t^{\pm 1}]
\label{def_f_kpm}
\end{eqnarray}
be the functor sending $(C^+,f,C^-)$ to $C^{\pm}$.

The category $\calx$ is naturally an exact category by declaring a sequence to be exact, if
and only if becomes (split) exact both after applying $k^+$ and $k^-$. With this exact
structure we obtain the structure of a homotopical  Waldhausen category on $\Chcat(\calx)$.
Hence $\bfKinftyW(\Chcat(\calx))$ is defined, in contrast to $\bfKinfty(\calx)$
which does not make sense, since $\calx$ is neither an additive category
nor a homotopical  Waldhausen category.

The next theorem will be a main ingredient in the proof of
Theorem~\ref{the:Splitting_the_Bass-Heller-Swan_map}.

\begin{theorem}[The algebraic $K$-theory of the projective line]\label{the:homotopy_pull_back_of_calx}
Consider the following (not necessarily commutative) diagram of spectra
\[
\xymatrix@!C=16em{\bfKinftyW(\Chcat(\calx)) \ar[r]^-{\bfKinftyW(\Chcat(k^-))} 
\ar[d]_{\bfKinftyW(\Chcat(k^+))}
& \bfKinftyW(\Chcat(\calc[t^{-1}]))\ar[d]^{\bfKinftyW\Chcat((j_-))}
\\
\bfKinftyW(\Chcat(\calc[t])) \ar[r]_-{\bfKinftyW(\Chcat(j_+))}
&
\bfKinftyW(\Chcat(\calc[t,t^{-1}])).
}
\]
There is a natural equivalence of functors
$T \colon j_+ \circ k^+ \xrightarrow{\cong} j_- \circ k^-$, which is given on an object $(A^+,f,A^-)$ by $f$.
It induces a preferred homotopy 
$\bfKinftyW(\Chcat(j_+)) \circ \bfKinftyW(\Chcat(k^+)) \simeq \bfKinftyW(\Chcat(j_-))  \circ \bfKinftyW(\Chcat(k^-))$.

Then the diagram above is a weak homotopy pullback, i.e.,
the canonical map from $\bfKinftyW(\calx)$ to the homotopy pullback of
\[
\xymatrix@!C=16em{
& \bfKinftyW(\Chcat(\calc[t^{-1}]))\ar[d]^{\bfKinftyW\Chcat((j_-))}
\\
\bfKinftyW(\Chcat(\calc[t])) \ar[r]_-{\bfKinftyW(\Chcat(j_+))}
&
\bfKinftyW(\Chcat(\calc[t,t^{-1}])).
}
\]
is a weak homotopy equivalence.
\end{theorem}

For the proof of Theorem~\ref{the:homotopy_pull_back_of_calx}, we can use the ideas
of~\cite[Section~5]{Lueck-Steimle(2016BHS)} taking into account that meanwhile the basic
tools appearing in~\cite[Section~4]{Lueck-Steimle(2016BHS)} have been proved also for
non-connective $K$-theory.

In the first step of the proof of Theorem~\ref{the:homotopy_pull_back_of_calx} we replace
the additive category $\calc[t]$ by a larger exact category $\caly$ with equivalent
$K$-theory. It is defined as follows: An object of $\caly$ is a triple $(A^+, f, A)$
consisting of objects $A^+$ in $\calc[t]$ and $A^-$ of $\calc[t,t^{-1}]$ and an
isomorphism $f\colon j_+A^+\to A$ in $\calc[t,t^{-1}]$. A morphism from $(A^+, f, A)$ to
$(B^+, g, B)$ is a morphism $\varphi^+\colon A^+\to B^+$ in $\calc[t]$ and a morphism
$\varphi \colon A \to B$ in $\calc[t,t^{-1}]$ such that diagram in $\calc[t,t^{-1}]$
\[
  \xymatrix{j_+(A^+)\ar[r]^-f \ar[d]_{j_+(\varphi^+)} & A \ar[d]^{\varphi}\\
    j_+(B^+) \ar[r]_-g & B }
\]
commutes. Note that $\varphi$ is determined already by $\varphi^+$.  The
category $\caly$ is exact in the same way as $\calx$ is. Note that $\calx$ and $\caly$
have the same set of objects but $\caly$ contains more morphisms, since in contrast to
$\calx$ we do not require that $\varphi$ belongs to $\calc[t^{-1}]$.

\begin{lemma}\label{lem:caly_and_cala}
The functors
\[
a\colon \calc[t]\to \caly, \quad A\mapsto (A, \id, j_-A)
\]
and
\[
b\colon \caly\to\calc[t], \quad (A^+, f, A^-)\mapsto A^+
\]
are exact. The composite $b\circ a$ is the identity and the composite $a\circ b$ is
naturally isomorphic to the identity functor.  In particular, they induce homotopy
equivalences on non-connective $K$-theory, homotopy inverse to each other,
\begin{eqnarray*}
  \bfKinftyW(\Chcat(a)) \colon \bfKinftyW(\Chcat(\calc[t]))
  & \xrightarrow{\simeq}  &
  \bfKinftyW(\Chcat(\caly));
  \\
  \bfKinftyW(\Chcat(b))\colon \bfKinftyW(\Chcat(\caly))
  & \xrightarrow{\simeq}  &
                            \bfKinftyW(\Chcat(\calc[t])).
\end{eqnarray*}
\end{lemma}

\begin{proof}
  It is clear that the functors are exact. Obviously $b\circ a$ is the identity. The
  composite $a\circ b$ is naturally isomorphic to the identity functor: the isomorphism in
  $\caly$ at the object $(A^+, f, A)$ is given by
  $(\id,f) \colon (A^+,\id,j_+A^+)\xrightarrow{\cong}(A^+,f,A) $. This implies
  $\bfK(a)\circ \bfK(b)\simeq \id$.
\end{proof}

Denote by 
\[
k'\colon \calx\to\caly
\]
the inclusion functor, and define
\[
j'\colon \Chcat(\caly)\to \Chcat(\calc[t,t^{-1}]), \quad (A^+,f,A) \mapsto A.
\]
 
Then  the square
\[
\xymatrix{{\Chcat(\calx)} \ar[rr]^(.4){\Chcat(k^-)} \ar[d]_{\Chcat(k')} &&
\Chcat(\calc[t^{-1}]) \ar[d]^-{\Chcat(j^-)}\\
{\Chcat(\caly)} \ar[rr]_-{\Chcat(j')} &&   
\Chcat(\calc[t,t^{-1}])
}
\]
is strictly commutative, and we are going to show that it induces a weak 
homotopy pullback after applying $\bfKinftyW$. To show that the square is a weak
homotopy pullback on non-connective $K$-theory, we are going to show that the
horizontal homotopy fibers of $\bfKinftyW(\Chcat (k^-))$ and $\bfKinftyW(\Chcat(j'))$
agree up to weak homotopy equivalence.

Let $w\Chcat( \calx)$ be the subcategory of
$\Chcat(\calx)$ consisting of all chain maps, which become after applying $\Chcat(k^-)$ weak equivalences
in  $\Chcat(\calc[t^{-1}])$. Let $\Chcat(\calx)^w$ be the
full subcategory of $\Chcat(\calx)$ of all objects, which are $w$-acyclic. In other words, an
object $(C^+, f, C^-)$ belongs to $\Chcat(\calx)^w$, if and only if $C^-$ is
contractible as an $\calc[t^{-1}]$-chain complex.  Similarly, denote
by $w\Chcat(\caly)$ the subcategory of all morphisms $f$ such that 
$\Chcat(j')(f)$ is a chain homotopy equivalence in $\Chcat(\calc[t,t^{-1}])$, and adopt the
notation $\Chcat(\caly)^w$ for the $w$-acyclic objects.

\begin{lemma}\label{lem:certain_chain_homotopy_equivalences}
The maps
\begin{align*}
   \bfKinftyW(\Chcat(k^-))\colon & \bfKinftyW(\Chcat(\calx), w)\to \bfKinftyW(\Chcat(\calc[t^{-1}]));\\
   \bfKinftyW(\Chcat (j')) \colon & \bfKinftyW(\Chcat (\caly), w)\to \bfKinftyW(\Chcat (\calc[t,t^{-1}])),
\end{align*}
are homotopy equivalences.
\end{lemma}
\begin{proof}
  We want to apply the Cisinki's Approximation Theorem~\ref{the:Cisinki_Approximation_Theorem}.
  We give the details only for $\bfKinftyW(\Chcat(k^-))$, the analogous proof for $\bfKinftyW(\Chcat (j'))$ is
  left to the reader. We have to verify the
  conditions~\ref{the:Cisinki_Approximation_Theorem:A1} and~\ref{the:Cisinki_Approximation_Theorem:A2}
  appearing in Cisinki's Approximation
  Theorem~\ref{the:Cisinki_Approximation_Theorem}.

  A morphism $f$ in $\Chcat(\calx)$ is by definition in $w\Chcat(\calx)$, if and only if
  $\Chcat (k^-)(f)$ is a chain homotopy equivalence in $\Chcat(\calc[t^{-1}])$. This takes care of 
  condition~\ref{the:Cisinki_Approximation_Theorem:A1} for $\bfKinftyW(\Chcat(k^-))$.

  Next we deal with condition~\ref{the:Cisinki_Approximation_Theorem:A2}.  Consider an object
  $(C^+, f, C^-)$ in $\Chcat(\calx)$ and a morphism $\varphi^-\colon C^-\to D^-$ in 
  $\Chcat(\calc[t^{-1}])$. We will extend $\varphi^-$ to a morphism
  \[
   \varphi=(\varphi^+, \varphi^-)\colon (C^+, f, C^-)\to (D^+, g, D^-)
  \]
  in $\Chcat(\calx)$.  If we have achieved this, we are done by the following argument.
  Note that $\varphi=(\varphi^+, \varphi^-)$ is a morphism in $\Chcat(\calx)$ projecting
  to $\varphi^-$ under $\Chcat(k^-)$. Then, factorizing $\varphi=\mu \circ\psi$ into a
  cofibration $\psi$ followed by a weak equivalence $\mu$ (using the mapping cylinder), we
  can write $\varphi^-=\mu^-\circ \Chcat (k^-)(\psi)$, where $\psi$ is a cofibration and
  $\mu^-$ is a weak equivalence, as required in
  condition~\ref{the:Cisinki_Approximation_Theorem:A2}.

  The construction of $D^+$ and $\varphi^+$ require the following preparation.

  Consider  an object $\underline {A} \in \cals(\cala_*)$ and an element
  $\underline{k} = \{k_m \mid m \in \IN\} \in \prod_{m \in \IN} \IZ$. Define a morphism
  $t^{\underline{k}} \colon \underline{A} \to \underline{A}$ in
  $\cals(\cala_*[t,t^{-1}])$ by  $(t^{k_m})_{m \in \IN}$, where
  $t^{k_m} \colon A_m \to A_m$ is the obvious morphism in $\cala_m[t,t^{-1}]$ determined
  by $t^{k_m}$.  We can and will regard
  $t^{\underline{k}} \colon \underline{A} \to \underline{A}$ also as a morphism in
  $\calc[t,t^{-1}]$. One easily checks for any morphism $f \colon A \to B$ in
  $\calc[t,t^{-1}]$ that $t^{\underline{k}}\circ f = f \circ t^{\underline{k}}$ holds.
  Moreover we have
  $t^{\underline{k} + \underline{k'}} = t^{\underline{k}} \circ t^{\underline{k'}}$ for
  $\underline{k} + \underline{k'}$ defined by the componentwise addition in
  $\prod_{m \in \IN} \IZ$ and $t^{\underline{0}} = \id$ for
  $\underline{0} \in \prod_{m \in \IN} \IZ$ given by the element which has zero as entry for
  the component for $m \in \IN$. The following property is important for us.  Given a
  morphism $f \colon A \to B$ in $\calc[t,t^{-1}]$, there exists
  $\underline{k}_f \in \prod_{m \in \IN} \IZ$ such that for every
  $\underline{k} \in \prod_{m \in \IN} \IZ$ with $\underline{k} \ge \underline{k}_f$ we have
  $ t^{\underline{k}} \circ f$ in $\calc[t]$, where $\underline{k} \ge \underline{k}_f$ is
  to be understood componentwise.

  Choose a natural  number $N$ such that
  $C^+_m =  D^-_m = 0$ for $|m| > N$. Consider $\underline{k} \in \prod_{m \in \IN} \IZ$.
  For any integer $n$ define $n \cdot \underline{k}$ to be the element, whose $i$-th entry is $n \cdot k_m$.
  Then we obtain the following commutative diagram in $\calc[t,t^{-1}]$

  \[
  \xymatrix@!C=10em{\vdots \ar[d]
    & \vdots \ar[d]
    & \vdots \ar[d]
    \\
    0 \ar[r]^{\id} \ar[d]
    &
    0 \ar[r]^{\id} \ar[d]
    &
    0  \ar[d]
    \\
    C^+_N \ar[d]^{c^+_N}
    \ar[r]^{t^{\underline{k}} \circ \varphi^-_N \circ f_N}
      &
     D^-_N \ar[d]^{t^{\underline{k}} \circ  d^-_N}
     \ar[r]^{t^{-\underline{k}}}
     &
       D^-_N \ar[d]^{d^-_N}
      \\
      C^+_{N-1} \ar[d]^{c^+_{N-1}}
    \ar[r]^{t^{2 \cdot \underline{k}}\circ \varphi^-_{N-1} \circ f_{N-1}}
      &
     D^-_{N-1} \ar[d]^{t^{\underline{k}} \circ  d^-_{N-1}}
     \ar[r]^{t^{-2 \cdot \underline{k}}}
     &
       D^-_{N-1} \ar[d]^{d^-_{N-1}}
         \\
     {\vdots} \ar[d]^{c^{+}_{-N+1}}
    & {\vdots}  \ar[d]^{t^{\underline{k}} \circ d^{-}_{-N+1}}
    & {\vdots} \ar[d]^{d^{-}_{-N+1}}
    \\
    C^+_{-N+1} \ar[d]^{c^+_{-N+1}}
    \ar[r]^{t^{2N \cdot \underline{k}} \circ \varphi^-_{-N+1} \circ f_{-N+1}}
      &
     D^{-}_{-N+1} \ar[d]^{t^{\underline{k}} \circ  d^{-}_{-N+1}}
     \ar[r]^{t^{-2N \cdot \underline{k}}}
     &
       D^-_{-N+1} \ar[d]^{d^-_{-N+1}}
      \\
      C^+_{-N} \ar[d]^{0}
    \ar[r]^{t^{(2N+1) \cdot \underline{k}}\circ \varphi^-_{-N} \circ f_{-N}}
      &
     D^-_{-N} \ar[d]^{0}
     \ar[r]^{t^{-(2N+1) \cdot \underline{k}}}
     &
       D^-_{-N} \ar[d]^{0}  
       \\
       0 \ar[d]^{0} \ar[r]^{\id}
       &
       0 \ar[d]^{0}  \ar[r]^{\id}
       &
       0 \ar[d]^{0}
       \\
       \vdots
       &
       \vdots
       &
       \vdots
        }
      \]
      where the left column is the chain complex $C^+$, the right column is the chain
      complex $D^-$ and the middle column is obtained from $D^-$ by composing each
      differential with $t^{\underline{k}}$.  Since there are only finitely many objects
        and morphisms in the diagram above, which are different from zero, we can choose
        $\underline{k}$ such that each horizontal arrow from the first column to the
        second column and each vertical arrow in the middle column belong to $\calc[t]$.
        Hence the middle column defines a chain complex $D^+$ over $\calc[t]$, the
        horizontal arrows from the left column to the middle column define a chain map
        $\phi^+ \colon C^+ \to D^+$ over $\calc[t]$ and the collection of the morphisms  from
        the middle column to the right column define a chain homotopy equivalence
        $g \colon D^+ \to D^-$ over $\calc[t,t^{-1}]$. By construction we have
        $ j_-(\varphi^-) \circ f = g \circ j_+(\varphi^+) $.
        This finishes the construction of
        the desired morphism $\phi$ and hence the proof of
        Lemma~\ref{lem:certain_chain_homotopy_equivalences}.
\end{proof}

\begin{theorem}\label{the:relative_fibration_sequences}
  There are weak homotopy fibration sequences 
\begin{align*}
   \bfKinftyW(\Chcat(\calx)^w) & \to \bfKinftyW(\Chcat(\calx))\to \bfKinftyW(\calc[t^{-1}]);
   \\
   \bfKinftyW(\Chcat(\caly)^w) &\to \bfKinftyW(\Chcat(\caly))\to \bfKinftyW(\calc[t,t^{-1}]).
\end{align*}
\end{theorem}
\begin{proof}
We give the
  details only for the first sequence, the analogous proof for the second one is
  left to the reader.

  We apply the Fibration Theorem~\ref{the:Fibration_Theorem} in the case
  $\calw = \Chcat(\calx)$, $w$ as described above and $v$ the structure of weak equivalences
  coming from chain homotopy equivalences. The necessary conditions appearing in the
  Fibration Theorem~\ref{the:Fibration_Theorem} are satisfied by~\cite[Lemma~3.1 and
  Lemma~4.12]{Lueck-Steimle(2016BHS)}.  Thus we obtain a weak homotopy
  fibration of spectra
\[
\bfKinftyW(\Chcat(\calx)^w) \to \bfKinftyW(\Chcat(\calx)) \to \bfKinftyW(\Chcat(\calx),w).
\]
Because of Lemma~\ref{lem:certain_chain_homotopy_equivalences}
we obtain a weak homotopy fibration
\[
\bfKinftyW(\Chcat(\calx)^w) \to \bfKinftyW(\Chcat(\calx)) \to \bfKinftyW(\Chcat(\calc[t^{-1}])).
\]
\end{proof}

\begin{lemma}\label{k'_is_homotopy_equivalence}
The functor $k'$ induces a weak homotopy equivalence
\[
\bfKinftyW(\Chcat(\calx)^w) \xrightarrow{\simeq} \bfKinftyW(\Chcat(\caly)^w).
\]
\end{lemma}

\begin{proof}
  Again we will use the Cisinki's Approximation Theorem~\ref{the:Cisinki_Approximation_Theorem}.
  We have to verify the
  conditions~\ref{the:Cisinki_Approximation_Theorem:A1} and~\ref{the:Cisinki_Approximation_Theorem:A2}
  appearing in Cisinki's Approximation
  Theorem~\ref{the:Cisinki_Approximation_Theorem}.

  Let
\begin{equation}\label{diagram_defining_a_morphism_in_calx}
  \xymatrix{j_+C^+ \ar[r]^f \ar[d]_{j_+\varphi^+} & j_-C^- \ar[d]^{j_-\varphi^-}\\
 j_+D^+ \ar[r]^g & j_-D^- 
}
\end{equation}
represent a morphism in $\Chcat(\calx)^w$, which maps to a weak equivalence in
$\Chcat(\caly)^w$. Then $\varphi^+$ is a chain homotopy equivalence in
$\calc[t]$ and $\varphi^-$ is a chain homotopy equivalence in
$\calc[t,t^{-1}]$. By assumption, $C^-$ and $D^-$ are contractible in $\calc[t^{-1}]$, so
$\varphi^-$ has to be an equivalence in $\calc[t^{-1}]$. It follows that the
morphism given by~\eqref{diagram_defining_a_morphism_in_calx} is a weak equivalence
in~$\Chcat(\calx)^w$ already. This takes care of 
condition~\ref{the:Cisinki_Approximation_Theorem:A1}.

It remains to check condition~\ref{the:Cisinki_Approximation_Theorem:A2}.
Suppose now that
\begin{equation}\label{diagram_defining_a_morphism_in_caly}
  \xymatrix{j_+C^+ \ar[r]^f \ar[d]_{j_+\varphi^+} & C^- \ar[d]^{\varphi^-}\\
 j_+D^+ \ar[r]^g & D^- 
}
\end{equation}
represents a morphism in
$\Chcat(\caly)^w$ satisfying $(C^+, f, C^-)$ in $\Chcat(\calx)^w$. We have to
factor this morphism through a map in $\Chcat(\calx)^w$ (which we may then
replace by a cofibration using the mapping cylinder) and a weak
equivalence in $\Chcat(\caly)^w$.

Note that the morphism $\varphi^-$ is a chain homotopy equivalence in
$\calc[t,t^{-1}]$, as both $C^-$ and $D^-$ are contractible in that
category by assumption. We conclude from~\cite[Lemma~3.1~(ix)]{Lueck-Steimle(2016BHS)}
that there is a chain isomorphism of the shape
\[
\begin{pmatrix}
     \varphi^- & y \\ x & z
\end{pmatrix}
\colon
C^- \oplus E \xrightarrow{\cong} D^- \oplus E',
\]
where $E$ and $E'$ are elementary chain complexes in
$\calc[t,t^{-1}]$, or even in $\calc$, since both categories have the
same objects.

For appropriate $\underline{k} \in \prod_{m \in \IN} \IZ$, the commutative diagram
\[
\xymatrix@!C=4em{j_+C^+ \ar[rrr]^f
\ar[d]_{\footnotesize\begin{pmatrix}j_+\varphi^+\\ t^{\underline{k}}\circ  x\circ f
\end{pmatrix}}
&&& C^-
\ar[d]^{\footnotesize\begin{pmatrix}1 \\ 0 \end{pmatrix}}
\\
j_+D^+ \oplus i_0 E'
\ar[rrr]^{{\footnotesize\begin{pmatrix}g^{-1}\circ \varphi^- & g^{-1}\circ
y\\ t^{\underline{k}}\circ x & t^{\underline{k}}\circ z \end{pmatrix}}^{-1} }
\ar[d]_{\footnotesize\begin{pmatrix}1 & 0 \end{pmatrix}}
&&& C^-\oplus i_0 E
\ar[d]^{\footnotesize\begin{pmatrix}\varphi^- & y \end{pmatrix}}
\\
j_+D^+ \ar[rrr]^g
&&& D^-
}
\]
provides the desired  factorization of~\eqref{diagram_defining_a_morphism_in_caly}.
\end{proof}

\begin{proof}[Proof of Theorem~\ref{the:homotopy_pull_back_of_calx}]
  Theorem~\ref{the:homotopy_pull_back_of_calx} follows from 
Lemma~\ref{lem:caly_and_cala}, Lemma~\ref{the:relative_fibration_sequences}, and Lemma~\ref{k'_is_homotopy_equivalence}.
\end{proof}

%%%%%%%%%%%%%%%%%%%%%%%%%%%%%%%%%%%%%%%%%%%%%%%%%%%%%%%%%%%%%%%%%%%%

\subsection{The global section functor}\label{subsec:The_global_section_functor}

Recall that we have defined truncation functors in Subsection~\ref{subsec:truncation_functors}.

\begin{definition}[Global section functor]\label{def:global_section_functor}
  The \emph{global section functor} 
  \[
   \Gamma \colon \Chcat(\calx)\to\Chcat(\calc^\kappa)
   \] 
   sends an object $(C^+, f, C^-)$ to the $\calc^\kappa$-chain complex
  \[
   \Sigma^{-1}\cone\bigl(f\trun\colon C^+[0,\infty]\to C^-[1,\infty]\bigr).
   \] 
   A morphism
  $(\varphi^+,\varphi^-) \colon (C^+, f, C^-)  \to (D^+, g, D^-)$ of $\Chcat(\calx)$ 
  is sent to the  morphism in $\Chcat(\cala^\kappa)$ obtained from
  the commutative diagram (using the trivial homotopy)
  \[
    \xymatrix@!C= 11em{C^+[0,\infty]  \ar[r]^-{f\trun}
      \ar[d]_{\varphi^+\trun}
   &
   C^-[-1,\infty] \ar[d]^{\varphi^-\trun}
   \\
   D^+[0,\infty] \ar[r]_-{g\trun} 
   &
   D^-[-1,\infty].
   }
   \]

 \end{definition}

 Note that $f\trun$ is indeed a chain map, as it is the composite of the three chain maps
 $C^+[0,\infty]  \xrightarrow{\id\trun} C^+[-\infty,\infty]  \xrightarrow{i^0f = f\trun} C^-[-\infty,\infty]  \xrightarrow{\id\trun}  C^-[-1,\infty]$
 Let $\Chcathf(\calc)\subset \Chcat(\Idem(\calc^\kappa))$  be the full
 subcategory of homotopy finite chain complexes over $\Idem(\calc^{\kappa})$ , i.e.,
 chain complexes over $\Idem(\calc^{\kappa})$,
 which are homotopy equivalent to a bounded chain complex over $\Idem(\calc)$.

 It follows
 from~\cite[Lemma~3.5]{Lueck-Steimle(2016BHS)} that this category is closed under pushouts
 along a cofibration, so it is a Waldhausen subcategory of $\Chcat(\Idem(\calc^\kappa))$. The
 Approximation Theorem~\ref{the:Cisinki_Approximation_Theorem} of Cisinski shows that the inclusion
 $I(\calc) \colon \Chcat(\Idem(\calc)) \to\Chcathf(\calc)$
 induces an equivalence
\begin{equation}
  \bfKinftyW(I(\calc)) \colon  \bfKinftyW(\Chcat(\Idem(\calc)))\xrightarrow{\sim} \bfKinftyW(\Chcathf(\calc))
  \label{bfKinftyW(Ch(calc))_simeq_bfKinftyW(Chhf(calc))}
\end{equation}
on non-connective $K$-theory.

\begin{lemma}\label{lem:Gamma_values_in_CHhf}

\begin{enumerate}
   \item\label{lem:Gamma_values_in_CHhf:exact}
The functor $\Gamma$ is Waldhausen exact (for the canonical Waldhausen structures);

\item\label{lem:Gamma_values_in_CHhf:hf} For any object $(C^+, f, C^-)$ of
  $\Chcat(\calx)$, the chain complex given by the projective line
  $\Gamma(C^+,f,C^-)\in \Chcat(\calc^\kappa) \subseteq \Chcat(\Idem(\calc^\kappa))$ is
  chain homotopy equivalent to an object in $\Chcat(\Idem(\calc))$. 
  
  Thus, $\Gamma$ defines a Waldhausen
  exact functor
\[
\Gamma(\calc) \colon \Chcat(\calx)\to \Chcathf(\calc).
\]
\end{enumerate}
\end{lemma}
\begin{proof}~\ref{lem:Gamma_values_in_CHhf:exact}
  It is not hard to check using~\cite[Section~4.1]{Lueck-Steimle(2016BHS)} that the two functors 
  $k^\pm\colon  \Chcat(\calx)\to\Chcat(\calc[t^{\pm1}])$ are Waldhausen exact.    The restriction functors
  from $\Chcat(\calc[t])$, $\Chcat(\calc[t^{-1}])$ and
  $\Chcat(\calc[t,t^{-1}])$ to $\Chcat(\calc^\kappa)$ are defined on the level of
  additive categories and hence are Waldhausen exact. Taking cones and suspensions is also
  Waldhausen exact.  
  \\[1mm]~\ref{lem:Gamma_values_in_CHhf:hf} 
 Write the morphism $f_n^{-1} \colon C^+_n \to C^-_n$ in $\calc[t,t^{-1}]$ as a sum
 $\sum_{l_n = a_n}^{b_n} f_n[l_n] \cdot s^{l_n}$ for appropriate morphisms
 $f_n[l_n] \colon \cals(\Phi[t,t^{-1}])^{l_n}(C^-_n) \to C^+_n$ in $\cals(\cala_*[t,t^{-1}])$.  Note
 that $f_n[l_n] =(f_n[l_n]_m)_{m \in \IN} $ for appropriate morphisms
 $f_n[l_n]_m \colon \Phi^{l_n}((C^-_n)_m) \to (C^+_n)_m$ in $\cala_0[t,t^{-1}]$.  Now choose for
 $n \in \IZ$, $l_n \in \IZ$, and $m \in \IN$ a natural number $k_n[l]_m$ such that we can
 write
 $f_n[l_n]_m = \sum_{j_{n,l_n,m} = -k_n[l_n]_m}^{k_n[l_n]_m} f_n[l_n]_m[j_{n,l_n,m}] \cdot  t^{j_{n,l_n,m}}$
for appropriate morphisms
 $f_n[l_n]_m[j_{n,l_n,m}] \colon  \Phi^{l_n}((C^-_n)_m) \to (C^+_n)_m$ in $\cala_0$. Since
   $C^+$ and $C^-$ are bounded, there exists a natural number $N$ such that $C^-_n = 0$
   and $C^+_n =0$ holds for $|n| >  N$. Hence we get for every $n \in \IZ$ with $|n| >  N$ and
   every $m \in \IN$ that $(C^-_n)_m = 0$ and $(C^+_n)_m=0$ hold.  Now define for $m \in \IN$
   a natural number $k_i$ by
   \[k_m = 1 + \max\{k_n[l_n]_m \mid -N \le n \le N, a_n \le l_n \le b_n\}.
   \]
   This definition makes sense since  the set $\bigcup_{n = -N}^N \{l_n \in
   \IZ \mid a_n \le l_n \le b_n\}$ is finite. Note that 
$k_m$ depends only on $m$, in contrast to $k_n[l_n]_m$. We conclude that 
   $f_n[l_n]_m = \sum_{j_{n,l_n,m} = -k_m+1}^{k_m-1} f_n[l_n]_m[j_{n,l_n,m}] \cdot
   t^{j_{n,l_n,m}}$ holds for every $m \in \IN$.

     Define $\underline{k} \in \prod_{m \in \IN} \IZ$ by the collection $\{k_m \mid m \in \IN\}$.
     Then ${f^{-1}}\trun \colon C^-[\underline{-\infty},\underline{\infty}] \to C^+[\underline{k},\underline{\infty}]$ 
     factorizes as 
\[
\xymatrix{C^-[\underline{-\infty},\underline{\infty}] \ar[r]^{f^{-1}\trun} \ar[d]^{\id\trun} & C^+[\underline{k},\underline{\infty}]\\
 C^-[\underline{1},\underline{\infty}] \ar[ru]_{\overline{f^{-1}}}
}
\] 
 and the composite
 \[
   C^+[\underline{k},\underline{\infty}] \xrightarrow{f\trun} C^-[\underline{1},\underline{\infty}]
   \xrightarrow{\overline{f^{-1}}} C^+[\underline{k},\infty]
 \]
 is the identity map.
 
  Hence over  $\Idem(\calc^{\kappa})$, the chain complex $C^-[1,\infty]$ splits as
 \[
 C^-[1,\infty] \cong C^+[\underline{k},\underline{\infty}] \oplus R,
 \]
 where $R_n$ is given by $(C^-_n[\underline{1},\underline{\infty}], \id  - p_n)$ for the projection
 \[
 p_n \colon  C^-[\underline{1},\underline{\infty}]
   \xrightarrow{\overline{f^{-1}}} C^+[\underline{k},\infty] \xrightarrow{f\trun}    C^-[\underline{1},\underline{\infty}].
 \]
 Since the composite
 \[
   C^-[2\cdot \underline{k},\underline{\infty}] \xrightarrow{f^{-1}\trun}
   C^+[\underline{k},\underline{\infty}] \xrightarrow{f\trun} C^-[2\cdot \underline{k},\underline{\infty}]
 \]
 is the identity, $\id -p_n$ restricted to $C^-[2\cdot \underline{k},\underline{\infty}]$
 is trivial. Hence we can find a projection
 $q_n\colon C^-_n[\underline{1},2 \cdot \underline{k} -\underline{1}] \to
 C^-_n[\underline{1},2 \cdot \underline{k} -\underline{1}]$ such that $R_n$ and
 $(C^-_n[\underline{1},2 \cdot \underline{k} -\underline{1}],q_n)$ are isomorphic.  Since
 $C^-_n[\underline{1},2 \cdot \underline{k} -\underline{1}]$ belongs to $\calc$,
 $R_n$ is isomorphic to an object in $\Idem(\calc)$
 for every $n \in \IZ$. This implies that the bounded $\Idem(\calc^{\kappa})$-chain complex $R$ is
 isomorphic to a bounded $\Idem(\calc)$-chain complex.
 We obtain an exact sequence of $\Idem(\calc^{\kappa})$-chain complexes
 \[
  \xymatrix{0 \ar[r] 
 & 
 C^+[\underline{k},\infty] \ar[r] \ar[d]^{\id}
 &  C^+[0,\infty] \ar[d]^{-f\trun} \ar[r] 
 & 
 C^+[0,\underline{k}-\underline{1}]\ar[d]^{g} \ar[r] 
 & 0
 \\
 0 \ar[r] 
 & C^+[\underline{k},\infty] \ar[r]^{-f\trun} 
 & C^-[1,\infty] \ar[r]^r
 & R  \ar[r] 
 & 0
 }\]
where $r_n \colon C^-_n[1,\infty]  \to R_n$ is  the canonical projection  given by $\id - p_n$
and $g$ is the induced map on the quotients. We conclude  that $\Sigma^{-1}\cone(-f\trun) \simeq
 \Sigma^{-1}\cone(g)$, which is isomorphic to a chain complex in $\Idem(\calc)$. Hence $\Gamma(C^+,f,C^-)$ belongs to
 $\Chcathf(\calc)$.
 \end{proof}

 %%%%%%%%%%%%%%%%%%%%%%%%%%%%%%%%%%%%%%%%%%%%%%%%%%%%%%%%%%%%%%%%%%%%

\subsection{Embedding lower-K-theory in higher $K$-theory}\label{subsec:Embedding_lower-K-theory_in_higher_K-theory}

For $k \in \IZ$ and a spectrum $\bfE$, let
$\Sigma^k \bfE$ be the spectrum obtained by shifting, i.e.,
$(\Sigma^k \bfE)_n = \bfE_{n-k}$. Let
\begin{equation}
  \bfd(\bfE) \colon \Sigma^{-1} S^1_+ \wedge \bfE \to \bfE
\label{bfd}
\end{equation}
be the weak homotopy equivalence, natural in $\bfE$, which is given in dimension $n$ by the $n$-th
structure map of $\bfE$.  Let $i \colon S^1 \to S^1 \vee S^1$ be the pinching map. For
a spectrum $\bfE$, define a map of spectra
\[
  \nabla \colon \Sigma^{–1} S^1_+ \wedge \bfE
  \to (\Sigma^{–1} S^1_+ \wedge  \bfE) \vee (\Sigma^{–1} S^1_+ \wedge \bfE),
\]
natural in $\bfE$, by
\begin{multline*}
(\Sigma^{–1} S^1_+ \wedge\bfE)_n
= S^1_+ \wedge E_{n+1} \xrightarrow{i_+ \wedge E_{n+1}} (S^1 \vee S^1)_+ \wedge E_{n+1}
\\
\xrightarrow{\cong}
(S^1_+ \wedge E_{n+1})\vee (S^1_+ \wedge E_{n+1})
=
(\Sigma^{-1} S^1_+ \wedge \bfE )_n \vee (\Sigma^{-1} S^1_+ \wedge \bfE )_n.
\end{multline*}

The next theorem essentially says  that 
$S^1_+ \wedge \bfKinftyW(\Chcat(\calc))$ is up to weak homotopy equivalence
a retract of $\bfKinftyW(\Chcat(\calc[t,t^{-1}]))$.

\begin{theorem}\label{the:Splitting_the_Bass-Heller-Swan_map}
  There is a diagram of spectra
  \[
    \xymatrix{S^1_+ \wedge \bfKinftyW(\Chcat(\calc))
      &
     \\
    S^1_+ \wedge (\Sigma^{-1} S^1_+ \wedge \bfKinftyW(\Chcat(\calc)))
    \ar[u]^{\simeq}_{\id_{S^1_+} \wedge \bfd(\bfKinftyW(\Chcat(\calc)))}
    \ar[d]^{\bfs} \ar[dr]_{\simeq}^-{\bfi}
    &
    \\
    \bfH(\calc)
    \ar[d]_{\simeq}^{\bff} \ar[r]^-{\bfr}
      &
      S^1_+ \wedge \bigl(\Sigma^{-1} S^1_+ \wedge \bfKinftyW(\Chcathf(\calc))\bigr),
      \\
      \bfKinftyW(\Chcat(\calc[t,t^{-1}]))
      &
    }
    \]
    such that the triangle commutes up to a preferred homotopy $\bfh$ and the maps marked
    with $\simeq$ are weak homotopy equivalences.  Moreover, everything including the
    preferred homotopy $\bfh$ is natural in $\Phi$.

\end{theorem}
For its proof we need the next Lemma~\ref{lem:global_section_composed_with_L_j},

We define for $j =0,1$ a functor of additive categories
\begin{equation}
  L_j\colon \calc \to \calx
  \label{L_j_colon_calc_to_calx}
\end{equation}
as follows. The functor $L_j(\calc)$ sends an object $C$ in $\calc$ to the object
$(C, t^{\underline{j}}, C)$ in $\calx$, where $\underline{j} \in \prod_{m \in \IN} \IZ$ is
given by the constant function $I \to \IZ$ with value $j$.  A morphism $u \colon C \to C'$
in $\calc$ is sent to the morphism from $(C, t^{\underline{j}}, C)$ to
$(C', t^{\underline{j}}, C')$ in $\calx$ given by the morphism $j_+(u) \colon C \to C$ in
$\calc[t]$ and the morphism $j_-(u) \colon C \to C$ in $\calc[t^{-1}]$.  Let $J(C) \colon \calc \to \Idem(\calc)$
and  $I(\calc) \colon \Chcat(\Idem(\calc)) \to \Chcathf(\calc)$ be the obvious inclusions and
$0 \colon \Chcat(\calc) \to \Chcathf(\calc)$ be the constant functor with value the chain
complex, all of whose chain objects are $0$.

\begin{lemma}\label{lem:global_section_composed_with_L_j}\
  \begin{enumerate}
  \item\label{lem:global_section_composed_with_L_j:j_is_1} There is a natural equivalence
    of functors $\Chcat(\calc)  \to\Chcathf(\calc)$
    \[
      T_0 \colon I(\calc) \circ \Chcat(J(C))  \xrightarrow{\simeq}  \Gamma \circ \Chcat(L_0),
   \] 
   which is natural in $\Phi$;
\item\label{lem:global_section_composed_with_L_j:j_is_0}
  There is a natural equivalence   of functors $\Chcat(\calc)  \to \Chcathf(\calc)$
    \[
      T_1 \colon 0 \xrightarrow{\simeq} \Gamma \circ \Chcat(L_1),
   \]   
which is natural in $\Phi$;
\end{enumerate}
\end{lemma}
\begin{proof}~\ref{lem:global_section_composed_with_L_j:j_is_1} Fix an object $C$ in
  $\Chcat(\calc)$.  Then $\Gamma(L_0(C))$ is by definition
\[
   \Sigma^{-1}\cone\bigl(\id\trun\colon C[0,\infty]\to C[1,\infty]\bigr).
   \] 
 We have the obvious split exact sequence in
  $\Chcat(\calc^{\kappa})$
  \[
    0 \to   C[\underline{0},\underline{0}] \xrightarrow{\id \trun} C[\underline{0},\underline{\infty}]
    \xrightarrow{\id \trun} C[\underline{1},\underline{\infty}]
     \to 0.
  \]
  It induces a chain homotopy equivalence in $\Chcathf(\calc)$
  \[
   T_0(C) \colon C\xrightarrow{\simeq} \Gamma(L_0(C)).
  \]
  \\[1mm]~\ref{lem:global_section_composed_with_L_j:j_is_0}
  Fix an object $C$ in $\Chcat(\calc)$. Then $\Gamma(L_1(C))$ is by definition
\[
   \Sigma^{-1}\cone\bigl(t^{\underline{1}}  \trun\colon C[0,\infty]\to C[1,\infty]\bigr).
   \]  
   The chain map $t^{\underline{1}} \trun\colon C[0,\infty]\to C[1,\infty]$
   is a chain  isomorphism, its inverse is
  $t^{\underline{-1}}\trun \colon C[\underline{1},\underline{\infty}] \to  C[\underline{0},\underline{\infty}]$.
  Hence $\Gamma(L_1(C))$ is contractible and we obtain a chain homotopy equivalence in $\Chcathf(\calc)$
  \[
   T_1(C) \colon 0 \xrightarrow{\simeq} \Gamma(L_0(C)).
  \]
\end{proof}

\begin{proof}[Proof of Theorem~\ref{the:Splitting_the_Bass-Heller-Swan_map}]
  Denote by $\ast$ the trivial spectrum. 
  Consider the following diagram of non-connective
  spectra
\[
\xymatrix@!C=10em{\ast \ar[dd]
  &
  S \wedge   \bfKinftyW(\Chcat(\calc))  \ar[l] \ar[r]  \ar[d]^{\nabla}
  &
  \ast \ar[dd]
  \\
  &
  (S \wedge    \bfKinftyW(\Chcat(\calc)))
  \vee (S\wedge   \bfKinftyW(\Chcat(\calc)))
  \ar[d]^{\bfa}
  \ar[d]
  &
  \\
  S \wedge  \bfKinftyW(\Chcat(\calc[t])) \ar[d]
  &
  S\wedge \bfKinftyW(\Chcat(\calx)) \ar[d]^{S\wedge \bfKinftyW(\Gamma)}
  \ar[l]_-{\bfb^+} \ar[r]^-{\bfb^-}
  &
  S \wedge  \bfKinftyW(\Chcat(\calc[t^{-1}])) \ar[d]
  \\
  \ast
  &
  S\wedge \bfKinftyW(\Chcathf(\calc)) \ar[l] \ar[r]
    &
    \ast}
\]
where we abbreviate $S = \Sigma^{-1} S^1_+$, $u \colon S^1 \to S^1$ sends $z$ to
$z^{-1}$, and the maps $\bfa$ and $\bfb^{\pm}$ are given by
\begin{eqnarray*}
 \bfa
  & := &
(\Sigma^{-1} \id_{S^1_+} \wedge  \bfK(\Chcat(L_0))) \vee (\Sigma^{-1} u \wedge  \bfKinftyW(\Chcat(L_1)));
  \\
  \bfb^{\pm}
  & := &
        \Sigma^{-1} \id_{S^1_+} \wedge \bfKinftyW(\Chcat(k^{\pm})).
\end{eqnarray*}
Note that $u$ is a base point preserving selfmap of $S^1$ of degree $-1$. Fix a
pointed nullhomotopy for the composite
$S^1 \xrightarrow{i} S^1 \vee S^1 \xrightarrow{\id_{S^1} \vee u} S^1$. Since
$k^{\pm} \circ L_0$ and $k^{\pm} \circ L_1$ agree, both are equal to with $i^{\pm}$, the
upper left square and the upper right square commute up to a preferred homotopy,
natural in $\Phi$.  The lower left and the lower right square commute. The homotopy
pushout of the upper row is
$S^1_+ \wedge (\Sigma^{-1} S^1_+ \wedge \bfKinftyW(\Chcat(\calc)))$
and of the lower row is $S^1_+ \wedge (\Sigma^{-1} S^1_+ \wedge \bfKinftyW(\Chcathf\calc))$. Let
$\bfHP(\calc)$ be the homotopy pushout of the middle row. Then we get from the data above maps of
spectra 
\begin{eqnarray*}
\bfs \colon S^1_+ \wedge (\Sigma^{-1} S^1_+ \wedge \bfKinftyW(\Chcat(\calc)))
& \to &
\bfHP(\calc);
\\
\bfr \colon \bfHP(\calc)
  & \to &
 S^1_+ \wedge \bigl(\Sigma^{-1} S^1_+ \wedge \bfKinftyW(\Chcathf(\calc))\bigr).
\end{eqnarray*}         
The canonical inclusion $J(C) \colon \calc \to \Idem(\calc)$ induces a weak homotopy
equivalence
$\bfKinftyW(\Chcat(J(\calc)))\colon \bfKinftyW(\Chcat(\calc)) \xrightarrow{\simeq}
\bfKinftyW(\Chcat(\Idem(\calc)))$.  This follows from Theorem~\ref{the:Cofinality_Theorem},
where the elementary proof that the conditions appearing in
Theorem~\ref{the:Cofinality_Theorem} are satisfied is left to the reader.  Define the weak
homotopy equivalence
\begin{multline*}
\bfi = 
\colon S^1_+ \wedge \bigl(\Sigma^{-1} S^1_+ \wedge \bfKinftyW(\Chcat(\calc))\bigr)
\xrightarrow{S^1_+ \wedge (\Sigma^{-1} S^1_+ \wedge \bfKinftyW(\Chcat(J(C))))}
\\
S^1_+ \wedge \bigl(\Sigma^{-1} S^1_+ \wedge \bfKinftyW(\Chcat(\Idem(\calc)))\bigr) 
\\
\xrightarrow{\id_{S^1_+} \wedge \Sigma^{-1} \id_{S^1_+} \wedge \bfKinftyW(I(\calc))}
  S^1_+ \wedge \bigl(\Sigma^{-1} S^1_+ \wedge \bfKinftyW(\Chcathf(\calc))\bigr),
\end{multline*}
where $\bfKinftyW(I(\calc))~$ is the weak homotopy equivalence
of~\eqref{bfKinftyW(Ch(calc))_simeq_bfKinftyW(Chhf(calc))}.
  
The composite
\[
\bfr \circ \bfs \colon S^1_+ \wedge (\Sigma^{-1} S^1_+ \wedge \bfKinftyW(\Chcat(\calc)))
\to S^1_+ \wedge \bigl(\Sigma^{-1} S^1_+ \wedge \bfKinftyW(\Chcathf(\calc))\bigr)
\]
is homotopic to $\bfi$  by Lemma~\ref{lem:global_section_composed_with_L_j}.
From the commutative diagram
\[
\xymatrix@!C=15em{\Sigma^{-1} S^1_+ \wedge  \bfKinftyW(\Chcat(\calc[t])) \ar[r]^-{\bfd}
  &
  \bfKinftyW(\Chcat(\calc[t]))
  \\
  \Sigma^{-1} S^1_+ \wedge \bfKinftyW(\Chcat(\calx))\ar[u]^-{\Sigma^{-1} \id_{S^1_+}
    \wedge \bfKinftyW(\Chcat(k^+))} \ar[d]_-{\Sigma^{-1} \id_{S^1_+} \wedge \bfKinftyW(\Chcat(k^-))} \ar[r]^-{\bfd}
  &
  \bfKinftyW(\Chcat(\calx)) \ar[u]_-{\bfKinftyW(\Chcat(k^+))} \ar[d]^-{\bfKinftyW(\Chcat(k^-))} 
  \\
  \Sigma^{-1} S^1_+ \wedge  \bfKinftyW(\Chcat(\calc[t^{-1}]) ) \ar[r]^-{\bfd}
   &
  \bfKinftyW(\Chcat(\calc[t^{-1}]))}
\]
and Theorem~\ref{the:homotopy_pull_back_of_calx}, we obtain a weak homotopy equivalence of
spectra, natural in $\Phi$,
\[
\bff\colon \bfHP(\calc) \xrightarrow{\simeq}  \bfKinftyW(\Chcat(\calc[t,t^{-1}])).
\]
We get from the weak homotopy equivalence $\bfd(\bfKinftyW(\Chcat(\calc)))$ of~\eqref{bfd} a weak
homotopy equivalence
\[
\id_{S^1_+} \wedge \bfd(\bfKinftyW(\Chcat(\calc))) \colon S^1_+ \wedge \bigl(\Sigma^{-1} S^1_+ \wedge \bfKinftyW(\Chcat(\calc))\bigr)
\xrightarrow{\simeq} S^1_+ \wedge \bfKinftyW(\Chcat(\calc)).
\]
This finishes the proof of Theorem~\ref{the:Splitting_the_Bass-Heller-Swan_map}.
\end{proof}

Next we repeat the same construction in the easy case, where we do not take the
automorphisms $\Phi$ into account.  So, given $\cala_*$, we define additive categories
\begin{eqnarray*}
    \widehat{\calc}
    & = &
   \cals(\cala_*);
    \\
    \widehat{\calc}[t^{\pm}]
    & = &
   \cals(\cala_*[t^{\pm}]);
    \\
    \widehat{\calc}[t,t^{-1}]
    & = &
    \cals(\cala_*[t,t^{-1}]),
\end{eqnarray*}
induction functors fitting in a commutative dagram

\begin{equation*}
    \xymatrix@!C=8em{\widehat{\calc} \ar[r]^{\widehat{i}_+} \ar[d]_{\widehat{i}_-} \ar[rd]^{\widehat{i}_0}
        &
        \widehat{\calc}[t] \ar[d]^{\widehat{j}_+}
        \\
        \widehat{\calc}[t^{-1}] \ar[r]_{\widehat{j}_-} 
        &
        \widehat{\calc}[t,t^{-1}]}
      \end{equation*}
      additive categories
      \begin{eqnarray*}
    \widehat{\calc}^{\kappa}
    & = &
   \cals(\cala_*^{\kappa});
    \\
    \widehat{\calc}[t^{\pm}]^{\kappa}
    & = &
   \cals(\cala_*[t^{\pm}]^{\kappa});
    \\
    \widehat{\calc}[t,t^{-1}]^{\kappa}
    & = &
    \cals(\cala_*[t,t^{-1}]^{\kappa}),
\end{eqnarray*}
  and restriction functors

  \begin{eqnarray*}
      \widehat{i}^0 \colon \widehat{\calc}[t,t^{-1}]^{\kappa} & \to & \widehat{\calc}^{\kappa};
      \\
      \widehat{i}^{\pm}  \colon \widehat{\calc}[t^{ \pm 1}]^{\kappa} & \to & \widehat{\calc}^{\kappa},
    \end{eqnarray*}
    such that we have adjunctions $((\widehat{i}_0)^{\kappa},\widehat{i}^0)$,
    $((\widehat{i}_+)^{\kappa},\widehat{i}^+)$, and
    $((\widehat{i}_-)^{\kappa},\widehat{i}^-)$. There is an obvious definition of the
    projective line $\widehat{\calx}$ and of the global section functor
    $\widehat{\Gamma} \colon \widehat{\calx}\to \Chcathf(\widehat{\calc}^{\kappa})$.

   Now analogously to the proof of
   Theorem~\ref{the:Splitting_the_Bass-Heller-Swan_map} one can show

 \begin{theorem}\label{the:Splitting_the_Bass-Heller-Swan_map_widehat}
  There is a diagram of spectra
  \[
    \xymatrix{S^1_+ \wedge \bfKinftyW(\Chcat(\widehat{\calc}))
      &
     \\
    S^1_+ \wedge (\Sigma^{-1} S^1_+ \wedge \bfKinftyW(\Chcat(\widehat{\calc})))
    \ar[u]^{\simeq}_{\id_{S^1_+}\wedge \bfd(\bfKinftyW(\widehat{\calc}))}
    \ar[d]^{\widehat{\bfs}} \ar[dr]_{\simeq}^-{\widehat{\bfi}}
    &
    \\
    \bfHP(\widehat{\calc})
    \ar[d]_{\simeq}^{\widehat{\bff}} \ar[r]^-{\widehat{\bfr}}
      &
      S^1_+ \wedge \bigl(\Sigma^{-1} S^1_+ \wedge \bfKinftyW(\Chcathf(\widehat{\calc}))\bigr),
      \\
      \bfKinftyW(\Chcat(\widehat{\calc}[t,t^{-1}]))
      &
    }
    \]
    such that the triangle commutes up to a preferred homotopy $\widehat{\bfh}$ and the
    maps marked with $\simeq$ are weak homotopy equivalence.  Moreover, everything
    including the preferred homotopy $\widehat{\bfh}$ is natural in $\Phi$.
\end{theorem}

There are obvious inclusions $l \colon \widehat{\calc} \to \calc$,
$l[t^{\pm}] \colon \widehat{\calc}[t^{\pm}] \to \calc[t^{\pm}]$, and
$l[t,t^{-1}] \colon \widehat{\calc}[t,t^{-1}] \to \calc[t,t^{-1}]$.
 They induce a map denoted
  by $\bfl$ from the diagram appearing in
  Theorem~\ref{the:Splitting_the_Bass-Heller-Swan_map_widehat} to the one appearing in
  Theorem~\ref{the:Splitting_the_Bass-Heller-Swan_map}.

  The pro-automorphism $\Phi \colon \cala_* \to \cala_*$ induces in the obvious way
  an automorphism of the diagram appearing in
  Theorem~\ref{the:Splitting_the_Bass-Heller-Swan_map_widehat}, denoted by $\bft$. Taking
  the mapping torus in each entry of the diagram appearing in
  Theorem~\ref{the:Splitting_the_Bass-Heller-Swan_map_widehat} yields a diagram of the
  form
  \begin{equation}
    \xymatrix{S^1_+ \wedge T_{\bfKinftyW(\Chcat(\cals(\Phi)))}
      &
     \\
    S^1_+ \wedge (\Sigma^{-1} S^1_+ \wedge T_{\bfKinftyW(\Chcat(\cals(\Phi)))})
    \ar[u]^{\simeq}_{\id_{S^1_+} \wedge \bfd(T_{\bfKinftyW(\cals(\Phi))})}
    \ar[d]^{\widehat{\bfs}} \ar[dr]_{\simeq}^-{T_{\widehat{\bfi}}}
    &
    \\
    T_{\bfHP(\cals(\Phi))}
    \ar[d]_{\simeq}^{T_{\widehat{\bff}}} \ar[r]^-{T_{\widehat{\bfr}}}
      &
      S^1_+ \wedge \bigl(\Sigma^{-1} S^1_+ \wedge T_{\bfKinftyW(\Chcathf(\cals(\Phi)))}\bigr).
      \\
       T_{\bfKinftyW(\Chcat(\cals(\Phi[t,t^{-1}])))}
      &
    }
    \label{diagram_of_mapping_tori}
  \end{equation}
  
  It is not true that $\bfl \circ \bft$ agrees with $\bfl$ again. However, this is true
  up to preferred homotopy, and therefore  we get
  a map from the diagram~\eqref{diagram_of_mapping_tori} to the diagram appearing in 
  Theorem~\ref{the:Splitting_the_Bass-Heller-Swan_map}.
  These homotopies are all induced by obvious natural
  transformations of functors. For instance, for the automorphism
  $\widehat{\Phi} = \cals(\Phi) \colon \widehat{\calc} = \cals(\cala_*)
  \xrightarrow{\cong} \widehat{\calc} = \cals(\cala_*)$
  and the inclusion $l \colon \widehat{\calc} = \cals(\cala_*) \to \calc = \cals(\cala_*)_{\cals(\Phi)}[\IZ]$
  we get a natural transformation $l   \to l \circ \widehat{\Phi} $, if we assign an object $\underline{A}$
  in $\widehat{\calc} = \cals(\cala_*)$ the
  morphism $\underline{A} \to \underline{A}$ in $\calc = \cals(\cala_*)_{\cals(\Phi)}[\IZ]$ given by
  $\id_{\widehat{\Phi}(\underline{A})} \cdot t$.  From these data we get
  a map from the diagram~\eqref{diagram_of_mapping_tori} to the diagram appearing in 
  Theorem~\ref{the:Splitting_the_Bass-Heller-Swan_map}.

  Now we apply the functor $\pi_n$ to the 
  diagram~\eqref{diagram_of_mapping_tori}, the diagram appearing in
  Theorem~\ref{the:Splitting_the_Bass-Heller-Swan_map} and the map between them constructed
  above. Taking into account that some of the arrows appearing in the
  diagram~\eqref{diagram_of_mapping_tori} and the diagram
  Theorem~\ref{the:Splitting_the_Bass-Heller-Swan_map} are weak equivalences, we get for
  every $n \in \IZ$ a commutative diagram
  \begin{equation}
  \xymatrix{\pi_{n-1}\bigl(T_{\bfKinftyW(\Chcat(\cals(\phi)))}\bigr) \ar[r] \ar[d]
    &
    \pi_{n-1}\bigl(\bfKinftyW(\Chcat(\cals(\cala_*)_{\cals(\Phi)} [\IZ]))\bigr) \ar[d]
    \\
    \pi_n\bigl(T_{\bfKinftyW(\Chcat(\cals(\Phi[t,t^{-1}])))}\bigr) \ar[d] \ar[r]
    &
    \pi_{n}\bigl(\bfKinftyW(\Chcat(\cals(\cala_*[t,t^{-1}]))_{\cals(\Phi[t,t^{-1}])} [\IZ])\bigr) \ar[d]
    \\
    \pi_{n-1}\bigl( T_{\bfKinftyW(\Chcat(\cals(\phi)))  }\bigr) \ar[r] 
    &
     \pi_{n-1}\bigl(\bfKinftyW(\Chcat(\cals(\cala_*)_{\cals(\Phi)} [\IZ]))\bigr)
  }
\label{splitting_for_cals_with_dimension_shift_by_one}
\end{equation}
such that the composite of the left two vertical arrows and the composite of the right two
vertical arrows are isomorphisms.  

Next we explain, how we get the corresponding diagram for $\call$ instead of $\cals$
\begin{equation}
  \xymatrix{\pi_{n-1}\bigl(T_{\bfKinftyW(\Chcat(\call(\phi)))}\bigr) \ar[r] \ar[d]
    &
    \pi_{n-1}\bigl(\bfKinftyW(\Chcat(\call(\cala_*)_{\call(\Phi)} [\IZ]))\bigr) \ar[d]
    \\
    \pi_n\bigl(T_{\bfKinftyW(\Chcat(\call(\Phi[t,t^{-1}])))}\bigr) \ar[d] \ar[r]
    &
    \pi_{n}\bigl(\bfKinftyW(\Chcat(\call(\cala_*[t,t^{-1}])_{\call(\Phi[t,t^{-1}])} [\IZ]))\bigr) \ar[d]
    \\
    \pi_{n-1}\bigl( T_{\bfKinftyW(\Chcat(\call(\phi)))  }\bigr) \ar[r] 
    &
     \pi_{n-1}\bigl(\bfKinftyW(\Chcat(\call(\cala_*)_{\call(\Phi)} [\IZ]))\bigr)
  }
  \label{splitting_for_call_with_dimension_shift_by_one}
\end{equation}
such that the composite of the left two vertical arrows and the composite of the right two
vertical arrows are isomorphisms. Note that  the upper horizontal arrow is an
isomorphism, if the middle arrow is an isomorphism and that the upper arrow deals with
the homotopy groups in a dimension, which is the dimension of the middle arrow minus $1$.

The point is that the  construction of map from the diagram~\eqref{diagram_of_mapping_tori} to
the diagram Theorem~\ref{the:Splitting_the_Bass-Heller-Swan_map} including the statement
 that some maps are weak homotopy equivalences, carries over word by word, if we replace
$\calt$ by $\cals$ everywhere. These two constructions are compatible with the various
inclusions and therefore yields also a version of a map from the
diagram~\eqref{diagram_of_mapping_tori} to the diagram
Theorem~\ref{the:Splitting_the_Bass-Heller-Swan_map} including the statement that some 
maps are weak homotopy equivalences, where we replace $\cals$ by $\call$ everywhere
and use Lemma~\ref{lem:Karoubi_filtration_sequence_categories}.
Now diagram~\eqref{splitting_for_call_with_dimension_shift_by_one} is derived from this map
analogously to~\eqref{splitting_for_cals_with_dimension_shift_by_one}.

Recall that $\bfKinftyW(\Chcat(\call(\cala_*)))$ is a spectrum with a $\IZ$-action, which comes
from $\bfKinftyW(\Chcat(\call(\Phi)))$.  We obtain a covariant functor, see for
instance~\cite[Section~9]{Bartels-Lueck(2009coeff)},
\begin{equation}
  \bfKinftyW_{\Chcat(\call(\cala_*))} \colon \OrG{\IZ} \to \Spectra.
  \label{bfKinftyW_(Ch(call(cala_ast)))}
\end{equation}
It determines a $\IZ$-homology theory $H_n^{\IZ}(-,\bfKinftyW_{\Chcat(\call(\cala_*))} )$
with the property that for every subgroup $H \subseteq \IZ$ and $n \in \IZ$ we have the 
natural isomorphism
\[
  H_n^{\IZ}(\IZ/H,\bfKinftyW_{\Chcat(\call(\cala_*))}) \xrightarrow{\cong}
  K_n(\Chcat(\call(\cala_*)) \rtimes_{\Chcat(\call(\Phi))|_{H}} H)
\]
as explained for instance in~\cite[Section~9]{Bartels-Lueck(2009coeff)}. Analogously we
get $\IZ$-homology theories $H_n^{\IZ}(-,\bfKinftyW_{\Chcat(\call(\cala_*[t,t^{-1}]))} )$,
$H_n^{\IZ}(-,\bfKinfty_{\call(\cala_*)} )$, and
$H_n^{\IZ}(-,\bfKinfty_{\call(\cala[t,t^{-1}]_*)} )$ with the property that for every
subgroup $H \subseteq \IZ$ and $n \in \IZ$ we have the natural isomorphisms
\begin{eqnarray*}
H_n^{\IZ^r}(\IZ/H,\bfKinftyW_{\Chcat(\call(\cala_*[t,t^{-1}]))})
&\xrightarrow{\cong} &
K_n(\Chcat(\call(\cala_*[t,t^{-1}])) \rtimes_{\Chcat(\call(\Phi[t,t^{-1}]))|_H} H);
\\
H_n^{\IZ}(\IZ/H,\bfKinfty_{\call(\cala_*)} )
 &\xrightarrow{\cong} &
K_n(\call(\cala_*) \rtimes_{\call(\Phi)|_{H}} H);
\\
H_n^{\IZ}(\IZ/H,\bfKinfty_{\call(\cala_*[t,t^{-1}])} )
 &\xrightarrow{\cong} &
K_n(\call(\cala_*[t,t^{-1}]) \rtimes_{\call(\Phi[t,t^{-1}])|_{H}} H).
\end{eqnarray*}

\begin{lemma}\label{lem:isomorphisms_of_Z-homology_theories}
The are natural equivalences of $\IZ$-homology theories
\begin{eqnarray*}
H_n^{\IZ}(-,\bfKinftyW_{\Chcat(\call(\cala_*))} )
& \xrightarrow{\cong} &  
H_n^{\IZ}(-,\bfKinfty_{\call(\cala_*)} );
\\
H_n^{\IZ}(-,\bfKinftyW_{\Chcat(\call(\cala_*[t,t^{-1}]))} )
& \xrightarrow{\cong} &   
                        H_n^{\IZ}(-,\bfKinfty_{\call(\cala_*[t,t^{-1}])} ).
\end{eqnarray*}
\end{lemma}
\begin{proof}
  They are induced by the zigzag of natural weak homotopy equivalences appearing
  in Theorem~\ref{the:Gillet-Waldhausen_zigzag_for_non-connective_K-theory}
  using~\cite[Lemma~4.6]{Davis-Lueck(1998)}
\end{proof}

\begin{lemma}\label{lem:splitting_in_terms_of_homology_groups}
For every $n \in \IZ$ there exists a commutative diagram
\[
\xymatrix{H_{n-1}^{\IZ}(E\IZ,\bfKinfty_{\call(\cala_*)} )
\ar[r] \ar[d]
&
H_{n-1}^{\IZ}(\pt,\bfKinfty_{\call(\cala_*)}) \ar[d]
\\
H_{n}^{\IZ}(E\IZ,\bfKinfty_{\call(\cala_*[t,t^{-1}])} ) \ar[d]
\ar[r]
&
H_{n}^{\IZ}(\pt,\bfKinfty_{\call(\cala_*[t,t^{-1}])}) \ar[d]
\\
H_{n-1}^{\IZ}(E\IZ,\bfKinfty_{\call(\cala_*)} )
\ar[r]
&
H_{n-1}^{\IZ}(\pt,\bfKinfty_{\call(\cala_*)})}
\]
such that the composite of the two vertical arrows appearing in the left
and in the right column are isomorphisms and the horizontal  arrows are induced by the projection  $E\IZ \to \pt$.

In particular the  upper horizontal arrow is bijective, if the middle horizontal arrow is bijective.
\end{lemma}
\begin{proof}
  Because of Lemma~\ref{lem:isomorphisms_of_Z-homology_theories} it suffices to proof
  Lemma~\ref{lem:splitting_in_terms_of_homology_groups} in the case, where we replace
  $\bfKinfty_{\call(\cala_*)}$ by $\bfKinftyW_{\Chcat(\call(\cala_*))}$and
  $\bfKinfty_{\call(\cala_*[t,t^{-1}])}$ by $\bfKinfty_{\Chcat(\call(\cala_*[t,t^{-1}]))}$.
  
  Now one easily checks unravelling the definitions that
  the rows of the diagram above after this replacement
  agree with the rows of the diagram~\eqref{splitting_for_call_with_dimension_shift_by_one}.
  Now the claim follows from the diagram~\eqref{splitting_for_call_with_dimension_shift_by_one}.
\end{proof}

%%%%%%%%%%%%%%%%%%%%%%%%%%%%%%%%%%%%%%%%%%%%%%%%%%%%%%%%%%%%%%%%%%%%

\subsection{Reduction to the connective case}\label{subsec:Reduction_to_the_connective_case}

\begin{theorem}[Reduction to the connective case]\label{the:reduction_to_the_connective_case}
 Fix a natural number $n_0$.  Suppose that for every natural number $d$ and
  every natural number $n$ satisfying  $n \ge n_0$ the map
  \[H_{n}^{\IZ}(E\IZ,\bfKinfty_{\call(\cala_*[\IZ^d])})  \to
    H_{n}^{\IZ}(\pt,\bfKinfty_{\call(\cala_*[\IZ^d])}) = K_n(\call(\cala_*[\IZ^d]) \times_{\call(\Phi[\IZ^d])} \IZ)
  \]
  is an isomorphism.

  Then  for every $n \in \IZ$ the map
  \[H_{n}^{\IZ}(E\IZ,\bfKinfty_{\call(\cala_*)})  \to
    H_{n}^{\IZ}(\pt,\bfKinfty_{\call(\cala_*)}) = K_n(\call(\cala_*) \times_{\call(\Phi)} \IZ)
  \]
  is an isomorphism.
\end{theorem}

\begin{proof}
 One can iterate the passage from $\cala_*$ to $\cala_*[t,t^{-1}] = \cala_*[\IZ]$ and obtain a passage from
$\cala_*$  to $\cala_*[\IZ^d] = (\cala_*[\IZ^{d-1}])[t,t^{-1}]$ for every natural  number $d$.
Now Theorem~\ref{the:reduction_to_the_connective_case} follows from 
Lemma~\ref{lem:splitting_in_terms_of_homology_groups}.
\end{proof}

%%%%%%%%%%%%%%%%%%%%%%%%%%%%%%%%%%%%%%%%%%%%%%%%%%%%%%%%%%%%%%%%%%%%

\subsection{Proof of Theorem~\ref{the:assembly_for_cald_upper_0(cals_ast)_with_Zr-action}}%
\label{subsec:Proof_of_Theorem_ref(the:assembly_for_cald_upper_0(cals_ast)_with_Zr-action)}

Now we are ready to finalize the proof of Theorem~\ref{the:assembly_for_cald_upper_0(cals_ast)_with_Zr-action}.

Thanks to Lemma~\ref{lem:reduction-to_the_case_r_is_1}, we can assume without loss of generality that $r = 1$.
Because of Theorem~\ref{the:reduction_to_the_connective_case} it suffices to show
that for every natural number $d$ and every natural number $n$ satisfying  $n \ge 2 $ the map
  \[H_{n}^{\IZ}(E\IZ,\bfKinfty_{\call(\cala_*[\IZ^d])})  \to
    H_{n}^{\IZ}(\pt,\bfKinfty_{\call(\cala_*[\IZ^d])}) = K_n(\call(\cala_*[\IZ^d]) \times_{\call(\Phi[\IZ^d])} \IZ)
  \]
  is an isomorphism.

Let $\bfK_{\Idem(\call(\cala_*[\IZ^d]))}$ be the connective
version of $\bfKinfty_{\Idem(\call(\cala_*[\IZ^d]))}$. Since we have $\dim(E\IZ) \le 1$, we conclude from
Lemma~\ref{K_and_cofinal} that  the vertical arrows appearing in the commutative diagram

\[
  \xymatrix{H_n^{\IZ}(E\IZ;\bfK_{\Idem(\call(\cala_*[\IZ^d]))}) \ar[d]_{\cong} \ar[r]
    &
    H_n^{\IZ}(\pt;\bfK_{\Idem(\call(\cala_*[\IZ^d]))}) \ar[d]^{\cong}
    \\
    H_n^{\IZ}(E\IZ;\bfKinfty_{\Idem(\call(\cala_*[\IZ^d]))}) \ar[r]
    &
    H_n^{\IZ}(\pt;\bfKinfty_{\Idem(\call(\cala_*[\IZ^d]))})
    \\
    H_n^{\IZ}(E\IZ;\bfKinfty_{\call(\cala_*[\IZ^d])}) \ar[r] \ar[u]^{\cong}
    &
    H_n^{\IZ}(\pt;\bfKinfty_{\call(\cala_*[\IZ^d])}) \ar[u]_{\cong}
  }
\]
are bijective for $n \ge 2$. Hence it suffices show that for every natural number $d$ the map
\begin{equation}
H_n^{\IZ}(E\IZ;\bfK_{\Idem(\call(\cala_*[\IZ^d]))}) \to H_n^{\IZ}(\pt;\bfK_{\Idem(\call(\cala_*[\IZ^d]))}) 
\label{gleich_sind_wir_fertig}
\end{equation}
is bijective for $n\in \IZ, n \ge 2$.

By assumption the category $\cala_m[\IZ^d]$ is uniformly $l(d)$-regular coherent and the
inclusion $\cala_m[\IZ^d] \to \cala_{m+1}[\IZ^d]$ is flat for every $m \ge 0$. We conclude
from Lemma~\ref{lem:inverse_systems_and_regularity} that $\call(\cala_*[\IZ^d])$ is
uniformly $l(d)$-regular coherent. We conclude from
Lemma~\ref{lem:inclusion_of_full_cofinal_subcategory}~%
\ref{lem:inclusion_of_full_cofinal_subcategory:regular_coherent}
that $\Idem(\call(\cala_*[\IZ^d]))$ is uniformly $l(d)$-regular coherent. Hence
$\Idem(\call(\cala_*[\IZ^d]))$ is idempotent complete and regular coherent.
Now the bijectivity of~\eqref{gleich_sind_wir_fertig} for
$n\in \IZ, n \ge 2$ follows from
Theorem~\ref{the:The_connective_K-theory_of_additive_categories}  applied to $\Idem(\call(\cala_*[\IZ^d]))$,
since for the map $\bfa$ appearing there the homomorphism  $\pi_n(\bfa)$ can be identified with the
map~\eqref{gleich_sind_wir_fertig} for $n \ge 2$. This finishes the proof of
Theorem~\ref{the:assembly_for_cald_upper_0(cals_ast)_with_Zr-action}.

%%%%%%%%%%%%%%%%%%%%%%%%%%%%%%%%%%%%%%%%%%%%%%%%%%%%%%%%%%%%%%%%%%%% 
%%%%%%%%%%%%%%%%%%%%%%%%%%%% References  %%%%%%%%%%%%%%%%%%%%%%%%%%%%%%%%
%%%%%%%%%%%%%%%%%%%%%%%%%%%%%%%%%%%%%%%%%%%%%%%%%%%%%%%%%%%%%%%%%%%%

\typeout{----------------------------- References ------------------------------}

% \bibliographystyle{abbrv}
% \bibliography{dbpub,dbpre}

%\version{19.08.2021 (Wolfgang)}

\end{document}